\documentclass[english,letter paper,12pt,leqno]{article}
\usepackage{array}
\usepackage{stmaryrd}
\usepackage{amsmath, amscd, amssymb, mathrsfs, accents, amsfonts,amsthm}
\usepackage[all]{xy}
\usepackage{dsfont}
\usepackage{relsize}
\usepackage{tikz}

\definecolor{Myblue}{rgb}{0,0,0.6}
\usepackage[a4paper,colorlinks,citecolor=Myblue,linkcolor=Myblue,urlcolor=Myblue,pdfpagemode=None]{hyperref}

\usetikzlibrary{snakes}

\SelectTips{cm}{}

\newtheorem{theorem}{Theorem}[section]

\newtheorem{lemma}[theorem]{Lemma}
\newtheorem{corollary}[theorem]{Corollary}
\newtheorem{setup}[theorem]{Setup}

\newcommand{\tagarray}{\mbox{}\refstepcounter{equation}$(\theequation)$}

\newtheoremstyle{example}{\topsep}{\topsep}
	{}
	{}
	{\bfseries}
	{.}
	{2pt}
	{\thmname{#1}\thmnumber{ #2}\thmnote{ #3}}
	
	\theoremstyle{example}
	\newtheorem{definition}[theorem]{Definition}
	\newtheorem{example}[theorem]{Example}
	\newtheorem{remark}[theorem]{Remark}

\numberwithin{equation}{section}

\def\AA{\mathcal{A}}
\def\BB{\mathcal{B}}
\def\HH{\HH}
\def\HH{\mathcal{H}}
\def\KK{\mathcal{K}}
\def\ob{\operatorname{ob}}
\def\eval{\operatorname{ev}}
\def\res{\operatorname{Res}}

\def\Ker{\operatorname{Ker}}

\def\stab{\operatorname{stab}}
\def\Hom{\operatorname{Hom}}

\def\be{\begin{equation}}
\def\ee{\end{equation}}

\def\nZ{\mathds{Z}}

\DeclareMathOperator{\End}{End}

\DeclareMathOperator{\hmf}{hmf}

\DeclareMathOperator{\mfdg}{mf}
\DeclareMathOperator{\Spec}{Spec}
\DeclareMathOperator{\At}{At}
\DeclareMathOperator{\HF}{HF}
\DeclareMathOperator{\sigmastar}{\sigma_t}
\DeclareMathOperator{\vAt}{At}

\begin{document}

\def\Res{\res\!}
\newcommand{\ud}{\mathrm{d}}
\newcommand{\Ress}[1]{\res_{#1}\!}
\newcommand{\cat}[1]{\mathcal{#1}}
\newcommand{\lto}{\longrightarrow}
\newcommand{\xlto}[1]{\stackrel{#1}\lto}
\newcommand{\md}[1]{\mathscr{#1}}
\newcommand{\coeff}[1]{\widetilde{#1}}
\def\sus{\l}
\def\l{\,|\,}
\def\sgn{\textup{sgn}}
\def\samp{\zeta}
\def\Samp{Z}
\def\traff{N}

\title{Constructing $A_\infty$-categories of matrix factorisations}
\author{Daniel Murfet}

\maketitle

\begin{abstract}
We study constructive $A_\infty$-models of the DG-category of matrix factorisations of a potential over a commutative $\mathbb{Q}$-algebra $k$, consisting of a Hom-finite $A_\infty$-category equipped with an $A_\infty$-idempotent functor.
\end{abstract}

\tableofcontents

\section{Introduction}

This paper continues the project from \cite{pushforward,cut,lgdual} of making the theory of affine B-twisted topological Landau-Ginzburg models \emph{constructive}, in the sense of deriving formulas and algorithms which compute the fundamental categorical structures from coefficients of the potential and the differentials of matrix factorisations. For example, the units and counits of adjunction in the bicategory of Landau-Ginzburg models \cite{lgdual} and the pushforward and convolution operations \cite{pushforward} may be described in terms of various flavours of Atiyah classes.

We study \emph{idempotent finite $A_\infty$-models} of DG-categories of matrix factorisations over arbitrary $\mathbb{Q}$-algebras $k$. An idempotent finite model of a DG-category $\AA$ over $k$ is a pair $(\mathcal{B}, E)$ consisting of an $A_\infty$-category $\mathcal{B}$ over $k$ which is Hom-finite, in the sense that for each pair $X,Y$ of objects the complex $\mathcal{B}(X,Y)$ is a finitely generated projective $k$-module, together with a diagram of $A_\infty$-functors
\[
\xymatrix@C+3pc{
\AA \ar@<1ex>[r]^I & \mathcal{B} \ar@(ur,dr)^E \ar@<1ex>[l]^P
}
\]
with $A_\infty$-homotopies $P \circ I \simeq 1$ and $I \circ P \simeq E$. Using the finite pushforward construction of \cite{pushforward} and Clifford actions of \cite{cut} we explain how to define an idempotent finite $A_\infty$-model of the DG-category of matrix factorisations $\AA$ of any potential $W$ over $k$. When $k$ is a field this is constructive in the sense that we give algorithms for computing the entries in the higher $A_\infty$-operations on $\BB$ and the components of $E$, viewed as matrices, from the Atiyah class of $\AA$ and a Gr\"obner basis of the defining ideal of the critical locus of $W$.

We note that our goal is \emph{not} to construct a minimal model: $\BB$ has nonzero differential. When $k$ is a field the usual approach to finding a Hom-finite $A_\infty$-model of $\AA$ is to equip $H^*(\AA)$ with an $A_\infty$-structure, so that there is a diagram of $A_\infty$-functors
\[
\xymatrix@C+3pc{
\AA \ar@<1ex>[r]^I & H^*(\AA) \ar@<1ex>[l]^P
}
\]
and $A_\infty$-homotopies $P \circ I \simeq 1$ and $I \circ P \simeq 1$. The problem is that the minimal model is only under good control for those matrix factorisations $X,Y$ where we happen to know a good cohomological splitting on $\AA(X,Y)$. In our approach the idempotent finite model $(\BB,E)$ is constructed directly from $\AA$, and if we happen to know a cohomological splitting then this provides the data necessary to split $E$. An important example is the case where $X = Y = k^{\stab}$ is the standard generator, where there is a natural cohomological splitting and the resulting $A_\infty$-structure has been studied by Seidel \cite[\S 11]{seidel_hms}, Dyckerhoff \cite[\S 5.6]{d0904.4713}, Efimov \cite[\S 7]{efimov} and Sheridan \cite{sheridan}. We explain how to split our idempotent $E$ in this special case and recover the minimal $A_\infty$-model of $\AA(k^{\stab},k^{\stab})$ in Section \ref{section:generator}. 
\\

To explain in more detail, let $\AA = \mfdg(R,W)$ be the DG-category of finite-rank matrix factorizations of a potential $W \in R = k[x_1,\ldots,x_n]$ over an arbitrary $\mathbb{Q}$-algebra $k$. This is a DG-category over $R$, which we can view as a $\nZ_2$-graded $R$-module
\[
\HH_{\AA} = \bigoplus_{X,Y} \AA(X,Y)
\]
equipped with some $R$-linear structure, namely, the differential $\mu_1$ and composition operator $\mu_2$. Let $\varphi: k \lto R$ denote the inclusion of constants. We can ask if the restriction of scalars $\varphi_*( \HH_{\AA} )$ is $A_\infty$-homotopy equivalent over $k$ to an $A_\infty$-category which is Hom-finite over $k$, and can such a finite model be described constructively? This is related to the problem of pushforwards considered in \cite{pushforward} and, here as there, it is unclear that in general a \emph{direct} construction of the finite pushforward exists.

Instead, following \cite{pushforward} the conceptual approach we adopt here is to seek an algorithm which constructs both a larger object $\mathcal{B}$ which is a Hom-finite $A_\infty$-category over $k$ together with an idempotent $E: \mathcal{B} \lto \mathcal{B}$ which splits to $\AA$. More precisely, the larger object is obtained by adjoining to $\AA$ a number of odd supercommuting variables $\theta_1,\ldots,\theta_n$ and completing along the the critical locus of $W$ to obtain the DG-category
\[
\AA_\theta = \bigwedge F_\theta \otimes_k \AA \otimes_R \widehat{R}\,,
\]
where $F_\theta$ is the $\nZ_2$-graded $k$-module $\bigoplus_{i=1}^n k \theta_i$ which is concentrated in odd degree, and $\widehat{R}$ denotes the $I$-adic completion where $I = (\partial_{x_1} W, \ldots, \partial_{x_n} W)$. The differential on $\AA_\theta$ is the one inherited from $\AA$. Note that the completion $\AA \otimes_R \widehat{R}$ can be recovered from $\AA_\theta$ by splitting the idempotent DG-functor on $\AA_\theta$ which arises from the morphism of algebras
\be\label{eq:projector_e}
\bigwedge F_\theta \lto k \cdot 1 \lto \bigwedge F_\theta
\ee
which projects onto the identity by sending all $\theta$-forms of nonzero degree to zero.

One should think of $F_\theta$ as the normal bundle to the critical locus of $W$, and $\AA_\theta$ can be equipped with a natural strong deformation retract which arises from a connection $\nabla$ ``differentiating'' in the normal directions to the critical locus. This strong deformation retract can be interpreted as an analogue in algebraic geometry of the deformation retract associated to the Euler vector field of a tubular neighborhood; see Appendix \ref{section:formaltub}. Applying the homological perturbation lemma results in a Hom-finite $A_\infty$-category $\BB$ with
\[
\BB(X,Y) = R/I \otimes_R \Hom_R(X,Y)\,.
\]
By construction there is an idempotent finite $A_\infty$-model
\[
\xymatrix@C+3pc{
\AA \otimes_R \widehat{R}\, \ar@<1ex>[r]^-I & \mathcal{B} \ar@(ur,dr)^E \ar@<1ex>[l]^-P
}
\]
of the completion $\AA \otimes_R \widehat{R}$, which we show is homotopy equivalent to $\AA$ over $k$ (in particular $\AA$ and $\AA \otimes_R \widehat{R}$ are both DG-enhancements of the $k$-linear triangulated category $\hmf(R,W)$, with the latter being analogous to working with matrix factorisations over the power series ring $k\llbracket \bold{x} \rrbracket$). The $A_\infty$-idempotent $E$ arises in the obvious way as the transfer to $\BB$ of the idempotent \eqref{eq:projector_e}.

\vspace{0.3cm}

\textbf{Outline of the paper.} In Section \ref{section:the_model} we give the details of the above sketch of the construction of the idempotent finite model of $\operatorname{mf}(R,W)$, although the proofs are collected in Appendix \ref{section:proofs}. The algorithmic content of the theory is summarised in Section \ref{section:the_algorithms}, but developed over the course of Section \ref{section:towards} and Section \ref{section:feynman_diagram} culminating in the Feynman rules of Section \ref{section:feynman_diagram_4}. Some of the geometric intuition for the central strong deformation retract is developed in Section \ref{section:formaltub}. In Appendix \ref{section:noetherian} we give some technical observations necessary to remove a Noetherian hypothesis from \cite{cut}.

\vspace{0.3cm}

\textbf{Related work.} The approach we develop here seems to be related to ideas developed in the complex analytic setting by Shklyarov \cite{shklyarov_cy} in order to put a Calabi-Yau structure on $\AA$, although we do not understand the precise connection. For applications to string field theory it is important to construct \emph{cyclic} $A_\infty$-minimal models; see for example \cite{carqueville}. For a recent approach to this problem for the endomorphism DG-algebra of $k^{\stab}$ see \cite{tu}. We do not understand the interplay between cyclic $A_\infty$-structures and idempotent finite $A_\infty$-models. This project began as an attempt to understand the work on deformations of matrix factorisations and effective superpotentials in the mathematical physics literature \cite{baumgartl1, baumgartl2, baumgartl3, baumgartl4, carqueville2, carqueville3, knapp} which should be better known to mathematicians. Some ideas developed here were inspired by old work of Herbst-Lazaroiu \cite{herbst} that has now culminated in a new approach to non-affine Landau-Ginzburg models \cite{babalic}.

\vspace{0.3cm}

\textbf{Acknowledgements.} Thanks to Nils Carqueville for introducing me to $A_\infty$-categories, Calin Lazaroiu for encouragement and the opportunity to present the results at the workshop ``String Field Theory of Landau-Ginzburg models'' at the IBS Center for Geometry and Physics in Pohang. The author was supported by the ARC grant DP180103891.

\section{Background}

Throughout $k$ is a commutative $\mathbb{Q}$-algebra and unless specified otherwise $\otimes$ means $\otimes_k$. If $\bold{x} = (x_1,\ldots,x_n)$ is a sequence of formal variables then $k[\bold{x}]$ denotes $k[x_1,\ldots,x_n]$ and similarly for power series rings. Given $M \in \mathbb{N}^n$ we write $x^M$ for $x_1^{M_1} \cdots x_n^{M_n}$. 

Let $R$ be a commutative ring. Given finite-rank free $\nZ_2$-graded $R$-modules $M, N$ and $\phi \in \Hom_R(M,N)$ we say that $\phi$ is \emph{even} (resp. \emph{odd}) if $\phi(M_i) \subseteq N_i$ (resp. $\phi(M_i) \subseteq N_{i+1}$) for all $i \in \nZ_2$. This makes $\Hom_R(M,N)$ into a $\nZ_2$-graded $R$-module. Given two homogeneous operators $\psi, \phi$ the \emph{graded commutator} is
\be
[\phi, \psi] = \phi \psi - (-1)^{|\phi||\psi|} \psi \phi\,.
\ee
In this note all operators are given a $\nZ_2$-grading and the commutator always denotes the graded commutator. We briefly recall some important operators on exterior algebras
\[
\bigwedge F_\xi = \bigwedge \bigoplus_{i=1}^r k \xi_i
\]
where $F_\xi = \bigoplus_{i=1}^r k \xi_i$ denotes a free $k$-module of rank $r$ with basis $\xi_1,\ldots,\xi_r$. We give $F_\xi$ a $\nZ_2$-grading by assigning $|\xi_i| = 1$, that is, $F_\xi \cong k^{\oplus r}[1]$. The inherited $\nZ_2$-grading on $\bigwedge F_\xi$ is the reduction mod $2$ of the usual $\nZ$-grading on the exterior algebra, e.g. $|\xi_1 \xi_2| = 0$.

We define odd operators $\xi_j \wedge (-), \xi_j^* \,\lrcorner\, (-)$ on $\bigwedge F_\xi$ by wedge product and contraction, respectively, where contraction is defined by the formula
\begin{align*}
\xi_j^* \,\lrcorner\, \Big( \xi_{i_1} \wedge \cdots \wedge \xi_{i_s} \Big) = \sum_{l=1}^s (-1)^{l-1} \delta_{j, i_l} \xi_{i_1} \wedge \cdots \wedge \widehat{ \xi_{i_l} } \wedge \cdots \wedge \xi_{i_s}\,.
\end{align*}
Often we will simply write $\xi_j$ for $\xi_j \wedge (-)$ and $\xi_j^*$ for $\xi_j^* \,\lrcorner\, (-)$. Clearly with this notation, as operators on $\bigwedge F_\xi$, we have the commutator (as always, graded)
\be\label{eq:wedge_contract_comm}
\big[ \xi_i, \xi_j^* \big] = \xi_i \xi_j^* + \xi_j^* \xi_i = \delta_{ij} \cdot 1
\ee
and also $[ \xi_i, \xi_j ] = [\xi_i^*, \xi_j^*] = 0$.

\subsection{$A_\infty$-categories}\label{section:ainfcat}

For the theory of $A_\infty$-categories we follow the notational conventions of \cite[\S 2]{lazaroiu}, which we now recall. Another good reference is Seidel's book \cite{seidel}. A small $\nZ_2$-graded $A_\infty$-category $\cat{A}$ over $k$ is specified by a set of objects $\operatorname{ob}(\cat{A})$ and $\nZ_2$-graded $k$-modules $\AA(a,b)$ for any pair $a,b \in \operatorname{ob}(\cat{A})$ together with $k$-linear maps
\[
\mu_{a_n,\ldots,a_0}: \AA(a_{n-1}, a_n) \otimes \cdots \otimes \AA(a_0, a_1) \lto \AA(a_0,a_n)
\]
of degree $2 - n \equiv n$ for every sequence of objects $a_0,\ldots,a_n$ with $n \ge 0$. If the objects involved are clear from the context, we will write $\mu_n$ for this map. These maps are required to satisfy the following equation for $n \ge 1$
\be\label{eq_ainf_constraints}
\sum_{\substack{i \ge 0, j \ge 1 \\ 1 \le i + j \le n}} (-1)^{ij + i + j + n} \mu_{n-j+1}\Big( x_n \otimes \cdots \otimes x_{i+j+1} \otimes \mu_j( x_{i+j} \otimes \cdots \otimes x_{i+1} ) \otimes x_i \otimes \cdots \otimes x_1 \Big) = 0
\ee
In particular we have a degree zero map
\[
\mu_2 = \mu_{cba}: \AA(b,c) \otimes \AA(a,b) \lto \AA(a,c)
\]
which satisfies the $n = 3$ equation
\begin{align*}
- &\mu_2( x \otimes \mu_2( y \otimes z ) ) + \mu_2( \mu_2( x \otimes y ) \otimes z ) + \mu_1 \mu_3( x \otimes y \otimes z )\\
& + \mu_3( \mu_1(x) \otimes y \otimes z ) + \mu_3( x \otimes \mu_1(y) \otimes z ) + \mu_3( x \otimes y \otimes \mu_1(z) ) = 0
\end{align*}
expressing that $\mu_2$ is associative up to the homotopy $\mu_3$ relative to the differential $\mu_1$. The operators $\mu_n$ are sometimes referred to as \emph{higher operations}. Any DG-category is an $A_\infty$-category where $\mu_1$ is the differential, $\mu_2$ is the composition and $\mu_n = 0$ for $n \ge 3$. Note that a DG-category has identity maps $u_a \in \AA^0(a,a)$ for all objects $a$, and these make $\cat{A}$ a \emph{strictly unital} $A_\infty$-category \cite[\S 2.1]{lazaroiu}, \cite[\S I.2]{seidel}.

To minimise the trauma of working with $A_\infty$-categories, it is convenient to adopt a different point of view on the higher operations, which eliminates most of the signs: from the $\mu$ we can define \emph{suspended forward compositions} \cite[\S 2.1]{lazaroiu}
\be
r_{a_0,\ldots,a_n}: \AA(a_0, a_1)[1] \otimes \cdots \otimes \AA(a_{n-1}, a_n)[1] \lto \AA(a_0,a_n)[1]
\ee
for which the $A_\infty$-constraints \eqref{eq_ainf_constraints} take the more attractive form
\[
\sum_{\substack{i \ge 0, j \ge 1 \\ 1 \le i + j \le n}} r_{a_0, \ldots, a_i, a_{i+j}, \ldots, a_n} \circ \Big( \operatorname{id}_{a_0a_1} \otimes \cdots \otimes \operatorname{id}_{a_{i-1}a_i} \otimes r_{a_i,\ldots,a_{i+j}} \otimes \operatorname{id}_{a_{i+j}a_{i+j+1}} \otimes \cdots \otimes \operatorname{id}_{a_{n-1}a_n} \Big) = 0
\]
As before we write $r_n$ for $r_{a_0,\ldots,a_n}$ if the indices are clear. Note that while $\mu_n$ has $\nZ_2$-degree $n$, the $r_n$'s are all odd operators. The $\nZ_2$-degree of a homogeneous element $x \in \AA(a,b)$ will be denoted $|x|$ and we write $\widetilde{x} = |x| + 1$ for the degree of $x$ viewed as an element of $\AA(a,b)[1]$. Sometimes we refer to this as the \emph{tilde grading}. We refer the reader to \cite{lazaroiu} for the definition of the suspended forward compositions, but note $\mu_1(x) = r_1(x)$ and
\be\label{eq:mu2vsr2}
\mu_2(x_1 \otimes x_2) = (-1)^{\widetilde{x}_1\widetilde{x}_2 + \widetilde{x}_1 + 1} r_2(x_2 \otimes x_1)\,.
\ee
The Koszul sign rule always applies when we evaluate the application of a tensor product of homogeneous linear maps on a tensor, for example since $r_2$ is odd
\[
1 \otimes r_2: \AA(a,b)[1] \otimes \AA(b,c)[1] \otimes \AA(c,d)[1] \lto \AA(a,b)[1] \otimes \AA(b,d)[1]
\]
applied to a tensor $x_3 \otimes x_2 \otimes x_1$ is
\[
(1 \otimes r_2)(x_3 \otimes x_2 \otimes x_1) = (-1)^{\widetilde{x_3}} x_3 \otimes r_2( x_2 \otimes x_1 )
\]
where we had to know that the domain involved $\AA(a,b)[1]$ rather than $\AA(a,b)$ to know that we were supposed to use the tilde grading on $x_3$. 

We also use the sector decomposition of \cite[\S 2.2]{lazaroiu}. We associate to $\cat{A}$ the $k$-module
\[
\HH_{\cat{A}} = \bigoplus_{a,b \in \operatorname{ob}(\AA)} \AA(a,b)
\]
equipped with the induced $\nZ_2$-grading. Let $Q$ be the commutative associative $k$-algebra (without identity) generated by $\epsilon_a$ for $a \in \operatorname{ob}(\cat{A})$ subject to the relations $\epsilon_a \epsilon_b = \delta_{ab} \epsilon_a$ (this non-unital algebra is denoted $R$ in \cite{lazaroiu}). Then $\HH_{\cat{A}}$ has a $Q$-bimodule structure in which $\epsilon_a$ acts on the left by the projector of $\HH_{\AA}$ onto the subspace $\bigoplus_{b \in \operatorname{ob}(\AA)} \AA(a,b)$ and $\epsilon_b$ acts on the right by the projector of $\HH_{\AA}$ onto the subspace $\bigoplus_{a \in \operatorname{ob}(\AA)} \AA(a,b)$. The $n$-fold tensor product of the $Q$-bimodule $\HH_{\AA}$ over $Q$ is
\be\label{eq:bimodule_tensor_hh}
\HH_{\AA}^{\otimes_Q n} = \bigoplus_{a_0,\ldots,a_n \in \operatorname{ob}(\AA)} \AA(a_0,a_1) \otimes \AA(a_1,a_2) \otimes \cdots \otimes \AA(a_{n-1},a_n)
\ee
with the obvious $Q$-bimodule structure involving the values of $a_0, a_n$, so that the forward suspended product $r_n$ is an odd $Q$-bilinear map from $\HH_{\AA}[1]^{\otimes_Q n} \lto \HH_{\AA}[1]$.

\section{The idempotent finite model}\label{section:the_model}

Throughout $k$ is a commutative $\mathbb{Q}$-algebra and all $A_\infty$-categories are $k$-linear.

\begin{definition} An $A_\infty$-category $\mathcal{C}$ is called \emph{Hom-finite} if for every pair $a,b$ of objects the underlying $k$-module of $\mathcal{C}(a,b)$ is a finitely generated and projective $k$-module.
\end{definition}

\begin{definition} An \emph{idempotent finite $A_\infty$-model} of a DG-category $\AA$ is a pair $(\BB, E)$ consisting of a Hom-finite $A_\infty$-category $\BB$ and $A_\infty$-functor $E: \BB \lto \BB$, and a diagram
\[
\xymatrix@C+3pc{
\AA \ar@<1ex>[r]^I & \BB \ar@(ur,dr)^E \ar@<1ex>[l]^P
}
\]
of $A_\infty$-functors and $A_\infty$-homotopies $P \circ I \simeq 1$ and $I \circ P \simeq E$. 
\end{definition}

We recall from \cite{lgdual} the definition of a potential:

\begin{definition}\label{defn:potential} A polynomial $W \in k[x_1,\ldots,x_n]$ is a \textsl{potential} if
\begin{itemize}
\item[(i)] $\partial_{x_1} W,\ldots,\partial_{x_n} W$ is a quasi-regular sequence;
\item[(ii)] $k[x_1,\ldots,x_n]/(\partial_{x_1} W,\ldots,\partial_{x_n} W)$ is a finitely generated free $k$-module;
\item[(iii)] the Koszul complex of $\partial_{x_1} W,\ldots,\partial_{x_n} W$ is exact except in degree zero.
\end{itemize}
\end{definition}

A typical example is a polynomial $W \in \mathbb{C}[x_1,\ldots,x_n]$ with isolated critical points \cite[Example 2.5]{lgdual}. As shown in \cite{lgdual}, these hypotheses on a potential are sufficient to produce all the properties relevant to two-dimensional topological field theory, even if $k$ is not a field. If $k$ is Noetherian then (iii) follows from (i).

\begin{definition} Given a potential $W \in R = k[x_1,\ldots,x_n]$ the DG-category $\AA = \mfdg(R,W)$ has as objects \emph{matrix factorisations} of $W$ over $R$ \cite{EisenbudMF} that is, the pairs $(X, d_X)$ consisting of a $\mathbb{Z}_2$-graded free $R$-module $X$ of finite rank and an odd $R$-linear operator $d_X: X \lto X$ satisfying $d_X^2 = W \cdot 1_X$. We define
\begin{gather*}
\AA(X,Y) = \big( \Hom_R(X,Y) \,, d_{\Hom} \big)\,,\\
d_{\Hom}(\alpha)  = d_Y \circ \alpha - (-1)^{|\alpha|} \alpha \circ d_X\,.
\end{gather*}
The composition is the usual composition of linear maps.
\end{definition}

Throughout we set $I = ( \partial_{x_i} W, \ldots, \partial_{x_n} W )$ to be the defining ideal of the critical locus, and write $\widehat{R}$ for the $I$-adic completion. Let $W \in R$ be a potential and let $\AA$ be a full sub-DG-category of the DG-category of matrix factorisations $\mfdg( R, W )$. The first observation is that we can replace $\AA$ by the completion $\AA \otimes_R \widehat{R}$.

\begin{lemma}\label{lemma:completion_he} Let $W \in R$ be a potential. The canonical DG-functor
\[
\AA \lto \AA \otimes_R \widehat{R}
\]
is a $k$-linear homotopy equivalence, that is, for every pair of matrix factorisations $X,Y$
\[
\Hom_R(X,Y) \lto \Hom_R(X,Y) \otimes_R \widehat{R}
\]
is a homotopy equivalence over $k$.
\end{lemma}
\begin{proof}
See Appendix \ref{section:noetherian}.
\end{proof}

To construct an idempotent finite model of $\AA$ we form the extension
\be\label{eq:defn_AAtheta}
\AA_{\theta} = \bigwedge( k \theta_1 \oplus \cdots \oplus k \theta_n ) \otimes \AA \otimes_R \widehat{R}
\ee
which is a DG-category with the same objects as $\AA$ and mapping complexes
\[
\AA_{\theta}( X, Y ) = \bigwedge( k \theta_1 \oplus \cdots \oplus k \theta_n ) \otimes \AA(X,Y) \otimes_R \widehat{R}\,.
\]
The differentials in $\AA_{\theta}$ are induced from $\AA$ and the composition rule is obtained from multiplication in the exterior algebra and composition in $\AA$, taking into account Koszul signs when moving $\theta$-forms past morphisms in $\AA$. Next we consider the $\nZ_2$-graded modules $\BB(X,Y)$ and the $Q$-bimodule $\HH_{\BB}$ defined in Section \ref{section:ainfcat}, namely
\begin{gather*}
\BB(X,Y) = R/I \otimes_R \Hom_R(X,Y)\\
\HH_{\BB} = R/I \otimes_R \HH_{\AA} = \bigoplus_{X,Y \in \AA} \BB(X,Y)\,.
\end{gather*}
At the moment this has no additional structure: it is just a module, not an $A_\infty$-category. But we note that since $\Hom_R(X,Y)$ is a free $R$-module of finite rank, and $R/I$ is free of finite rank over $k$, the spaces $\BB(X,Y)$ are free $k$-modules of finite rank. The goal of this section is to construct higher $A_\infty$-operations $\rho_k$ on $\HH_{\BB}$. 

\begin{setup}\label{setup:overall} Throughout we adopt the following notation:
\begin{itemize}
\item $R = k[\bold{x}] = k[x_1,\ldots,x_n]$.
\item $F_\theta = \bigoplus_{i=1}^n k\theta_i$ is a free $\mathbb{Z}_2$-graded $k$-module of rank $n$, with $|\theta_i| = 1$.
\item $t_1,\ldots,t_n$ is a quasi-regular sequence in $R$, such that with $I = (t_1,\ldots,t_n)$
\begin{itemize}
\item $R/I$ is a finitely generated free $k$-module
\item each $t_i$ acts null-homotopically on $\AA(X,Y)$ for all $X,Y \in \AA$
\item the Koszul complex of $t_1,\ldots,t_n$ over $R$ is exact except in degree zero.
\end{itemize}
\item We choose a $k$-linear section $\sigma: R/I \lto R$ of the quotient map $R \lto R/I$ and as in Appendix \ref{section:formaltub} we write $\nabla$ for the associated connection with components $\partial_{t_i}$.
\item $\lambda_i^X$ is a null-homotopy for the action of $t_i$ on $\AA(X,X)$ for each $X \in \AA, 1 \le i \le n$.
\item We choose for $X \in \AA$ an isomorphism of $\nZ_2$-graded $R$-modules
\[
X \cong \coeff{X} \otimes R
\]
where $\coeff{X}$ is a finitely generated free $\nZ_2$-graded $k$-module. Hence
\be
\AA(X,Y) = \Hom_R(X,Y) \cong \Hom_k(\coeff{X},\coeff{Y}) \otimes R\,.\label{eq:chosenCiso}
\ee
\end{itemize}
\end{setup}

\begin{remark} By the hypothesis that $W$ is a potential, the sequence $\bold{t} = (\partial_{x_1} W, \ldots, \partial_{x_n} W)$ satisfies the hypotheses and we may choose $\lambda_i^X$ to be the operator $\partial_{x_i}(d_X)$ defined by choosing a homogeneous basis for $X$ and differentiating entry-wise the matrix $d_X$ in that basis. However some choices of $\bold{t}$ and the $\lambda_i^X$ may be better than others, in the sense that they lead to simpler Feynman rules.
\end{remark}

To explain the construction of the higher operations on $\BB$, it is convenient to switch to an alternative presentation of the spaces $\AA(X,Y), \BB(X,Y)$. Consider the following $\nZ_2$-graded $k$-modules, where the grading comes only from $\bigwedge F_\theta$ and the Hom-space:
\begin{gather*}
\AA'(X,Y) = R/I \otimes \Hom_k(\widetilde{X},\widetilde{Y}) \otimes k\llbracket \bold{t} \rrbracket\,,\\
\AA'_\theta(X,Y) = \bigwedge F_\theta \otimes R/I \otimes \Hom_k(\widetilde{X},\widetilde{Y}) \otimes k\llbracket \bold{t} \rrbracket\,, \\
\BB'(X,Y) = R/I \otimes \Hom_k(\widetilde{X},\widetilde{Y})\,.
\end{gather*}
Using \eqref{eq:chosenCiso} there is an isomorphism of $\nZ_2$-graded $R$-modules $\BB'(X,Y) \cong \BB(X,Y)$. By Lemma \ref{prop_algtube} there is a $k\llbracket \bold{t} \rrbracket$-linear isomorphism $\sigmastar: R/I \otimes k \llbracket \bold{t} \rrbracket \lto \widehat{R}$ and combined with \eqref{eq:chosenCiso} this induces an isomorphism of $\nZ_2$-graded $k\llbracket \bold{t} \rrbracket$-modules
\be\label{eq:transfer_iso_intro}
\xymatrix@C+2pc{
\sigmastar: R/I \otimes \Hom_k(\coeff{Y},\coeff{X}) \otimes k\llbracket \bold{t} \rrbracket \ar[r]^-{\cong} & 
\Hom_R(Y,X) \otimes_R \widehat{R}
}
\ee
which induces an isomorphism of $\nZ_2$-graded $k\llbracket \bold{t} \rrbracket$-modules
\[
\xymatrix@C+2pc{
\AA'_\theta(X,Y) \ar[r]_-{\cong}^{\sigmastar} & \AA_\theta(X,Y)
}\,.
\]
Hence there are induced isomorphisms $\HH_{\AA'_\theta} \cong \HH_{\AA_\theta}$ and $\HH_{\BB} \cong \HH_{\BB'}$. Using these identifications we transfer operators on $\AA_\theta,\BB$ to their primed cousins, usually without a change in notation. For example we write $d_\AA$ for
\[
\xymatrix@C+2pc{
\HH_{\AA'_\theta} \ar[r]^-{\sigmastar}_-{\cong} & \HH_{\AA_\theta} \ar[r]^-{d_{\AA}} & \HH_{\AA_\theta} \ar[r]^-{(\sigmastar)^{-1}}_-{\cong} & \HH_{\AA'_\theta}
}\,.
\]
This map is the differential in a $k\llbracket \bold{t} \rrbracket$-linear DG-category structure on $\AA'_\theta$, with the forward suspended composition $r_2$ in this DG-structure given by
\[
\xymatrix@C+2pc{
\HH_{\AA'_\theta}[1] \otimes_Q \HH_{\AA'_\theta}[1] \ar[r]^-{\cong} & \HH_{\AA_\theta}[1] \otimes_Q \HH_{\AA_\theta}[1] \ar[r]^-{r_2} & \HH_{\AA_\theta}[1] \ar[r]^-{\cong} & \HH_{\AA'_{\theta}}[1]
}
\]
where the unlabelled isomorphisms are $\sigmastar \otimes \sigmastar$ and $(\sigmastar)^{-1}$. Going forward when we refer to $\AA'_\theta$ as a DG-category this structure is understood. Finally the tensor product of the inclusions $k \subset \bigwedge F_\theta$ and $k \subset k\llbracket \bold{t} \rrbracket$, respectively the projections $\bigwedge F_\theta \lto k$ and $k \llbracket \bold{t} \rrbracket \lto k$ define $k$-linear maps $\sigma$ and $\pi$ as in the diagram
\[
\xymatrix@C+2pc{
R/I \otimes \Hom_k(\widetilde{X},\widetilde{Y}) \ar@<1ex>[r]^-{\sigma} & \bigwedge F_\theta \otimes R/I \otimes \Hom_k(\widetilde{X},\widetilde{Y}) \otimes k \llbracket \bold{t} \rrbracket\ar@<1ex>[l]^-{\pi}
}
\]
and hence degree zero $k$-linear maps
\[
\xymatrix@C+3pc{
\HH_{\BB'} \ar@<1ex>[r]^-{\sigma} & \HH_{\AA'_\theta}\ar@<1ex>[l]^-{\pi}
}\,.
\]

\begin{definition}\label{defn:atiyah_class} The \emph{critical Atiyah class} of $\AA$ is the operator on $\HH_{\AA'_\theta}$ given by
\[
\vAt_{\AA} = [ d_{\AA}, \nabla ] = d_{\AA} \nabla + \nabla d_{\AA}
\]
where $\nabla = \sum_{i=1}^n \theta_i \partial_{t_i}$ is the connection of Section \ref{section:formaltub}. This is a closed $k\llbracket \bold{t} \rrbracket$-linear operator, independent up to $k$-linear homotopy of the choice of connection.
\end{definition}

We call $\vAt_{\AA}$ the \emph{critical} Atiyah class since it is defined using the connection $\nabla$, which is a kind of derivative in the directions normal to the critical locus, and some name seems useful to distinguish $\vAt_{\AA}$ from various other Atiyah classes also playing a role in the theory of matrix factorisations, for example the associative Atiyah classes of \cite{lgdual}.

\begin{definition}\label{definition:zeta} Since $\HH_{\AA'_\theta}$ is a module over $\bigwedge F_\theta \otimes k\llbracket \bold{t} \rrbracket$ we may define
\be
(\boldsymbol{\theta}, \bold{t}) \HH_{\AA'_\theta} \subseteq \HH_{\AA'_\theta}
\ee
where $(\boldsymbol{\theta}, \bold{t})$ is the two-sided ideal spanned by the $\theta_i, t_j$. We define the $k$-linear operator
\begin{gather*}
\zeta: (\boldsymbol{\theta}, \bold{t}) \HH_{\AA'_\theta} \lto (\boldsymbol{\theta}, \bold{t}) \HH_{\AA'_\theta}\\
\zeta\big( \omega \otimes z \otimes \alpha \otimes f \big) = \sum_{\delta \in \mathbb{N}^n} \frac{1}{|\omega| + |\delta|} \omega \otimes z \otimes \alpha \otimes f_\delta t^\delta
\end{gather*}
for a homogeneous $\theta$-form $\omega$ and $f = \sum_{\delta \in \mathbb{N}^n} f_\delta t^\delta \in k\llbracket \bold{t} \rrbracket$. Evaluated on polynomial $f$ this is the inverse of the grading operator for \emph{virtual degree} which is a $\nZ$-grading $\Vert - \Vert$ in which $\Vert\theta_i\Vert = \Vert t_i \Vert = 1$ for $1 \le i \le n$ and $\Vert z \Vert = \Vert \alpha \Vert = 0$ for $z \in R/I, \alpha \in \Hom_k(\widetilde{X},\widetilde{Y})$.
\end{definition}

\begin{definition}\label{defn:important_operators} We introduce the following operators:
\begin{align*}
\sigma_\infty &= \sum_{m \ge 0} (-1)^m (\zeta \vAt_{\AA})^m \sigma: \HH_{\BB'} \lto \HH_{\AA'_\theta}\\
\phi_\infty &= \sum_{m \ge 0} (-1)^m (\zeta \vAt_{\AA})^m \zeta \nabla: \HH_{\AA'_\theta} \lto \HH_{\AA'_\theta}\\
\delta &= \sum_{i=1}^n \lambda^\bullet_i \theta_i^*: \HH_{\AA'_\theta} \lto \HH_{\AA'_\theta}
\end{align*}
where $\lambda_i^\bullet$ acts on $\Hom_R(X,Y)$ by post-composition
\[
\lambda_i^\bullet(\alpha) = \lambda_i^Y \circ \alpha\,.
\]
Note that the sums involved are all finite, since $\vAt_{\AA}$ has positive $\theta$-degree. The $\nZ_2$-degrees of these operators are $|\delta| = |\sigma_\infty| = 0$ and $|\phi_\infty| = 1$.
\end{definition}

We now have the notation to state the main theorem. Let $\cat{BT}_k$ denote the set of all valid plane binary trees with $k$ inputs (in the sense of Appendix \ref{section:trees}). Given such a tree $T$, we add some additional vertices and then \emph{decorate} the tree by inserting operators at each vertex. The \emph{denotation} of such an operator decorated tree is defined by reading the tree as a ``flowchart'' with inputs inserted at the leaves and the output read off from the root. For example the tree $T$ in Figure \ref{fig:opdectree} has for its denotation the operator
\begin{figure}
\begin{center}
\includegraphics[scale=0.35]{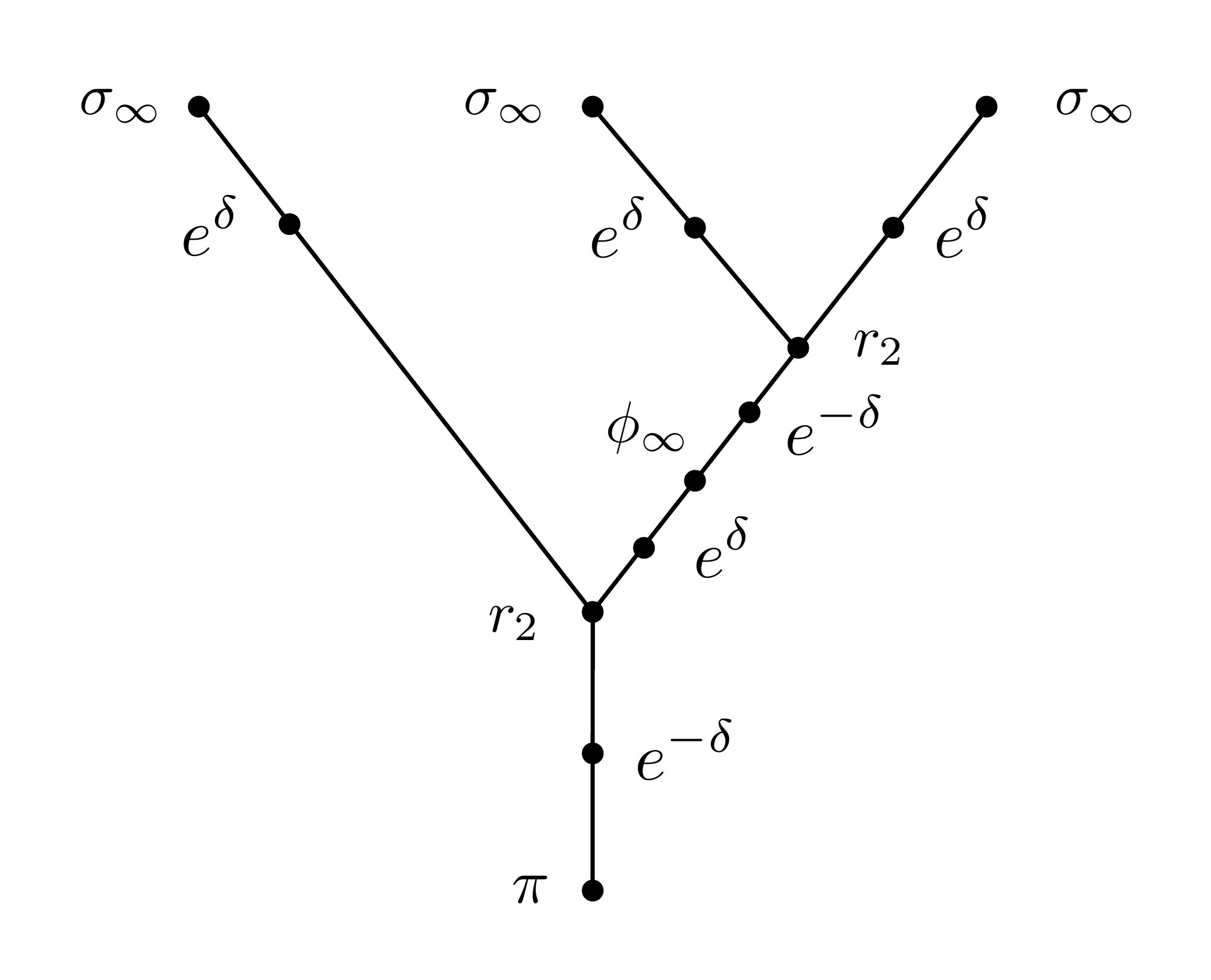}
\end{center}
\centering
\caption{Example of an operator decorated tree.}\label{fig:opdectree}
\end{figure}
\be\label{eq:explicit_tree_operator}
\pi e^{-\delta} r_2\Big( e^{\delta}\sigma_\infty \otimes e^{\delta} \phi_\infty e^{-\delta} r_2\Big( e^{\delta} \sigma_\infty \otimes e^{\delta} \sigma_\infty \Big) \Big)\,.
\ee
See Appendix \ref{section:trees} for our conventions on trees, decorations and denotations. We note that these denotations involve Koszul signs when evaluated, arising from the $\nZ_2$-degree (with respect to the tilde grading) of the involved operators (recall for example that $r_2$ is odd). See Section \ref{section:ainfcat} for the definition of the ring $Q$ and the $Q$-bimodule structure on $\HH_{\BB}$, and note that we write $e_i(T)$ for the number of internal edges in a tree $T$.

\begin{theorem}\label{theorem:main_ainfty_products} Define the odd $Q$-bilinear map
\be\label{eq:rho_kintermsoftrees}
\rho_k = \sum_{T \in \cat{BT}_k} (-1)^{e_i(T)} \rho_T : \HH_{\BB'}[1]^{\otimes_Q k} \lto \HH_{\BB'}[1]
\ee
where $\rho_T$ is the denotation of the decoration with coefficient ring $Q$ which assigns $\HH_{\AA'_\theta}[1]$ to every leaf and $\HH_{\BB'}[1]$ to every edge, and to
\begin{center}
\begin{itemize}
\item \textbf{inputs:} $e^\delta \sigma_\infty$
\item \textbf{internal edges:} $e^\delta \phi_\infty e^{-\delta}$
\item \textbf{internal vertices:} $r_2$
\item \textbf{root:} $\pi e^{-\delta}$
\end{itemize}
\end{center}
Then $( \BB, \rho = \{ \rho_k \}_{k \ge 1} )$ is a strictly unital $A_\infty$-category and there are $A_\infty$-functors
\[
\xymatrix@C+3pc{
\AA_\theta \ar@<1ex>[r]^F & \BB \ar@<1ex>[l]^G
}
\]
and an $A_\infty$-homotopy $G \circ F \simeq 1_{\AA_{\theta}}$. 
\end{theorem}
\begin{proof}
The full details are given in Appendix \ref{section:proofs}, but in short this is the usual transfer of $A_\infty$-structure via homological perturbation applied to a particular choice of strong deformation retract arising from the connection $\nabla$ and the isomorphism $e^{\delta}$.
\end{proof}

The projector $e$ of \eqref{eq:projector_e} can be written in terms of creation and annihilation operators
\[
e = \theta_n^* \cdots \theta_1^* \theta_1 \cdots \theta_n: \bigwedge F_\theta \lto \bigwedge F_\theta
\]
where $\theta_i$ denotes the operator $\theta_i \wedge (-)$ and $\theta_i^*$ denotes contraction $\theta_i^* \,\lrcorner\, (-)$. This is a morphism of algebras and induces a functor of DG-categories $e: \AA_\theta \lto \AA_\theta$.

\begin{definition} Let $E$ be the following composite of $A_\infty$-functors 
\[
\xymatrix@C+2pc{
\BB \ar[r]^-G & \AA_\theta \ar[r]^-{e} & \AA_\theta \ar[r]^-F & \BB
}\,.
\]
\end{definition}

\begin{corollary}\label{corollary:idempotent_finite_model} The tuple $(\BB, \rho, E)$ is an idempotent finite $A_\infty$-model of $\AA \otimes_R \widehat{R}$.
\end{corollary}
\begin{proof}
Consider the diagram
\[
\xymatrix@C+3pc{
\AA \otimes_R \widehat{R} \ar@<1ex>[r]^-{i} & \AA_\theta \ar@<1ex>[l]^-{p} \ar@<1ex>[r]^F & \BB \ar@<1ex>[l]^G
}
\]
where $i$ is the natural inclusion and $p$ is the projection, so that $p \circ i = 1$ and $i \circ p = e$. These are both DG-functors. We define $I = F \circ i$ and $P = p \circ G$ as $A_\infty$-functors. Since we have an $A_\infty$-homotopy $G \circ F \simeq 1$ we have an $A_\infty$-homotopy
\[
P I = p G F i \simeq p i = 1
\]
and by definition $E = I \circ P$.
\end{proof}

At a cohomological level the pushforward of matrix factorisations is expressed in terms of residues and null-homotopies $\lambda$, see for example the results in \cite[\S 11.2]{pushforward} on Chern characters. These residues can be understood as traces of products of commutators with the connection $\nabla$ \cite[Proposition B.4]{pushforward}. The results just stated extend this ``closed sector'' or cohomological level analysis of pushforwards via residues to the ``open sector'' or categorical level, where the supertraces are removed and the higher operations of the idempotent finite model $\BB$ are described directly in terms of the commutators $\vAt_{\AA} = [d_{\AA}, \nabla]$ and homotopies $\lambda$. Moreover, these formulas arise from homological perturbation applied to a kind of tubular neighborhood of the critical locus, so it seems natural to interpret $(\BB, E)$ as a kind of ``$A_\infty$-categorical residue'' of $\AA$ along the subscheme $\Spec(R/I)$.
\vspace{0.2cm}

Recall that the purpose of the idempotent $E$ is that it encodes the information necessary to ``locate'' $\AA$ within the larger object $\BB$. The information in the lowest piece $E_1$ of this $A_\infty$-idempotent is the simplest, as it locates $\AA$ as a subcomplex within $\BB$. 

\begin{definition}\label{defn:gamma_anddagger} Let $\gamma_i, \gamma_i^\dagger$ be the $k$-linear cochain maps
\begin{gather*}
\xymatrix@C+2pc{
\HH_{\BB'} \ar[r]^-{G_1} & \HH_{\AA'_{\theta}} \ar[r]^-{\theta_i^*} & \HH_{\AA'_{\theta}} \ar[r]^-{F_1} & \HH_{\BB'}
}\\
\xymatrix@C+2pc{
\HH_{\BB'} \ar[r]^-{G_1} & \HH_{\AA'_{\theta}} \ar[r]^-{\theta_i} & \HH_{\AA'_{\theta}} \ar[r]^-{F_1} & \HH_{\BB'}\,.
}
\end{gather*}
respectively.
\end{definition}

\begin{theorem}\label{theorem:homotopy_clifford} There is a $k$-linear homotopy
\be
E_1 \simeq \gamma_n \cdots \gamma_1 \gamma_1^\dagger \cdots \gamma_n^\dagger
\ee
and $k$-linear homotopies $\gamma_i \simeq \vAt_i$ and
\be
\gamma_i^\dagger \simeq -\lambda_i - \sum_{m \ge 1} \sum_{q_1,\ldots,q_m} \frac{1}{(m+1)!} \big[ \lambda_{q_m}\,, \big[ \lambda_{q_{m-1}}, \big[ \cdots [ \lambda_{q_1}, \lambda_i ] \cdots \big] \At_{q_1} \cdots \At_{q_m}
\ee
where $\vAt_i = [ d_{\AA}, \partial_{t_i} ]$ denotes the $i$th component of the Atiyah class $\vAt_{\AA}$, viewed as an odd closed $k$-linear operator on $\HH_{\BB'}$.
\end{theorem}
\begin{proof}
This is essentially immediate from \cite{cut}, see Appendix \ref{section:proofs} for details.
\end{proof}

\subsection{Algorithms}\label{section:the_algorithms}

When $k$ is a field there are algorithms which compute the $A_\infty$-functors $I, P, E$ and the higher $A_\infty$-products $\rho$ in the sense that once we choose a $k$-basis for $R/I$ and homogeneous $R$-bases for the matrix factorisations, for each fixed $k \ge 1$ there is an algorithm computing the coefficients in the matrices $I_k, P_k, E_k, \rho_k$. We explain this algorithm in detail only for $\rho_k$ as the algorithms for $I_k, P_k, E_k$ are variations on the same theme using \cite{markl_transfer}.

The algorithm is implicit in the presentation of $\rho_k$ as the sum of denotations of operator decorated trees, provided we have algorithms for computing the section $\sigma$, Atiyah classes $\vAt_{\AA}$ and homotopies $\lambda_i$ as operators on $\HH_{\AA'_\theta}$. If $k$ is a field, then by choosing a Gr\"obner basis of the ideal $I$ we obtain such algorithms; see Remark \ref{remark:compute_rdelta} and Remark \ref{remark:grobner}. Over the course of Section \ref{section:towards} and Section \ref{section:feynman_diagram} we present the details of this algorithm in the case where the matrix factorisations are of Koszul type, using Feynman diagrams.

\begin{remark} For general $k$ the algorithmic content of the theory depends on the availability of a replacement for Gr\"obner basis methods. One important case where such methods are available is the example of potentials $W \in k[x_1,\ldots,x_n]$ with $k = \mathbb{C}[u_1,\ldots,u_v]$, using Gr\"obner systems \cite{weispfenning} and constructible partitions.
\end{remark}

\begin{remark} Finding a Hom-finite $A_\infty$-category $A_\infty$-homotopy-equivalent to $\AA$ is equivalent to \emph{splitting} the idempotent $E$ within Hom-finite $A_\infty$-categories. We do not know a general algorithm which performs this splitting. However, this can be done when we have the data of a chosen cohomological splitting, for example in the case of the endomorphism DG-algebra of the standard generator when $k$ is a field; see Section \ref{section:generator}.
\end{remark}

\section{Towards Feynman diagrams}\label{section:towards}

In this section we collect some technical lemmas needed in the presentation of the Feynman rules, in the next section. Throughout the conventions of Setup \ref{setup:overall} remain in force. See Appendix \ref{section:trees} for our conventions on trees, decorations and denotations. Given a binary plane tree $T$ we denote by $T'$ the \emph{mirror} of $T$, which is obtained by exchanging the left and right branch at every vertex. Associated to a decoration $D$ of $T$ is a mirror decoration $D'$ of $T'$. Given a plane tree $T$ decorated by $D$ as explained in Theorem \ref{theorem:main_ainfty_products} let $\operatorname{eval}_{D'}$ be the mirror decoration evaluated without Koszul signs (Definition \ref{defn:evaluation_tree}).

\begin{lemma}\label{prop:replacer2} We have
\be\label{eq:proprhoT}
\rho_T( \beta_1, \ldots, \beta_k ) = (-1)^{\sum_{i < j} \widetilde{\beta}_i \widetilde{\beta}_j + \sum_i \widetilde{\beta}_i P_i + k + 1} \operatorname{eval}_{D'}( \beta_k, \ldots, \beta_1 )\,.
\ee
where $P_i$ is the number of times the path from the $i$th leaf in $T$ (counting from the left) enters a trivalent vertex as the right-hand branch on its way to the root, and $k+1$ is the number of internal vertices in $T$.
\end{lemma}
\begin{proof}
Let us begin with the special case given in Figure \ref{fig:opdectree}, using
\be\label{eq:mu2vsr2_v2}
r_2( \beta_1, \beta_2 ) = (-1)^{\widetilde{\beta_1} \widetilde{\beta_2} + \widetilde{\beta_2} + 1}
\mu_2(\beta_2 \otimes \beta_1)
\ee
and the operator given in \eqref{eq:explicit_tree_operator} to compute that
\begin{align*}
\rho_T( \beta_1, \beta_2, \beta_3 ) &= \pi e^{-\delta} r_2\Big( e^{\delta}\sigma_\infty \otimes e^{\delta} \phi_\infty e^{-\delta} r_2\Big( e^{\delta} \sigma_\infty \otimes e^{\delta} \sigma_\infty \Big) \Big)( \beta_1 \otimes \beta_2 \otimes \beta_3 )\\
&= (-1)^{a} \pi e^{-\delta} r_2\Big( e^{\delta}\sigma_\infty(\beta_1) \otimes e^{\delta} \phi_\infty e^{-\delta} r_2\Big( e^{\delta} \sigma_\infty(\beta_2) \otimes e^{\delta} \sigma_\infty(\beta_3) \Big) \Big)
\end{align*}
where $a = \widetilde{\beta_1}( |\phi_\infty| + |r_2| ) \equiv 0$ gives the Koszul sign arising from moving the inputs ``into position''. Note that since $|\delta| = |\sigma_\infty| = 0$ and every $r_2$ decorating the tree $T$, except for the one adjacent to the root, is followed immediately by a $\phi_\infty$, this sign is always $+1$.

Hence the signs that arise in computing $\rho_T(\beta_1,\ldots,\beta_k)$ in terms of $\mu_2$ on the mirrored tree arise entirely from \eqref{eq:mu2vsr2_v2}. If we continue to calculate, we find
\begin{align*}
&= (-1)^{b} \pi e^{-\delta} \mu_2\Big( e^{\delta} \phi_\infty e^{-\delta} r_2\Big( e^{\delta} \sigma_\infty(\beta_2) \otimes e^{\delta} \sigma_\infty(\beta_3) \Big) \otimes e^{\delta}\sigma_\infty(\beta_1) \Big)\\
&= (-1)^{b + c} \pi e^{-\delta} \mu_2\Big( e^{\delta} \phi_\infty e^{-\delta} \mu_2\Big( e^{\delta} \sigma_\infty(\beta_3) \otimes e^{\delta} \sigma_\infty(\beta_2) \Big) \otimes e^{\delta}\sigma_\infty(\beta_1) \Big)\\
&= (-1)^{b + c} \operatorname{eval}_{D'}( \beta_3, \beta_2, \beta_1 )
\end{align*}
where
\[
b = \widetilde{\beta_1}( \widetilde{\beta_2} + \widetilde{\beta_3} ) + \widetilde{\beta_2} + \widetilde{\beta_3} + 1\,, \qquad c = \widetilde{\beta_2}\widetilde{\beta_3} + \widetilde{\beta_3} + 1\,.
\]
This verifies the sign when $P_1 = 0, P_2 = 1, P_3 = 2$. By induction on the height of tree, it is easy to check that in general there is a contribution to the sign of a $\widetilde{\beta_i}\widetilde{\beta_j}$ at the vertex where the path from the $i$th and $j$th leaves to the root meet for the first time, and a $\widetilde{\beta_i}$ every time the path from the $i$th leaf enters a trivalent vertex on the right branch (of the original tree $T$), as claimed.
\end{proof}

\subsection{Transfer to $R/I \otimes k\llbracket \bold{t} \rrbracket$}\label{section:transfer_to_ri}

Recall that given a choice of section $\sigma: R/I \lto R$, which we have fixed above in Setup \ref{setup:overall}, there is by Lemma \ref{prop_algtube} an associated $k\llbracket \bold{t} \rrbracket$-linear isomorphism
\[
\sigmastar: R/I \otimes k \llbracket \bold{t} \rrbracket \lto \widehat{R}\,.
\]
From this we obtain \eqref{eq:transfer_iso_intro} which is used to transfer operators on $\Hom_R(X,Y) \otimes_R \widehat{R}$ (such as the differential or the homotopies $\lambda$) to operators on $R/I \otimes \Hom_k(\widetilde{X}, \widetilde{Y}) \otimes k\llbracket \bold{t} \rrbracket$. Since this introduces various complexities we should first justify why such transfers are necessary: that is, why do we prefer the left hand side of \eqref{eq:transfer_iso_intro} to the right hand side?

Recall that the higher products $\rho_k$ on $\HH_{\BB}$ are defined in terms of operators on the larger space $\HH_{\AA_\theta}$. If we are to reason about these higher products using Feynman diagrams, then to the extent that it is possible, the operators involved should be written as polynomials in creation and annihilation operators for either bosonic or fermionic Fock spaces (that is, in terms of multiplication by or the derivative with respect to ordinary polynomial variables $t$ or odd Grassmann variables $\theta$). It is not obvious \emph{a priori} how to do this: recall that in order to ensure that the connection $\nabla$ existed we had to pass from $R = k[x_1,\ldots,x_n]$ to the $I$-adic completion $\widehat{R}$, which in general is not a power series ring. For example, it is not clear how to express the operation of multiplication by $r \in R$, which we denote by $r^{\#}$, in terms of creation and annihilation operators on $\Hom_R(X,Y) \otimes_R \widehat{R}$.

The purpose of this section is then to explain how the isomorphism $\sigma_{\bold{t}}$ is the canonical means by which to express $r^{\#}$ in terms of creation operators for ``bosonic'' degrees of freedom, here represented by polynomials in the $t_i$.
\\

In what follows we fix a chosen $k$-basis of $R/I$, which we denote
\[
R/I = k z_1 \oplus \cdots \oplus k z_\mu\,.
\]
When $k$ is a field there is a natural monomial basis for $R/I$ associated to any choice of a monomial ordering on $k[x_1,\ldots,x_n]$ and Gr\"obner basis for $I$, see Remark \ref{remark:grobner}. Since $\sigmastar$ is not, in general, an algebra isomorphism (see Lemma \ref{prop_algtube}) there is information in the transfer of the multiplicative structure on $\widehat{R}$ to an operator on $R/I \otimes k\llbracket \bold{t} \rrbracket$, and we record this information in the following tensor:

\begin{definition}\label{defn_gamma} Let $\Gamma$ denote the $k$-linear map
\[
\xymatrix@C+2pc{
R/I \otimes R/I \ar[r]^-{ \sigma \otimes \sigma } & \widehat{R} \otimes \widehat{R} \ar[r]^-{m} & \widehat{R} \ar[r]^-{(\sigmastar)^{-1}} & R/I \otimes k\llbracket \bold{t} \rrbracket
}
\]
where $m$ denotes the usual multiplication on $\widehat{R}$. We define $\Gamma$ as a tensor via the formula
\[
\sigma(z_i)\sigma(z_j) = \sum_{k=1}^\mu \sum_{\delta \in \mathbb{N}^n} \Gamma^{ij}_{k \delta} \sigma(z_k) t^\delta\,.
\]
\end{definition}

\begin{definition}\label{defn:rsharp} Given $r \in R$ we write $r_{(i,\delta)}$ for the unique collection of coefficients in $k$ with the property that in $\widehat{R}$ there is an equality
\[
r = \sum_{i = 1}^\mu \sum_{\delta \in \mathbb{N}^n} r_{(i,\delta)} \,\sigma(z_i) t^{\delta}\,.
\]
Given $r \in R$ we denote by $r^{\#}$ the $k\llbracket \bold{t} \rrbracket$-linear operator
\[
\xymatrix@C+2pc{
R/I \otimes k\llbracket \bold{t} \rrbracket \ar[r]^-{\sigmastar} & \widehat{R} \ar[r]^{r} & \widehat{R} \ar[r]^-{(\sigmastar)^{-1}} & R/I \otimes k\llbracket \bold{t} \rrbracket
}
\]
where $r: \widehat{R} \lto \widehat{R}$ denotes multiplication by $r$.
\end{definition}

\begin{remark}\label{remark:compute_rdelta} For the overall construction of the idempotent finite model to be \emph{constructive} in the sense elaborated above, it is crucial that we have an algorithm for computing these coefficients $r_{(i, \delta)}$. In the notation of Section \ref{section:formaltub}, $r_{(i, \delta)}$ is the coefficient of $z_i$ in the vector $r_\delta \in R/I$, so it suffices to understand how to compute the $r_\delta$.

As a trivial example, if $\bold{t} = (x_1,\ldots,x_n)$ then $R/I = k$ so $\mu = 1$ and $r_{(1,\delta)}$ is just the coefficient of the monomial $t^{\delta} = x^{\delta}$ in the polynomial $r$. In general, when $k$ is a field there is an algorithm for computing $r_\delta$, see Remark \ref{remark:grobner}.
\end{remark}

\begin{lemma}\label{lemma:rsharp_explicit} The operator $r^{\#}$ is given in terms of the tensor $\Gamma$ by the formula
\be
r^{\#}(z_i) = \sum_{l=1}^\mu \sum_{\delta \in \mathbb{N}^n} \Big[ \sum_{\alpha + \beta = \delta } \sum_{k=1}^\mu r_{(k,\alpha)} \Gamma^{ki}_{l\beta} \Big] z_l \otimes t^\delta\,.
\ee
\end{lemma}
\begin{proof}
We have
\begin{align*}
\sigmastar r^{\#}(z_i \otimes 1) &= r \sigma(z_i)\\
&= \sum_{k, \alpha} r_{(k,\alpha)} [ \sigma(z_k) \sigma(z_i) ] t^{\alpha}\\
&= \sum_{k,\alpha,l,\beta} r_{(k,\alpha)} \Gamma^{ki}_{l \beta} \sigma(z_l) t^{\alpha + \beta}\\
&= \sum_\delta \sum_{k,l} \sum_{\alpha + \beta = \delta} r_{(k,\alpha)} \Gamma^{ki}_{l \beta} \sigma(z_l)t^\delta
\end{align*}
as claimed.
\end{proof}

\subsection{The operator $\zeta$}\label{section:propagator}

One of the most complex aspects of calculating the $A_\infty$-products described by Theorem \ref{theorem:main_ainfty_products} are the scalar factors contributed by the operator $\zeta$ which is the inverse of the grading operator for the virtual degree. In this section we provide a closer analysis of these factors.

While the virtual degree of Definition \ref{definition:zeta} is not a genuine $\nZ$-grading because $\HH_{\AA'_\theta}$ involves power series, for any given tree our calculations of higher $A_\infty$-product on $\BB$ only involve polynomials in the $t_i$, so there is no harm in thinking about the virtual degree as a $\nZ$-grading and $\zeta$ as its inverse. Observe that that the critical Atiyah class $\vAt_{\AA} = [ d_{\AA}, \nabla ]$ is not homogeneous with respect to this grading, because while $\nabla$ is homogeneous of degree zero with respect to the virtual degree (since $\theta_i$ has virtual degree $+1$ and $\partial_{t_i}$ has virtual degree $-1$) the operator $d_{\AA}$ involves multiplications by polynomials $r$ which need not have a consistent degree (viewed as operators $r^{\#}$ on $R/I \otimes k\llbracket \bold{t} \rrbracket$ as in the previous section). To analyse this operator on $\HH_{\AA'_\theta}$ we write
\[
d_{\AA} = \sum_{\delta \in \mathbb{N}^n} d_{\AA}^{\,(\delta)} t^\delta
\]
for some $k$-linear odd operators $d_{\AA}^{\,(\delta)}$ on $\bigoplus_{X,Y} R/I \otimes \Hom_k(\widetilde{X},\widetilde{Y})$. Then
\begin{align*}
(\zeta \vAt_{\AA})^m &= \sum_{\delta_1,\ldots,\delta_m} \zeta [ d^{\,(\delta_1)}_{\AA} t^{\delta_1}, \nabla ] \cdots \zeta [ d^{\,(\delta_m)}_{\AA} t^{\delta_m}, \nabla ]\\
 &= \sum_{i_1,\ldots,i_m} \sum_{\delta_1,\ldots,\delta_m} \zeta \Big\{ \theta_{i_1} \partial_{t_{i_1}}(t^{\delta_1}) d_{\AA}^{\,(\delta_1)} \Big\} \cdots \zeta \Big\{ \theta_{i_m} \partial_{t_{i_m}}(t^{\delta_m}) d_{\AA}^{\,(\delta_m)} \Big\}\,.
\end{align*}
Evaluated on a tensor $\alpha$ of virtual degree $a = \Vert \alpha \Vert$ this gives
\begin{align*}
\sum_{i_1, \ldots, i_m} \sum_{\delta_1,\ldots,\delta_m} Z^{\,\rightarrow}(|\alpha|,|\delta_1|,\ldots,|\delta_m|) \Big\{ \theta_{i_1} \partial_{t_{i_1}}(t^{\delta_1}) d_{\AA}^{\,(\delta_1)} \Big\} \cdots \Big\{ \theta_{i_m} \partial_{t_{i_m}}(t^{\delta_m}) d_{\AA}^{\,(\delta_m)} \Big\}(\alpha)
\end{align*}
where the scalar factor is computed by

\begin{definition}\label{defn:Z_factors} Given integers $a > 0$ and a sequence $d_1, \ldots, d_m > 0$ we define
\be\label{eq:defn_Z_factor_oriented}
Z^{\,\rightarrow}(a,d_1,\ldots,d_m) = \frac{1}{a + d_1} \frac{1}{a + d_1 + d_{2}} \cdots \frac{1}{a + d_1 + \cdots + d_m}
\ee
and a symmetrised version
\be
Z(a,d_1,\ldots,d_m) = \sum_{\sigma \in S_m} Z^{\,\rightarrow}(a,d_{\sigma 1},\ldots,d_{\sigma m})\,.
\ee
\end{definition}

In general there is no more to say, and the generic factors contributed by $\zeta$ to Feynman diagrams have the form given in \eqref{eq:defn_Z_factor_oriented} above. However, there is a useful special case:

\begin{lemma}\label{lemma:technical_antic} Let $\KK \subseteq \HH_{\AA'_\theta}$ be a subspace with the following properties
\begin{itemize}
\item[(a)] $\KK$ is closed under $d_{\AA}^{\,(\delta)}$ for every $\delta \neq \bold{0}$.
\item[(b)] As operators on $\KK$, we have $\big[ d_{\AA}^{\,(\delta)}, d_{\AA}^{\, (\gamma)} \big] = 0$ for all $\delta, \gamma \neq \bold{0}$.
\end{itemize}
Then for any $\alpha \in \KK$, $(\zeta \vAt_{\AA})^m(\alpha)$ is equal to
\be
\sum_{i_1 < \ldots < i_m} \sum_{\delta_1,\ldots,\delta_m} (-1)^{\binom{m}{2}} \theta_{i_1} \cdots \theta_{i_m} Z(|\alpha|,|\delta_{1}|,\ldots,|\delta_{m}|) \prod_{r=1}^m \Big\{ \partial_{t_{i_r}}(t^{\delta_{r}}) d_{\AA}^{\,(\delta_{r})} \Big\} (\alpha)\,.
\ee
\end{lemma}
\begin{proof}
Only sequences of distinct $\theta$'s contribute, so we find that $(\zeta \vAt_{\AA})^m(\alpha)$ is equal to
\begin{align*}
&\sum_{\substack{i_1,\ldots,i_m \\ \delta_1,\ldots,\delta_m}} Z^{\,\rightarrow}(|\alpha|,|\delta_1|,\ldots,|\delta_m|) (-1)^{\binom{m}{2}} \theta_{i_1} \cdots \theta_{i_m} \prod_{r=1}^m \Big\{ \partial_{t_{i_r}}(t^{\delta_r}) d_{\AA}^{\,(\delta_r)} \Big\} (\alpha)\\
&= \sum_{\substack{i_1 < \ldots < i_m \\ \delta_1,\ldots,\delta_m}} \sum_{\sigma \in S_m} Z^{\,\rightarrow}(|\alpha|,|\delta_1|,\ldots,|\delta_m|) (-1)^{\binom{m}{2}} \theta_{i_{\sigma 1}} \cdots \theta_{i_{\sigma m}} \prod_{r=1}^m \Big\{ \partial_{t_{i_{\sigma r}}}(t^{\delta_r}) d_{\AA}^{\,(\delta_r)} \Big\} (\alpha)\\
&= \sum_{i_1 < \ldots < i_m} \theta_{\bold{i}} \sum_{\sigma \in S_m} (-1)^{\binom{m}{2} + |\sigma|} \sum_{\delta_1,\ldots,\delta_m} Z^{\,\rightarrow}(|\alpha|,|\delta_1|,\ldots,|\delta_m|)  \prod_{r=1}^m \Big\{ \partial_{t_{i_{\sigma r}}}(t^{\delta_r}) d_{\AA}^{\,(\delta_r)} \Big\} (\alpha)\,.
\end{align*}
Note that only sequences $\delta_1,\ldots,\delta_m \in \mathbb{N}^n$ with all $\delta_i \neq \bold{0}$ contribute to this sum, so by hypothesis all the operators $d_{\AA}^{\,(\delta_r)}$ involved anti-commute
\begin{align*}
&= \sum_{i_1 < \ldots < i_m} \theta_{\bold{i}} \sum_{\sigma \in S_m} (-1)^{\binom{m}{2} + |\sigma|} \sum_{\delta_1,\ldots,\delta_m} Z^{\,\rightarrow}(|\alpha|,|\delta_{\sigma 1}|,\ldots,|\delta_{\sigma m}|)  \prod_{r=1}^m \Big\{ \partial_{t_{i_{\sigma r}}}(t^{\delta_{\sigma r}}) d_{\AA}^{\,(\delta_{\sigma r})} \Big\} (\alpha)\\
&= \sum_{i_1 < \ldots < i_m} \theta_{\bold{i}} \sum_{\sigma \in S_m} (-1)^{\binom{m}{2} + |\sigma|} \sum_{\delta_1,\ldots,\delta_m} Z^{\,\rightarrow}(|\alpha|,|\delta_{\sigma 1}|,\ldots,|\delta_{\sigma m}|) (-1)^{|\sigma|} \prod_{r=1}^m \Big\{ \partial_{t_{i_r}}(t^{\delta_{r}}) d_{\AA}^{\,(\delta_{r})} \Big\} (\alpha)\\
&= \sum_{i_1 < \ldots < i_m} (-1)^{\binom{m}{2}} \theta_{\bold{i}}   \sum_{\delta_1,\ldots,\delta_m} \sum_{\sigma \in S_m} Z^{\,\rightarrow}(|\alpha|,|\delta_{\sigma 1}|,\ldots,|\delta_{\sigma m}|) \prod_{r=1}^m \Big\{ \partial_{t_{i_r}}(t^{\delta_{r}}) d_{\AA}^{\,(\delta_{r})} \Big\} (\alpha)\\
&= \sum_{i_1 < \ldots < i_m} (-1)^{\binom{m}{2}} \theta_{\bold{i}}   \sum_{\delta_1,\ldots,\delta_m} Z(|\alpha|,|\delta_{1}|,\ldots,|\delta_{m}|) \prod_{r=1}^m \Big\{ \partial_{t_{i_r}}(t^{\delta_{r}}) d_{\AA}^{\,(\delta_{r})} \Big\} (\alpha)
\end{align*}
as claimed.
\end{proof}

The lemma is sometimes useful in reducing the number of Feynman diagrams that one has to actually calculate, see Remark \ref{remark:symmetry_A_type}. While the hypotheses of Lemma \ref{lemma:technical_antic} are technical, in the typical cases they are easy to check:

\begin{example}\label{example:computing_min_kstab} Suppose that $X$ is a Koszul matrix factorisation as in \eqref{eq:defn_X_tilde} and that we use the isomorphism of Lemma \ref{lemma:iso_rho} to identify $\HH_{\AA'_\theta}(X,X)$ with
\[
\bigwedge \big( F_\theta \oplus F_\xi \oplus F_{\Bar{\xi}} \big) \otimes R/I \otimes k\llbracket \bold{t} \rrbracket
\]
on which space by Lemma \ref{lemma:transfer_rho} we have
\[
d_{\AA}^{\,(\delta)} = \sum_{i=1}^r f_i^{\,(\delta)} \xi_i^* + \sum_{i=1}^r g_i^{\,(\delta)} \Bar{\xi}_i^*
\]
for some operators $f_i^{\,(\delta)},g_i^{\,(\delta)}$ on $R/I$ computed by Lemma \ref{lemma:rsharp_explicit}. Let $\KK$ be the subspace
\[
\bigwedge \big( F_\theta \oplus F_{\Bar{\xi}} ) \otimes R/I \otimes k\llbracket \bold{t} \rrbracket
\]
with no $\xi_i$'s, then as an operator on $\KK$
\[
d_{\AA}^{\,(\delta)}\Big|_{\KK} = \sum_{i=1}^r g_i^{\,(\delta)} \Bar{\xi}_i^*\,.
\]
These operators will all pair-wise anticommute, provided that $[ g_i^{\,(\delta)}, g_j^{\,(\varepsilon)}] = 0$ as operators on $R/I$ for all $1 \le i,j \le r$ and $\delta, \epsilon \neq \bold{0}$. This is true trivially when $R/I = k$, which means that the previous Lemma applies to calculating $(\zeta \vAt_{\AA})^m$ everywhere in Feynman diagrams computing the minimal model of $\AA(k^{\stab},k^{\stab})$, see Section \ref{section:generator}.
\end{example}

\begin{remark} Scalar factors like $Z$ occur in the context of infrared divergences involving soft virtual particles (such as soft virtual photons in quantum electrodynamics) see for instance \cite[Ch. 13]{weinberg} and \cite[p.204]{ps}. The operator $\zeta$ is part of a propagator \cite[\S 4.1.3]{lazaroiu_sft} which like $\frac{1}{p^2 - m^2 + i \varepsilon}$ in QFT has the effect generically of suppressing contributions from terms far off the mass-shell (the further off the mass-shell you are, the larger $p^2-m^2$ is). In our case, Feynman diagrams with large numbers of internal virtual particle lines ($\theta$ and $t$ lines) are suppressed with respect to the usual metric on $\mathbb{Q}$. 
\end{remark}

The most commonly treated case of the soft amplitudes in textbooks is the case of an on-shell external electron line, which corresponds to taking $a = 0$. In this case there is a simple formula for $Z$, which is easily proved by induction:

\begin{lemma} Given a sequence $d_1,\ldots,d_m > 0$ of integers,
\be
Z(0,d_1,\ldots,d_m) = \sum_{\sigma \in S_m} \frac{1}{d_{\sigma 1}} \frac{1}{d_{\sigma 1} + d_{\sigma 2}} \cdots \frac{1}{d_{\sigma 1} + \cdots + d_{\sigma m}} = \frac{1}{d_1 \cdots d_m}\,.
\ee
\end{lemma}

We do not know any simple formula for $Z$ in general. 

\section{Feynman diagrams}\label{section:feynman_diagram}

In quantum field theory, the calculus of Feynman diagrams provides algorithms for computing scattering amplitudes (with some caveats) and reasoning about physical processes. The role of Feynman diagrams in the theory of $A_\infty$-categories is similar: they provide an algorithmic method for computing the higher $A_\infty$-products on $\BB$ as well as a set of tools for reasoning about these products. The connection between $A_\infty$-structures, homological perturbation and Feynman diagrams is well-known; see \cite{lazaroiu_sft,lazaroiu_roiban}, \cite[p.42]{lazaroiu} and \cite[\S 2.5]{gwilliam}. However, in this context nontrivial examples with fully explicit Feynman rules accounting for all signs and symmetry factors and \emph{actual diagrams} like Figure \ref{fig:feynman_1} below, are rare. For background in the physics of Feynman diagrams we recommend \cite[Ch. 6]{weinberg}, \cite[\S 4.4]{ps} and for a more mathematical treatment \cite{qftstring}. 

The presentation of $A_\infty$-products in terms of Feynman diagrams is most useful when the objects of $\AA$ are matrix factorisations of Koszul type, and so we will focus on this case below. Throughout we adopt the hypotheses of Setup \ref{setup:overall}, and we write
\be\label{eq:hhalt_koszul}
\HH = \HH_{\AA'_\theta}\,, \qquad \HH(X,Y) = \bigwedge F_\theta \otimes R/I \otimes \Hom_k(\widetilde{X},\widetilde{Y}) \otimes k\llbracket \bold{t} \rrbracket\,.
\ee
Our aim is give a diagrammatic interpretation of the operators $\rho_T$, as given for example in \eqref{eq:explicit_tree_operator}. Implicitly $\rho_T$ consists of many summands, obtained by expanding the $e^{\delta}, e^{-\delta}$ and $\sigma_\infty, \phi_\infty$ operators. Among the summands generated from \eqref{eq:explicit_tree_operator} is for example
\be\label{eq:contrib_feynman}
\pi \delta^2 r_2\Big( (\zeta \vAt_{\AA})^2 \sigma \otimes \delta^3 (\zeta \vAt_{\AA})^3 \zeta \nabla \delta r_2\Big( \delta^5 (\zeta \vAt_{\AA})^6 \sigma \otimes (\zeta \vAt_{\AA}) \sigma \Big) \Big)\,.
\ee
The aim is to
\begin{itemize}
\item represent the space $\HH$ on which these operators act as a tensor product of exterior algebras and (completed) symmetric algebras, and
\item represent the operators as polynomials in creation and annihilation operators (that is, as multiplication with, or the derivative with respect to, even or odd generators of the relevant algebras).
\end{itemize}
Once this is done we can represent the operator \eqref{eq:contrib_feynman} as the contraction of a set of polynomials in creation and annihilation operators, with the pattern of contractions dictated by the structure of the original tree. The process of \emph{reducing this contraction to normal form} (with all annihilation operators on the right, and creation operators on the left) involves commuting creation and annihilation operators past one another, and their commutation relations generate many new terms. Feynman diagrams provide a calculus for organising these terms, and thus computing the normal form. There are three classes of operators making up \eqref{eq:contrib_feynman} which need to be given a diagrammatic interpretation:
\begin{itemize}
\item In Section \ref{section:koszul_mf} we represent $\HH$ as suggested above.
\item In Section \ref{section:fenyman_diagram_1} we treat $\vAt_{\AA}, \delta$.
\item In Section \ref{section:fenyman_diagram_2} we treat $\zeta$.
\item In Section \ref{section:feynman_diagram_3} we treat $\mu_2$.
\end{itemize}
Finally, in Section \ref{section:feynman_diagram_4} we give the Feynman rules and explain the whole process of computing with Feynman diagrams in an example.

\subsection{Koszul matrix factorisations}\label{section:koszul_mf}

Our Feynman diagrams will have vertices representing certain operators on $\HH(X,Y)$ for a pair of matrix factorisations $X,Y$ of $W$ of Koszul type. This means that we suppose given collections of polynomials $\{ f_i, g_i \}_{i=1}^r$ and $\{ u_j, v_j \}_{j=1}^s$ in $R$ satisfying
\[
W = \sum_{i=1}^r f_i g_i = \sum_{j=1}^s u_j v_j\,.
\]
To these polynomials we may associate matrix factorisations $X,Y$ defined as follows: we take odd generators $\xi_1,\ldots,\xi_r,\eta_1,\ldots,\eta_s$, set $F_\xi = \bigoplus_{i=1}^r k \xi_i, F_\eta = \bigoplus_{j=1}^s k \eta_j$ and
\begin{align*}
\widetilde{X} &= \bigwedge F_\xi = \bigwedge( k \xi_1 \oplus \cdots \oplus k \xi_r )\\
\widetilde{Y} &= \bigwedge F_\eta = \bigwedge( k \eta_1 \oplus \cdots \oplus k \eta_s )
\end{align*}
and then define
\begin{align}
X &= \big( \widetilde{X} \otimes R, \sum_{i=1}^r f_i \xi_i^* + \sum_{i=1}^r g_i \xi_i \big)\,,\label{eq:defn_X_tilde}\\
Y &= \big( \widetilde{Y} \otimes R, \sum_{j=1}^s u_j \eta_j^* + \sum_{j=1}^s v_j \eta_j \big)\,.\label{eq:defn_Y_tilde}
\end{align}
We ultimately want to give a graphical representation of operators on the $k$-module \eqref{eq:hhalt_koszul}, for which relevant operators are polynomials in creation and annihilation operators. It is therefore convenient to rewrite $\Hom_k( \widetilde{X}, \widetilde{Y} )$ in the form of an exterior algebra.

\begin{lemma}\label{lemma:iso_nu} There is an isomorphism of $\nZ_2$-graded $k$-modules
\[
\nu: \Hom_k( \widetilde{X}, \widetilde{Y} ) \lto \bigwedge F_\eta \otimes \bigwedge F_\xi^* 
\]
defined by
\[
\nu( \phi ) = \sum_{p \ge 0} \sum_{i_1 < \cdots < i_p} (-1)^{\binom{p}{2}} \phi( \xi_{i_1} \cdots \xi_{i_p} ) \xi_{i_1}^* \cdots \xi_{i_p}^*\,.
\]
\end{lemma}

The contraction operator which removes $\xi^*_i$ from a wedge product in $\bigwedge F_\xi^*$ can be written $(\xi_i^*)^* \lrcorner (-)$ or $(\xi_i^*)^*$ for short, but this is awkward. Even worse, the operation of wedge product $\xi_i^* \wedge (-)$ in this exterior algebra cannot be safely abbreviated to $\xi_i^*$ because some of our formulas will involve precisely the same notation to denote the contraction operator on $\bigwedge F_\xi$. So we introduce the following notational convention:

\begin{definition}\label{defn:bar_convention} We write $\Bar{\xi}_i$ for $\xi_i^*$ and $F_{\Bar{\xi}} = F_\xi^*$ so that as operators on $\bigwedge F_\xi^*$ we have
\[
\Bar{\xi}_i = \xi_i^* \wedge (-)\,, \qquad \Bar{\xi}_i^* = (\xi_i^*)^* \lrcorner (-)\,.
\]
The same conventions apply to $\eta$ and any other odd generators.
\end{definition}

Using $\nu$ we may identify $\HH(X,Y)$ as a $\nZ_2$-graded $k$-module with
\be\label{eq:presentationHHalt}
\bigwedge \big( F_\theta \oplus F_\eta \oplus F_{\Bar{\xi}} \big) \otimes R/I \otimes k\llbracket \bold{t} \rrbracket
\ee
which we view as the tensor product of a $k$-module of coefficients $R/I$ with the (completed) bosonic Fock space $k\llbracket \bold{t} \rrbracket$ with creation and annihilation operators $t_i, \partial_{t_i}$ and fermionic Fock spaces $\bigwedge F_{\Bar{\xi}}, \bigwedge F_\eta, \bigwedge F_\theta$ with creation operators $\Bar{\xi}_i, \eta_j, \theta_k$ and annihilation operators $\Bar{\xi}_i^*, \eta_j^*, \theta_k^*$ respectively. 

\subsection{Diagrams for $\vAt_{\AA}, \delta$}\label{section:fenyman_diagram_1}

In the notation of the previous section we now elaborate on the explicit formulas for $\vAt_{\AA},\delta$ in terms of creation and annihilation operators.

\begin{lemma}\label{lemma:transfer} The operator $(*)$ which makes the diagram
\[
\xymatrix@C+3pc@R+2pc{
\Hom_k( \widetilde{X}, \widetilde{Y} ) \otimes R \ar[d]_-{d_{\AA}} \ar[r]_-{\cong}^-{ \nu \otimes 1} & \bigwedge F_\eta \otimes \bigwedge F_{\Bar{\xi}} \otimes R \ar[d]^-{(*)}\\
\Hom_k( \widetilde{X}, \widetilde{Y} ) \otimes R \ar[r]^-{\cong}_-{ \nu \otimes 1 } & \bigwedge F_\eta \otimes \bigwedge F_{\Bar{\xi}} \otimes R
}
\]
is given by the formula
\[
\sum_{j=1}^s u_j \eta_j^* + \sum_{j=1}^s v_j \eta_j - \sum_{i=1}^r f_i \Bar{\xi}_i + \sum_{i=1}^r g_i \Bar{\xi}_i^*\,.
\]
\end{lemma}
\begin{proof}
By direct calculation.
\end{proof}

With this notation, the operator $\vAt_{\AA}$ on $\HH(X,Y)$ corresponds to an operator on \eqref{eq:presentationHHalt} given by the following formula (we use the superscript $\nu$ to record that this isomorphism is used to transfer $\vAt_{\AA}$) 
\begin{align}
\vAt^{\nu}_{\AA} &= [ d_{\AA}, \nabla ] = \sum_{k=1}^n \theta_k [ \partial_{t_k}, d_{\AA} ]\nonumber\\
&= \sum_{k=1}^n \sum_{j=1}^s \theta_k \partial_{t_k} (u_j) \eta_j^* + \sum_{k=1}^n \sum_{j=1}^s \theta_k \partial_{t_k}( v_j ) \eta_j\label{eq:vat_formula}\\
& \qquad - \sum_{k=1}^n \sum_{i=1}^r \theta_k \partial_{t_k}(f_i) \Bar{\xi}_i + \sum_{k=1}^n \sum_{i=1}^r \theta_k \partial_{t_k}(g_i) \Bar{\xi}_i^*\nonumber
\end{align}
where for an element $r \in R$ what we mean by $\partial_{t_k}(r)$ is the $k$-linear operator on $R/I \otimes k\llbracket \bold{t} \rrbracket$ which is the commutator of $\partial_{t_k}$ with the operator $r^{\#}$ of Definition \ref{defn:rsharp}. By Lemma \ref{lemma:rsharp_explicit} we can write this explicitly in terms of the multiplication tensor $\Gamma$ of Definition \ref{defn_gamma} as
\be
\partial_{t_k}(r)(z_h \otimes t^{\tau}) = \sum_{l=1}^\mu \sum_{\delta \in \mathbb{N}^n} \Big[ \sum_{\alpha + \beta = \delta } \sum_{m=1}^\mu r_{(m,\alpha)} \Gamma^{mh}_{l\beta} \Big] z_l \otimes \partial_{t_k}(t^\delta) t^{\tau}
\ee
where the coefficients $r_{(m,\alpha)} \in k$ are as in Definition \ref{defn:rsharp}, $\{ z_h \}_{h=1}^\mu$ is our chosen $k$-basis of $R/I$ and $\Gamma$ is the multiplication tensor. Reading $\partial_{t_k}(t^\delta)$ as the operator of left multiplication by this monomial, we may write
\be\label{eq:partial_derivative_op}
\partial_{t_k}(r) = \sum_{h=1}^\mu \sum_{l=1}^\mu \sum_{\delta \in \mathbb{N}^n} \Big[ \sum_{\alpha + \beta = \delta } \sum_{m=1}^\mu r_{(m,\alpha)} \Gamma^{mh}_{l\beta} \Big] z_l \partial_{t_k}(t^\delta) z_h^*\,.
\ee
Next we describe the operators $\delta = \sum_{k=1}^n \lambda^\bullet_k \theta_k^*$ and for this we need to choose a particular homotopy $\lambda^Y_k: Y \lto Y$ with $[ \lambda^Y_k, d_Y ] = t_k \cdot 1_Y$. Recall that $t_1,\ldots,t_n$ is a quasi-regular sequence satisfying some hypotheses satisfied in particular by the partial derivatives of the potential $W$, but other choices are possible. We assume here that our homotopies $\lambda^Y_k$ are chosen to be of the form
\be\label{eq:formula_lambdak_koszul}
\lambda^Y_k = \sum_{j=1}^s F_{kj} \eta_j^* + \sum_{j=1}^s G_{kj} \eta_j
\ee
for polynomials $\{ F_{kj}, G_{kj} \}_{j=1}^s$ in $R$ which satisfy the equations
\[
\sum_{j=1}^s( F_{kj} v_j + G_{kj} u_j ) = t_k \qquad 1 \le k \le n\,.
\]

\begin{remark}\label{remark:default_homotopies}
If $\bold{t} = (\partial_{x_1} W, \ldots, \partial_{x_n} W)$ then $F_{kj} = \partial_{x_k}(u_j), G_{kj} = \partial_{x_k}(v_j)$ satisfy these requirements and hence define a valid sequence of homotopies $\lambda^Y_1,\ldots,\lambda^Y_n$.
\end{remark}

The operator $\lambda_k^\bullet$ of Definition \ref{defn:important_operators} acts by post-composition with $\lambda^Y_k$, and so the corresponding operator on \eqref{eq:presentationHHalt} under $\nu$ is by the same calculation as Lemma \ref{lemma:transfer} given by the formula \eqref{eq:formula_lambdak_koszul}. In this notation (again using a superscript $\nu$ to indicate the transfer)
\be\label{eq:delta_formula}
\delta^\nu = \sum_{k=1}^n \lambda^\bullet_k \theta_k^* = \sum_{k=1}^n \sum_{j=1}^s F_{kj} \eta_j^* \theta_k^* + \sum_{k=1}^n \sum_{j=1}^s G_{kj} \eta_j \theta_k^*\,.
\ee
Combining \eqref{eq:vat_formula} and \eqref{eq:partial_derivative_op} yields:

\begin{lemma}
The critical Atiyah class $\vAt_{\AA}$ may be presented using $\nu$ as an operator on \eqref{eq:presentationHHalt} given by the sum of the four terms given below, each of which is itself summed over the indices $1 \le h,l \le \mu, 1 \le k \le n$ and $\delta \in \mathbb{N}^n$:
\begin{align}
& \sum_{j=1}^s \Big[ \sum_{\alpha + \beta = \delta } \sum_{m=1}^\mu (u_j)_{(m,\alpha)} \Gamma^{m h}_{l \beta} \Big] \theta_k z_{l} \partial_{t_k}(t^\delta) z_h^* \eta_j^*  \qquad (\textup{A}.1)^{\nu} \label{eq:a1vertex}\\
& \sum_{j=1}^s \Big[ \sum_{\alpha + \beta = \delta } \sum_{m=1}^\mu (v_j)_{(m,\alpha)} \Gamma^{mh}_{l\beta} \Big] \theta_k z_l \partial_{t_k}(t^\delta) \eta_j z_h^*  \qquad (\textup{A}.2)^{\nu} \label{eq:a2vertex}\\
-&\sum_{i=1}^r \Big[ \sum_{\alpha + \beta = \delta } \sum_{m=1}^\mu (f_i)_{(m,\alpha)} \Gamma^{mh}_{l\beta} \Big] \theta_k z_l \partial_{t_k}(t^\delta) \Bar{\xi}_i z_h^*  \qquad (\textup{A}.3)^{\nu} \label{eq:a3vertex}\\
&\sum_{i=1}^r \Big[ \sum_{\alpha + \beta = \delta } \sum_{m=1}^\mu (g_i)_{(m,\alpha)} \Gamma^{mh}_{l\beta} \Big] \theta_k z_l \partial_{t_k}(t^\delta) z_h^* \Bar{\xi}_i^* \qquad (\textup{A}.4)^{\nu}\,. \label{eq:a4vertex}
\end{align}
Schematically, we can write
\be
\vAt^\nu_{\AA} = \sum_{h,l=1}^\mu \sum_{k=1}^n\sum_{\delta \in \mathbb{N}^n} \Big[ (\textup{A}.1)^{\nu} + (\textup{A}.2)^{\nu} + (\textup{A}.3)^{\nu} + (\textup{A}.4)^{\nu} \Big]\,.
\ee
\end{lemma}

Combining \eqref{eq:delta_formula} and \eqref{eq:partial_derivative_op} yields:

\begin{lemma}
The operator $\delta$ may be presented using $\nu$ as an operator on \eqref{eq:presentationHHalt} given by the sum of the two terms given below, each of which is itself summed over the indices $1 \le h,l \le \mu, 1 \le k \le n$ and $\delta \in \mathbb{N}^n$:
\begin{align}
& \sum_{j=1}^s \Big[ \sum_{\alpha + \beta = \delta } \sum_{m=1}^\mu (F_{kj})_{(m,\alpha)} \Gamma^{m h}_{l \beta} \Big] z_{l} t^\delta \eta_j^* z_h^* \theta_k^* \qquad (\textup{C}.1)^{\nu}\label{eq:c1vertex}\\
& \sum_{j=1}^s \Big[ \sum_{\alpha + \beta = \delta } \sum_{m=1}^\mu (G_{kj})_{(m,\alpha)} \Gamma^{m h}_{l \beta} \Big] z_{l} t^\delta \eta_j z_h^* \theta_k^* \qquad (\textup{C}.2)^{\nu}\,.\label{eq:c2vertex}
\end{align}
Schematically, we can write
\be
\delta^\nu = \sum_{h,l=1}^\mu \sum_{k=1}^n\sum_{\delta \in \mathbb{N}^n} \Big[ (\textup{C}.1)^{\nu} + (\textup{C}.2)^{\nu} \Big]\,.
\ee
\end{lemma}

Each of these monomials in creation and annihilation operators is associated with its own type of \emph{interaction vertex} in our Feynman diagrams. Eventually these vertices will be drawn on the same trees used to define the $A_\infty$-products $\rho_k$ and they will be given a formal interpretation by the Feynman rules, but for the moment they are just pictures. In our description we tend to imagine time evolving from from top of the page (the input) to the bottom (the output). At an interaction vertex associated with a monomial, each annihilation operator becomes an \emph{incoming line} (entering the vertex from above) and each creation operator becomes an \emph{outgoing line} (leaving the vertex downward). Different line styles are used to distinguish creation and annihilation operators of different ``types''.

\begin{center}
\begin{tabular}{ >{\centering}m{8cm} >{\centering}m{6cm} }
\[
\xymatrix@C+2pc@R+3pc{
& \ar@{.}[d]^-{h} & \ar[dl]^-{\eta_j}\\
& \bullet \ar@{=}[dl]^-{\theta_k} \ar@{.}[d]^-{l} \ar@{~}[dr]^-{\partial_{t_k}(t^\delta)}\\
& &
}
\]
&
\textbf{(A.1)${}^\nu$}
\vspace{1cm}
\[\sum_{\alpha + \beta = \delta } \sum_{m=1}^\mu (u_j)_{(m,\alpha)} \Gamma^{m h}_{l \beta}\]
\vspace{0.5cm}
$\theta_k z_{l} \partial_{t_k}(t^\delta) z_h^* \eta_j^*$
\end{tabular}
\end{center}

\begin{center}
\begin{tabular}{ >{\centering}m{8cm} >{\centering}m{6cm} }
\[
\xymatrix@C+1pc@R+3pc{
& & \ar@{.}[d]^-{h}\\
& & \bullet \ar@{=}[dll]_-{\theta_k} \ar@{.}[dl]^-{l} \ar@{~}[dr]_-{\partial_{t_k}(t^\delta)} \ar[drr]^-{\eta_j}\\
& & & &
}
\]
&
\textbf{(A.2)${}^\nu$}
\vspace{1cm}
\[ \sum_{\alpha + \beta = \delta } \sum_{m=1}^\mu (v_j)_{(m,\alpha)} \Gamma^{mh}_{l\beta}\]
\vspace{0.5cm}
$\theta_k z_l \partial_{t_k}(t^\delta) \eta_j z_h^*$
\end{tabular}
\end{center}

\begin{center}
\begin{tabular}{ >{\centering}m{8cm} >{\centering}m{6cm} }
\[
\xymatrix@C+1pc@R+3pc{
& & \ar@{.}[d]^-{h}\\
& & \bullet \ar@{=}[dll]_-{\theta_k} \ar@{.}[dl]^-{l} \ar@{~}[dr]_-{\partial_{t_k}(t^\delta)}\\
& & & & \ar[ull]_-{\xi_i}
}
\]
&
\textbf{(A.3)${}^\nu$}
\vspace{1cm}
\[- \sum_{\alpha + \beta = \delta } \sum_{m=1}^\mu (f_i)_{(m,\alpha)} \Gamma^{mh}_{l\beta}\]
\vspace{0.5cm}
$\theta_k z_l \partial_{t_k}(t^\delta) \Bar{\xi}_i z_h^*$
\end{tabular}
\end{center}

\begin{center}
\begin{tabular}{ >{\centering}m{8cm} >{\centering}m{6cm} }
\[
\xymatrix@C+2pc@R+3pc{
& \ar@{.}[d]^-{h} &\\
& \bullet \ar[ur]_-{\xi_i} \ar@{=}[dl]^-{\theta_k} \ar@{.}[d]^-{l} \ar@{~}[dr]^-{\partial_{t_k}(t^\delta)}\\
& &
}
\]
&
\textbf{(A.4)${}^\nu$}
\vspace{1cm}
\[ \sum_{\alpha + \beta = \delta } \sum_{m=1}^\mu (g_i)_{(m,\alpha)} \Gamma^{mh}_{l\beta} \]
\vspace{0.5cm}
$\theta_k z_l \partial_{t_k}(t^\delta) z_h^* \Bar{\xi}_i^*$
\end{tabular}
\end{center}

\begin{center}
\begin{tabular}{ >{\centering}m{8cm} >{\centering}m{6cm} }
\[
\xymatrix@R+3pc{
\ar@{~}[d]^-{t_k}\\
\bullet \ar@{=}[d]^-{\theta_k}\\
\;
}
\]
&
\textbf{(B)}
\vspace{1cm}
\[\theta_k \partial_{t_k}\]
\end{tabular}
\end{center}

\begin{center}
\begin{tabular}{ >{\centering}m{8cm} >{\centering}m{6cm} }
\[
\xymatrix@C+2pc@R+3pc{
\ar[dr]_-{\eta_j} & \ar@{.}[d]^-{h} & \ar@{=}[dl]^-{\theta_k}\\
& \bullet \ar@{.}[d]^-{l} \ar@{~}[dl]^-{t^\delta}\\
& &
}
\]
&
\textbf{(C.1)${}^\nu$}
\vspace{1cm}
\[\sum_{\alpha + \beta = \delta } \sum_{m=1}^\mu (F_{kj})_{(m,\alpha)} \Gamma^{m h}_{l \beta}\]
\vspace{0.5cm}
$z_{l} t^\delta \eta_j^* z_h^* \theta_k^*$
\end{tabular}
\end{center}

\begin{center}
\begin{tabular}{ >{\centering}m{8cm} >{\centering}m{6cm} }
\[
\xymatrix@C+2pc@R+3pc{
& \ar@{.}[d]^-{h} & \ar@{=}[dl]^-{\theta_k} \\
& \bullet \ar[dl]^-{\eta_j}\ar@{.}[d]^-{l} \ar@{~}[dr]^-{t^\delta}\\
& &
}
\]
&
\textbf{(C.2)${}^\nu$}
\vspace{1cm}
\[\sum_{\alpha + \beta = \delta } \sum_{m=1}^\mu (G_{kj})_{(m,\alpha)} \Gamma^{m h}_{l \beta}\]
\vspace{0.5cm}
$z_{l} t^\delta \eta_j z_h^* \theta_k^*$
\end{tabular}
\end{center}

We refer to the vertices \eqref{eq:a1vertex}-\eqref{eq:a4vertex} arising from the Atiyah class respectively as \emph{A-type vertices} (A.1, A.2, A.3, A.4) and the vertices \eqref{eq:c1vertex},\eqref{eq:c2vertex} arising from $\delta$ as \emph{C-type vertices} (C.1, C.2). There are also \emph{B-type vertices} given by the operator $\nabla = \sum_{k=1}^n \theta_k \partial_{t_k}$.

\begin{remark} Some remarks on these diagrams:
\begin{itemize} 

\item We do not think of the vectors in $\HH$ as states of literal particles (this word is generally reserved for state spaces transforming as representations of the inhomogeneous Lorentz group \cite{weinberg}) but the physics terminology is convenient and we sometimes refer to \emph{bosons} (the $t_i$) and \emph{fermions} (the $\xi_i, \eta_j, \theta_k$). Another useful concept is that of \emph{virtual particles} which is a term used to refer to lines propagating in the interior of Feynman diagrams. In our diagrams this role is played by the bosons $t_i$ and fermions $\theta_i$ (hence the virtual degree of Definition \ref{definition:zeta}) which represent the degrees of freedom that are being ``integrated out'' by the process of computing the $A_\infty$-products. 

 
\item Following standard conventions bosons (commuting generators) are denoted by wiggly or dashed lines, and fermions (anticommuting generators) by solid lines \cite[\S 4.7]{ps} (perhaps doubled). For simplicity we distinguish the $\xi$ and $\eta$ lines only by their labels and we write $h$ for a line labelled $z_h$. Strictly speaking a squiggly line labelled $t^\delta$ should be interpreted as $\delta_i$ lines labelled $t_i$ for $1 \le i \le n$. We use the orientation on a fermion line to determine whether it should be read as a creation or annihilation operator for $\xi$ (downward) or for $\Bar{\xi}$ (upward).

\item Each vertex above actually represents a family indexed by possible choices of indices. If we wish to speak about a specific instance we use subscripts, for example $(\textup{B})_{k=2}$ is the interaction vertex with an incoming $t_2$ and outgoing $\theta_2$.\footnote{Depending on the matrix factorisations involved, some families of A or C-type interaction vertex may be \emph{infinite}. For example, if $(f_i)_{(m, \alpha)}$ is nonzero for infinitely many $\alpha$ there may be infinitely many $(A.3)^{\nu}$ vertices with nonzero coefficients. However only finitely many distinct types of interaction vertices can contribute for any particular tree $T$. If $T$ has $k$ leaves there are $k - 2$ internal edges and hence $k - 2$ occurrences of $\nabla$ in the associated operator. In terms of Feynman diagrams, that means there are precisely $k - 2$ B-type interactions. Since each $t_i$ that is generated in a Feynman diagram must eventually annihilate with a $\partial_{t_i}$ at a B-type interaction vertex, and each interaction vertex with indices $\alpha, \beta$ generates either $|\delta|$ or $|\delta| - 1$ copies of the $t_i$'s, only coefficients with $|\alpha|, |\beta| \le k - 2$ contribute. In short, larger trees can support more ``virtual bosons'' and more complex interactions.}
\end{itemize}
\end{remark}

\subsection{Alternative isomorphism $\rho$}\label{section:altrho}

We have used the isomorphism $\nu$ to present operators on $\HH(X,Y)$ as creation and annihilation operators. In the case $X = Y$ there is an alternative isomorphism $\rho$ which leads to less interaction vertices.

\begin{lemma}\label{lemma:iso_rho} There is an isomorphism of $\nZ_2$-graded $k$-modules
\begin{gather*}
\rho: \bigwedge F_\xi \otimes \bigwedge F_{\Bar{\xi}} \lto \End_k\big( \bigwedge F_\xi \big)\\
\rho\big( \xi_{i_1} \wedge \cdots \wedge \xi_{i_a} \otimes \Bar{\xi}_{j_1} \wedge \cdots \wedge \Bar{\xi}_{j_b} \big) = \xi_{i_1} \circ \cdots \circ \xi_{i_a} \circ \xi_{j_1}^* \circ \cdots \circ \xi_{j_b}^*
\end{gather*}
where on the right hand $\xi_i, \xi_j^*$ denote the usual operators $\xi_i = \xi_i \wedge (-)$ and $\xi_j^* = \xi_j^* \lrcorner (-)$.
\end{lemma}

\begin{remark}\label{remark_rhoisoalg} Let $C$ be the $\nZ_2$-graded algebra generated by odd $\xi_i, \Bar{\xi}_i$ for $1 \le i \le r$ subject to the relations $[ \xi_i, \xi_j ] = [ \Bar{\xi}_i, \Bar{\xi}_j ] = 0$ and $[ \xi_i, \Bar{\xi}_j ] = \delta_{ij}$. We denote multiplication in this Clifford algebra by $\bullet$. There is an isomorphism of $\nZ_2$-graded $k$-modules
\begin{gather*}
\bigwedge F_\xi \otimes \bigwedge F_{\Bar{\xi}} \lto C\\
\xi_{i_1} \wedge \cdots \wedge \xi_{i_a} \otimes \Bar{\xi}_{j_1} \wedge \cdots \wedge \Bar{\xi}_{j_b} \mapsto \xi_{i_1} \bullet \cdots \bullet \xi_{i_a} \bullet \Bar{\xi}_{j_1} \bullet \cdots \bullet \Bar{\xi}_{j_b}\,.
\end{gather*}
Making this identification, $\rho$ is the linear map underlying an isomorphism of the Clifford algebra $C$ with the endomorphism algebra of $\bigwedge F_\xi$. In particular, the diagram
\be
\xymatrix@C+2pc{
\big( \bigwedge F_\xi \otimes \bigwedge F_{\Bar{\xi}} \big) \otimes \big( \bigwedge F_\xi \otimes \bigwedge F_{\Bar{\xi}} \big) \ar[r]^-{\rho \otimes \rho}\ar[d]_-\bullet & \End_k\big( \bigwedge F_\xi\big) \otimes \End_k\big( \bigwedge F_\xi \big) \ar[d]^-{- \circ -}\\
\bigwedge F_\xi \otimes \bigwedge F_{\Bar{\xi}} \ar[r]_-\rho & \End_k\big( \bigwedge F_\xi \big)
}
\ee
commutes, where on the left $\bullet$ is the multiplication in the Clifford algebra.
\end{remark}

\begin{lemma}\label{lemma:commutators_on_rho} The diagrams
\[
\xymatrix@C+2pc{
\bigwedge F_\xi \otimes \bigwedge F_{\Bar{\xi}} \ar[d]_-{\Bar{\xi}_i^*} \ar[r]^-{\rho} & \End_k\big( \bigwedge F_\xi \big) \ar[d]^-{[\xi_i, -]}\\
\bigwedge F_\xi \otimes \bigwedge F_{\Bar{\xi}} \ar[r]_-{\rho} & \End_k\big( \bigwedge F_\xi \big)
}
\qquad
\xymatrix@C+2pc{
\bigwedge F_\xi \otimes \bigwedge F_{\Bar{\xi}} \ar[d]_-{\xi_i^*} \ar[r]^-{\rho} & \End_k\big( \bigwedge F_\xi \big) \ar[d]^-{[ \xi_i^*, - ]}\\
\bigwedge F_\xi \otimes \bigwedge F_{\Bar{\xi}} \ar[r]_-{\rho} & \End_k\big( \bigwedge F_\xi \big)
}
\]
commute, where graded commutators are in the algebra $\End_k(\bigwedge F_\xi)$.
\end{lemma}
\begin{proof}
By direct calculation.
\end{proof}

\begin{lemma}\label{lemma:transfer_rho} The operator $(*)$ which makes the diagram
\[
\xymatrix@C+3pc@R+2pc{
\Hom_k( \widetilde{X}, \widetilde{X} ) \otimes R \ar[d]_-{d_{\AA}} \ar[r]_-{\cong}^-{ \rho \otimes 1} & \bigwedge F_\xi \otimes \bigwedge F_{\Bar{\xi}} \otimes R \ar[d]^-{(*)}\\
\Hom_k( \widetilde{X}, \widetilde{X} ) \otimes R \ar[r]^-{\cong}_-{ \rho \otimes 1 } & \bigwedge F_\xi \otimes \bigwedge F_{\Bar{\xi}} \otimes R
}
\]
is given by the formula
\[
\sum_{i=1}^r f_i \xi_i^* + \sum_{i=1}^r g_i \Bar{\xi}_i^*\,.
\]
\end{lemma}
\begin{proof}
By definition the differential is
\[
d_{\AA} = \sum_{i=1}^r f_i [ \xi_i^*, - ] + \sum_{i=1}^r g_i [ \xi_i, - ]
\]
so this is immediate from Lemma \ref{lemma:commutators_on_rho}.
\end{proof}

Using $\rho$ we may identify $\HH(X,X)$ as a $\nZ_2$-graded $k$-module with
\be\label{eq:presentationHHalt_rho}
\xymatrix@C+2pc{
\bigwedge \big( F_\theta \oplus F_\xi \oplus F_{\Bar{\xi}} \big) \otimes R/I \otimes k\llbracket \bold{t} \rrbracket
}\,.
\ee
With this notation, the operator $\vAt_{\AA}$ on $\HH(X,X)$ corresponds to
\begin{align*}
\vAt^\rho_{\AA} &= [ d_{\AA}, \nabla ] = \sum_{k=1}^n \theta_k [ \partial_{t_k}, d_{\AA} ]\\
&= \sum_{k=1}^n \sum_{i=1}^r \theta_k \partial_{t_k}(f_i) \xi_i^* + \sum_{k=1}^n \sum_{i=1}^r \theta_k \partial_{t_k}(g_i) \Bar{\xi}_i^*\,.
\end{align*}
With the same conventions about $\delta = \sum_{k=1}^n \lambda^\bullet_k \theta_k^*$ as above, we have
\be
\lambda^X_k = \sum_{i=1}^r F_{ki} \xi_i^* + \sum_{i=1}^r G_{ki} \xi_i\,.
\ee
The operator $\lambda_k^\bullet$ acts on $\End_k( \bigwedge F_\xi ) \otimes R$ by post-composition with $\lambda^X_k$, that is to say, by left multiplication in the endomorphism ring, and since $\rho$ is an isomorphism of algebras the corresponding operator on \eqref{eq:presentationHHalt_rho} using $\rho$ is
\be
\delta^\rho = \sum_{k=1}^n \lambda^\bullet_k \theta_k^* = \sum_{k=1}^n \sum_{i=1}^r F_{ki} [\Bar{\xi}_i \bullet (-)] \theta_k^* + \sum_{k=1}^n \sum_{i=1}^r G_{ki} [\xi_i \bullet (-)] \theta_k^*
\ee
where $\bullet$ means multiplication in the Clifford algebra structure on $\bigwedge F_\xi \otimes \bigwedge F_{\Bar{\xi}}$. It is easily checked that as operators on this tensor product, we have
\[
\xi_i \bullet (-) = \xi_i \otimes 1\,, \qquad \Bar{\xi}_i \bullet (-) = \xi_i^* \otimes 1 + 1 \otimes \Bar{\xi}_i
\]
with the usual convention that $\xi_i$ means $\xi_i \wedge (-)$ and $\xi_i^*$ means $\xi_i^* \lrcorner (-)$. So finally
\be
\delta = \sum_{k=1}^n \sum_{i=1}^r F_{ki} \xi_i^* \theta_k^* + \sum_{k=1}^n \sum_{i=1}^r F_{ki} \Bar{\xi}_i \theta_k^* + \sum_{k=1}^n \sum_{i=1}^r G_{ki} \xi_i \theta_k^*\,.
\ee

\begin{lemma}
The critical Atiyah class $\vAt_{\AA}$ may be presented using $\rho$ as an operator on \eqref{eq:presentationHHalt_rho}, given by the sum of the four terms below, each of which is itself summed over the indices $1 \le h,l \le \mu, 1 \le k \le n$ and $\delta \in \mathbb{N}^n$:
\begin{align}
&\sum_{i=1}^r \Big[ \sum_{\alpha + \beta = \delta } \sum_{m=1}^\mu (f_i)_{(m,\alpha)} \Gamma^{mh}_{l\beta} \Big] \theta_k z_l \partial_{t_k}(t^\delta) \xi_i^* z_h^*  \qquad (\textup{A}.1)^\rho \label{eq:a3vertex_rho}\\
&\sum_{i=1}^r \Big[ \sum_{\alpha + \beta = \delta } \sum_{m=1}^\mu (g_i)_{(m,\alpha)} \Gamma^{mh}_{l\beta} \Big] \theta_k z_l \partial_{t_k}(t^\delta) z_h^* \Bar{\xi}_i^* \qquad (\textup{A}.4)^\rho \label{eq:a4vertex_rho}
\end{align}
\end{lemma}

\begin{lemma}
The operator $\delta$ may be presented using $\rho$ as an operator on \eqref{eq:presentationHHalt_rho} given by the sum of the two terms given below, each of which is itself summed over the indices $1 \le h,l \le \mu, 1 \le k \le n$ and $\delta \in \mathbb{N}^n$:
\begin{align}
& \sum_{i=1}^r \Big[ \sum_{\alpha + \beta = \delta } \sum_{m=1}^\mu (F_{ki})_{(m,\alpha)} \Gamma^{m h}_{l \beta} \Big] z_{l} t^\delta \xi_i^* z_h^* \theta_k^* \qquad (\textup{C}.1)^\rho\label{eq:c1vertex_rho}\\
& \sum_{i=1}^r \Big[ \sum_{\alpha + \beta = \delta } \sum_{m=1}^\mu (G_{ki})_{(m,\alpha)} \Gamma^{m h}_{l \beta} \Big] z_{l} t^\delta \xi_i z_h^* \theta_k^* \qquad (\textup{C}.2)^\rho\label{eq:c2vertex_rho}\\
& \sum_{i=1}^r \Big[ \sum_{\alpha + \beta = \delta } \sum_{m=1}^\mu (F_{ki})_{(m,\alpha)} \Gamma^{m h}_{l \beta} \Big] z_{l} t^\delta \Bar{\xi}_i z_h^* \theta_k^* \qquad (\textup{C}.3)^\rho\label{eq:c3vertex_rho}
\end{align}
\end{lemma}

The B-type interaction is as before. The other interactions are

\begin{center}
\begin{tabular}{ >{\centering}m{8cm} >{\centering}m{6cm} }
\[
\xymatrix@C+2pc@R+3pc{
& & \ar@{.}[d]^-{z_h} & \ar[dl]^-{\xi_i}\\
& & \bullet \ar@{=}[dl]_-{\theta_k} \ar@{.}[d]^-{z_l} \ar@{~}[dr]^-{\partial_{t_k}(t^\delta)}\\
& & & & 
}
\]
&
\textbf{(A.1)${}^\rho$}
\vspace{1cm}
\[\sum_{\alpha + \beta = \delta } \sum_{m=1}^\mu (f_i)_{(m,\alpha)} \Gamma^{mh}_{l\beta}\]
\vspace{0.5cm}
$\theta_k z_l \partial_{t_k}(t^\delta) \xi_i^* z_h^*$
\end{tabular}
\end{center}

\begin{center}
\begin{tabular}{ >{\centering}m{8cm} >{\centering}m{6cm} }
\[
\xymatrix@C+2pc@R+3pc{
& \ar@{.}[d]^-{z_h} &\\
& \bullet \ar[ur]_-{\xi_i} \ar@{=}[dl]^-{\theta_k} \ar@{.}[d]^-{z_l} \ar@{~}[dr]^-{\partial_{t_k}(t^\delta)}\\
& &
}
\]
&
\textbf{(A.4)${}^\rho$}
\vspace{1cm}
\[ \sum_{\alpha + \beta = \delta } \sum_{m=1}^\mu (g_i)_{(m,\alpha)} \Gamma^{mh}_{l\beta} \]
\vspace{0.5cm}
$\theta_k z_l \partial_{t_k}(t^\delta) z_h^* \Bar{\xi}_i^*$
\end{tabular}
\end{center}

\begin{center}
\begin{tabular}{ >{\centering}m{8cm} >{\centering}m{6cm} }
\[
\xymatrix@C+2pc@R+3pc{
\ar[dr]_-{\xi_i} & \ar@{.}[d]^-{z_h} & \ar@{=}[dl]^-{\theta_k}\\
& \bullet \ar@{.}[d]^-{z_l} \ar@{~}[dl]^-{t^\delta}\\
& &
}
\]
&
\textbf{(C.1)${}^\rho$}
\vspace{1cm}
\[\sum_{\alpha + \beta = \delta } \sum_{m=1}^\mu (F_{ki})_{(m,\alpha)} \Gamma^{m h}_{l \beta}\]
\vspace{0.5cm}
$z_{l} t^\delta \xi_i^* z_h^* \theta_k^*$
\end{tabular}
\end{center}

\begin{center}
\begin{tabular}{ >{\centering}m{8cm} >{\centering}m{6cm} }
\[
\xymatrix@C+2pc@R+3pc{
& \ar@{.}[d]^-{z_h} & \ar@{=}[dl]^-{\theta_k} \\
& \bullet \ar[dl]^-{\xi_i}\ar@{.}[d]^-{z_l} \ar@{~}[dr]^-{t^\delta}\\
& &
}
\]
&
\textbf{(C.2)${}^\rho$}
\vspace{1cm}
\[\sum_{\alpha + \beta = \delta } \sum_{m=1}^\mu (G_{ki})_{(m,\alpha)} \Gamma^{m h}_{l \beta}\]
\vspace{0.5cm}
$z_{l} t^\delta \xi_i z_h^* \theta_k^*$
\end{tabular}
\end{center}

\begin{center}
\begin{tabular}{ >{\centering}m{8cm} >{\centering}m{6cm} }
\[
\xymatrix@C+2pc@R+3pc{
& \ar@{.}[d]^-{z_h} & \ar@{=}[dl]^-{\theta_k} \\
& \bullet \ar@{.}[d]^-{z_l} \ar@{~}[dr]^-{t^\delta}\\
\ar[ur]^-{\xi_i} & &
}
\]
&
\textbf{(C.3)${}^\rho$}
\vspace{1cm}
\[\sum_{\alpha + \beta = \delta } \sum_{m=1}^\mu (F_{ki})_{(m,\alpha)} \Gamma^{m h}_{l \beta}\]
\vspace{0.5cm}
$z_{l} t^\delta \Bar{\xi}_i z_h^* \theta_k^*$
\end{tabular}
\end{center}

\subsection{Diagrams for $\zeta$}\label{section:fenyman_diagram_2}

A power of $\zeta \vAt_{\AA}$ will contribute a sequence of A-type interactions together with the scalar factors analysed in Section \ref{section:propagator}. To put these factors in the context of Feynman diagrams, consider a power $(\zeta \vAt)^m$ contributing $m$ A-type vertices, each of which emits a $\theta$, some monomial $t^\delta$, and acts in some way on the rest of $\HH$ which we ignore. Such a process is depicted generically in Figure \ref{fig:tapesofU} and the factor contributed by this diagram is
\be\label{eq:diagram_for_zeta_eq}
Z^{\,\rightarrow}(a,|\delta_1|,\ldots,|\delta_m|) = \frac{1}{a + |\delta_1|} \frac{1}{a + |\delta_1| + |\delta_3|} \cdots \frac{1}{a + |\delta_1| + \cdots + |\delta_m|}\,.
\ee
\begin{figure}
\begin{tikzpicture}[scale=0.8,auto]
\draw[line width=2pt] (-8,0) -- (5,0);
\draw[snake=coil,segment aspect=0, segment amplitude=2pt, segment length=7pt] (-6.7,0) -- (-5.7,4);
\draw[double] (-7,0) -- (-6,4);
\node (1t) at (-5.5,4.5) {$\theta_1 \partial_{t_{i_1}}(t^{\delta_1})$};
\draw[snake=coil,segment aspect=0, segment amplitude=2pt, segment length=7pt] (-3.7,0) -- (-2.7,4);
\draw[double] (-4,0) -- (-3,4);
\node (2t) at (-2.5,4.5) {$\theta_2 \partial_{t_{i_2}}(t^{\delta_2})$};
\node (dots) at (0,3) {$\cdots$};
\draw[snake=coil,segment aspect=0, segment amplitude=2pt, segment length=7pt] (2.3,0) -- (3.3,4);
\draw[double] (2,0) -- (3,4);
\node (nt) at (3.7,4.5) {$\theta_m \partial_{t_{i_m}}(t^{\delta_m})$};
\node (1) at (-7.6,-0.5) {$a$};
\node (2) at (-5.5,-0.5) {$a + |\delta_1|$};
\node (m) at (4,-0.5) {$a + \sum_{i=1}^m |\delta_i|$};
\end{tikzpicture}
\centering
\caption{Depiction of a process contributing $\zeta$ factors. The labels indicate the virtual weight of the tensor obtained by cutting the diagram vertically at that position. Moving left to right in this diagram is to be read as moving down the tree.}\label{fig:tapesofU}
\end{figure}
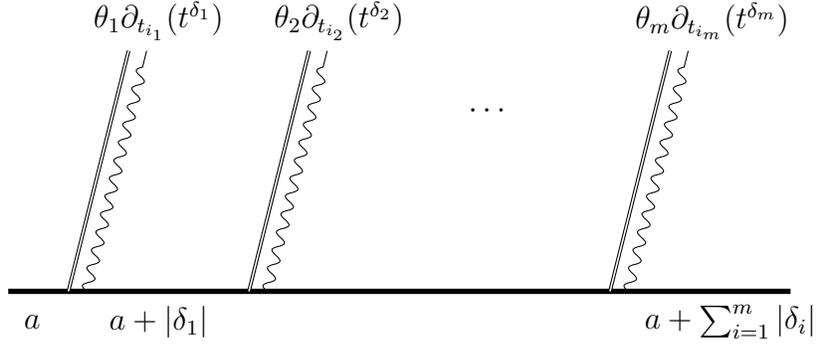
There is also an occurrence of $\zeta$ as $\zeta \nabla$ in $\phi_\infty$, which contributes a scalar factor immediately after every B-type vertex. To be precise, $(\zeta \vAt)^m \zeta$ will contribute
\be\label{eq:diagram_for_zeta_eq_b}
\frac{1}{a} \frac{1}{a + |\delta_1|} \frac{1}{a + |\delta_1| + |\delta_3|} \cdots \frac{1}{a + |\delta_1| + \cdots + |\delta_m|}\,.
\ee

\subsection{Diagrams for $\mu_2$}\label{section:feynman_diagram_3}

Finally, we require a diagrammatic representation for the composition
\[
\mu_2: \HH(Y, Z) \otimes \HH(X, Y) \lto \HH(X,Z)\,.
\]
In addition to the spaces $\widetilde{X} = \bigwedge F_\xi$ and $\widetilde{Y} = \bigwedge F_\eta$ underlying the matrix factorisations $X,Y$ we now introduce odd generators $\varepsilon_1,\ldots,\varepsilon_t$, $F_\varepsilon = \bigoplus_{i=1}^t k \varepsilon_i$ and set
\[
\widetilde{Z} = \bigwedge F_\varepsilon = \bigwedge\big( k \varepsilon_1 \oplus \cdots \oplus k \varepsilon_t \big)
\]
underlying a Koszul matrix factorisation $Z = \widetilde{X} \otimes R$. 

\begin{lemma}\label{lemma_mu2presentation} For $\omega, \omega' \in \bigwedge F_\theta$ and $\alpha \in \Hom_k(\widetilde{Y}, \widetilde{Z}), \beta \in \Hom_k(\widetilde{X}, \widetilde{Y})$
\begin{gather*}
\mu_2\Big( [ \omega \otimes z_h \otimes \alpha ] \otimes [ \omega' \otimes z_l \otimes \beta ] \Big) \\
= (-1)^{|\alpha||\omega'|} \sum_{k=1}^\mu \sum_{\delta} \Gamma^{hl}_{k \delta} \cdot \omega \wedge \omega' \otimes z_k \otimes \alpha \circ \beta \otimes t^{\delta}\,.
\end{gather*}
\end{lemma}

First we explain how to represent $\alpha \circ \beta$ diagrammatically. The following lemmas are proven by straightforward direct calculations, which we omit.

\begin{lemma}\label{lemma:mixedr2_0} The diagram
\be
\xymatrix@C+2pc@R+2pc{
\Hom_k( \widetilde{Y}, \widetilde{Z} ) \otimes \Hom_k( \widetilde{X}, \widetilde{Y} ) \ar[d]_-{\nu \otimes \nu} \ar[r]^-{- \circ -} & \Hom_k( \widetilde{X}, \widetilde{Z} ) \ar[d]^-\nu\\
\big( \bigwedge F_\varepsilon \otimes \bigwedge F_{\Bar{\eta}} \big) \otimes \big( \bigwedge F_\eta \otimes \bigwedge F_{\Bar{\xi}} \big) \ar[d]_-{\exp( \sum_i \eta_i^* \Bar{\eta}_i^* )} & \bigwedge F_\varepsilon \otimes \bigwedge F_{\Bar{\xi}} \ar@{=}[d]\\
\big( \bigwedge F_\varepsilon \otimes \bigwedge F_{\Bar{\eta}} \big) \otimes \big( \bigwedge F_\eta \otimes \bigwedge F_{\Bar{\xi}} \big) \ar[r]_-P & \bigwedge F_\varepsilon \otimes \bigwedge F_{\Bar{\xi}}
}
\ee
commutes, where $P$ denotes the projection onto the subspace with no $\Bar{\eta}$'s or $\eta$'s.
\end{lemma}

\begin{lemma}\label{lemma:mixedr2_1} The diagram
\be
\xymatrix@C+2pc@R+2pc{
\Hom_k( \widetilde{X}, \widetilde{Y} ) \otimes \Hom_k( \widetilde{X}, \widetilde{X} ) \ar[d]_-{\nu \otimes \rho^{-1}} \ar[r]^-{- \circ -} & \Hom_k( \widetilde{X}, \widetilde{Y} ) \ar[d]^-\nu\\
\big( \bigwedge F_\eta \otimes \bigwedge F_{\Bar{\xi}} \big) \otimes \big( \bigwedge F_\xi \otimes \bigwedge F_{\Bar{\xi}} \big) \ar[d]_-{\exp( \sum_i \xi_i^* \Bar{\xi}_i^* )} & \bigwedge F_\eta \otimes \bigwedge F_{\Bar{\xi}} \\
\big( \bigwedge F_\eta \otimes \bigwedge F_{\Bar{\xi}} \big) \otimes \big( \bigwedge F_\xi \otimes \bigwedge F_{\Bar{\xi}} \big) \ar[r]_-P & \bigwedge F_\eta \otimes \bigwedge F_{\Bar{\xi}} \otimes \bigwedge F_{\Bar{\xi}} \ar[u]_-{1 \otimes m}
}
\ee
commutes, where $P$ denotes the projection onto the subspace with no $\xi$'s and $m$ is multiplication in the exterior algebra.
\end{lemma}

\begin{lemma}\label{lemma:mixedr2_2} The diagram
\be
\xymatrix@C+2pc@R+2pc{
\Hom_k( \widetilde{Y}, \widetilde{Y} ) \otimes \Hom_k( \widetilde{X}, \widetilde{Y} ) \ar[d]_-{\rho^{-1} \otimes \nu} \ar[r]^-{- \circ -} & \Hom_k( \widetilde{X}, \widetilde{Y} ) \ar[d]^-\nu\\
\big( \bigwedge F_\eta \otimes \bigwedge F_{\Bar{\eta}} \big) \otimes \big( \bigwedge F_\eta \otimes \bigwedge F_{\Bar{\xi}} \big) \ar[d]_-{\exp( \sum_i \eta_i^* \Bar{\eta}_i^* )} & \bigwedge F_\eta \otimes \bigwedge F_{\Bar{\xi}} \\
\big( \bigwedge F_\eta \otimes \bigwedge F_{\Bar{\eta}} \big) \otimes \big( \bigwedge F_\eta \otimes \bigwedge F_{\Bar{\xi}} \big) \ar[r]_-P & \bigwedge F_\eta \otimes \bigwedge F_\eta \otimes \bigwedge F_{\Bar{\xi}} \ar[u]_-{m \otimes 1}
}
\ee
commutes, where $P$ projects onto the subspace with no $\Bar{\eta}$'s and $m$ is multiplication in the exterior algebra. The operator $\eta_i^*$ in the exponential acts on the third tensor factor.
\end{lemma}

These lemmas allow us to represent the $\alpha \circ \beta$ part of $\mu_2$ as a boundary condition $P$ together with new types of interaction vertices. From Lemma \ref{lemma:mixedr2_0} we obtain the (D.1)-type vertex, in which $\eta_i^*\Bar{\eta}_i^*$ couples an incoming $\eta_i$ in the right branch with an incoming $\Bar{\eta}_i$ (which we view as an $\eta_i$ travelling upward) in the left branch. From Lemma \ref{lemma:mixedr2_1} we obtain the (D.2)-type vertex, which has a similar description. To these interaction vertices we add the (D.3)-type vertex, which represents the $\Gamma^{hl}_{k\delta}z_k( z_h^* \otimes z_l^* )$ part of the $\mu_2$ operator in Lemma \ref{lemma_mu2presentation} (keeping in mind that the incoming $z_h$ and $z_l$ are on the left and right branch at an internal vertex of the tree, respectively):

\begin{center}
\begin{tabular}{ >{\centering}m{5cm} >{\centering}m{5cm} >{\centering}m{5cm} }
\textbf{(D.1) $+1$}
\vspace{0.1cm}
\[
\xymatrix@C+1pc@R+1.5pc{
& & \ar[dl]^-{\eta_i} \\
& \bullet \ar[ul]^-{\eta_i}\\
& & &
}
\]
&
\textbf{(D.2) $+1$}
\vspace{0.1cm}
\[
\xymatrix@C+1pc@R+1.5pc{
& & \ar[dl]^-{\xi_i} \\
& \bullet \ar[ul]^-{\xi_i}\\
& & &
}
\]
&
\textbf{(D.3) $\Gamma^{hl}_{k\delta}$}
\vspace{0.1cm}
\[
\xymatrix@C+1pc@R+1.5pc{
\ar@{.}[dr]^-{h} & & \ar@{.}[dl]^-{l}\\
& \bullet \ar@{.}[d]_-{k} \ar@{~}[dr]^-{t^\delta}\\
& & &
}
\]
\end{tabular}
\end{center}

Note that in the ``mixed'' cases of Lemma \ref{lemma:mixedr2_1} and Lemma \ref{lemma:mixedr2_2} there are still multiplication operators $m$ in the final presentation. For example the $1 \otimes m$ in Lemma \ref{lemma:mixedr2_1} means that an upward travelling $\xi$ entering the vertex can either continue upwards into the left branch, or into the right branch. More precisely, since $\Bar{\xi}_i^*$ is a graded derivation $\Bar{\xi}_i^* m = m (\Bar{\xi}_i^* \otimes 1) + m( 1 \otimes \Bar{\xi}_i^* )$. A similar description applies to the $m \otimes 1$ in Lemma \ref{lemma:mixedr2_2}. 

\subsection{The Feynman rules}\label{section:feynman_diagram_4}

We now integrate the previous sections into a method for reasoning about higher operations $\rho_k$ using Feynman diagrams. We fix matrix factorisations $X_0,\ldots,X_k \in \AA$ which we assume to be Koszul with underlying graded $k$-modules $\widetilde{X}_i = \bigwedge F^{(i)}_\xi$. Once we choose for each pair $(i,i+1)$ and $(0,k)$ either $\nu$ or $\rho$ to present the mapping spaces $\BB(X_i, X_{i+1})$ as a tensor product of exterior algebras, the higher operation $\rho_k$ is a $k$-linear map
\be
\rho_k: \bigotimes_{i=0}^{k-1}\Big[ R/I \otimes \bigwedge F_{\xi}^{(i+1)} \otimes \bigwedge F_{\Bar{\xi}}^{(i)} \Big][1] \lto \Big[ R/I \otimes \bigwedge F_{\xi}^{(k)} \otimes \bigwedge F_{\Bar{\xi}}^{(0)} \Big] [1]
\ee
which is a (signed) sum of operators $\rho_T$ for binary plane trees $T$ with $k$ inputs. According to Lemma \ref{prop:replacer2} evaluating $\rho_T( \beta_1, \ldots, \beta_k )$ involves applying $\operatorname{eval}_{D'}$ to the input $\beta_k \otimes \cdots \otimes \beta_1$. We develop a diagrammatic understanding of the evaluation of $\operatorname{eval}_{D'}$ on this tensor, via Feynman diagrams embedded in a thickening of the mirror tree $T'$. Given a basis vector
\[
\tau \in R/I \otimes \bigwedge F_\xi^{(k)} \otimes \bigwedge F_{\Bar{\xi}}^{(0)}
\]
we wish to know the coefficient of $\tau$ in the evaluation $\eval_{D'}( \beta_k, \ldots, \beta_1 )$, which we denote
\be\label{eq:defn_ctau}
C_\tau := \tau^* \eval_{D'}( \beta_k, \ldots, \beta_1 ) \in k\,.
\ee
The description of diagrams contributing to $C_\tau$ is reached in several stages, which are summarised by the Feynman rules in Definition \ref{defn:feynman_rules}. To explain the algorithm it will be helpful to keep in mind the data structure
\be
\cat{D} = \{ ( 1, \tau, \operatorname{eval}_{D'}, \beta_k \otimes \cdots \otimes \beta_1 ) \}\,.
\ee
This data structure will be modified as we proceed, but it will always be a sequence of tuples $(\lambda, \alpha, \psi, \beta)$ consisting of a scalar $\lambda \in k$, an output basis vector $\alpha$, a $k$-linear operator $\psi$ with the same domain and codomain as $\operatorname{eval}_{D'}$, and an input basis vector $\beta$. Each time we modify $\cat{D}$ the sum will remain invariant, that is, we will always have
\be\label{eq:ctauconstraint}
C_\tau = \sum_{(\lambda, \alpha, \psi, \beta) \in \cat{D}} \lambda \alpha^* \psi( \beta )\,.
\ee
In the following ``the tree'' means $T'$ unless specified otherwise. 
\\

\textbf{Stage one: expansion.} Recall $\operatorname{eval}_{D'}$ is defined as a composition of operators
\[
\sigma_\infty, \phi_\infty, e^{\delta}, e^{-\delta}, \mu_2, \pi
\]
which may be written, with some signs and factorials, in terms of the operators\footnote{For the reader's convenience, here is a cheatsheet: for $\zeta$ see Section \ref{section:propagator}, $\vAt_{\AA}$ is the Atiyah class of Definition \ref{defn:atiyah_class}, $\nabla$ the chosen connection from Corollary \ref{corollary:nabla_sigma}, $\sigma$ the chosen section of the quotient map $\pi: R \lto R/I$, for $\delta$ see Definition \ref{defn:important_operators}, $\mu_2$ is ordinary composition in $\AA'_\theta$.}
\[
\zeta, \vAt_{\AA}, \nabla, \sigma, \delta, \mu_2, \pi\,.
\]
The signs and factorials involved are accounted for carefully in Definition \ref{defn:feynman_rules} below; for clarity we omit them in the present discussion. Choose a presentation $\nu$ or $\rho$ for each edge in $T'$. We expand occurrences of $\vAt_{\AA}, \nabla, \delta$ using \eqref{eq:a1vertex}-\eqref{eq:a4vertex},\eqref{eq:c1vertex},\eqref{eq:c2vertex} on the edges of the tree for which $\nu$ has been chosen and \eqref{eq:a3vertex_rho},\eqref{eq:a4vertex_rho},\eqref{eq:c1vertex_rho},\eqref{eq:c2vertex_rho} on the edges for which $\rho$ is chosen. We replace the occurrences of $\mu_2$ using the lemmas of Section \ref{section:fenyman_diagram_2} with additional exponentials, occurrences of $m$ and $\Gamma$ contributions from (D.3) vertices. Next we absorb the $\tau^*$ and $\beta_k \otimes \cdots \otimes \beta_1$ from \eqref{eq:defn_ctau} into the operator decorated tree, by writing the input tensors $\beta_i$ as products of creation operators and the projection $\tau^*$ as a product of annihilation operators followed by the projection $1^*$ (here we assume for simplicity that $z_1 = 1$ in the chosen basis for $R/I$). In the resulting expression the remaining terms that are \emph{not} creation and annihilation operators are occurrences of $\zeta$ and the multiplications $m$ on exterior algebras from Lemma \ref{lemma:mixedr2_1} and Lemma \ref{lemma:mixedr2_2}. Since the relevant virtual degrees are now all fixed, we can calculate the scalar contribution from the $\zeta$ operators and absorb them into the $\lambda$ coefficients. Hence we can replace $\cat{D}$ by a sequence of tuples
\[
(\lambda, 1^*, \psi, 1)
\]
in which every $\psi$ is the denotation of an operator decorated tree (in the sense of Definition \ref{defn:evaluation_tree}) with operators taken from the list (here $\xi, \Bar{\xi}$ stand for any fermionic generator coming from the matrix factorisations themselves, so we do not separately list $\eta, \Bar{\eta}$)
\be\label{eq:possible_operators}
\theta, \theta^*, z, z^*, t, \partial_t, \xi, \xi^*, \Bar{\xi}, \Bar{\xi}^*, m, \pi\,.
\ee
This completes the expansion stage.
\\

\textbf{Stage two: reduction to normal form.} The operators $\psi$ in $\cat{D}$ are denotations of trees decorated by monomials in creation and annihilation operators. Such an operator (or more precisely the decoration from which it arises) is said to be in \emph{normal form} if any path from an annihilation operator in the tree to an input leaf encounters no creation operators (roughly, annihilation operators appear to the right of creation operators). After stage one the operators $\psi$ are not in normal form, and we now explain a rewrite process which transforms $\cat{D}$ such that after each step \eqref{eq:ctauconstraint} holds, and the process terminates with the operator $\psi$ in every tuple of $\cat{D}$ in normal form. During some steps of the rewrite process a tuple $(\lambda, 1^*, \psi, 1) \in \cat{D}$ is replaced by a pair $\{ (\lambda_i, 1^*, \psi_i, 1) \}_{i=1}^2$, because the rewriting is ``nondeterministic'' in the sense that it involves binary choices. A Feynman diagram is a \emph{graphical representation of such binary choices made during rewriting.} 

Here is the informal algorithm for the rewrite process: take the first tuple $(\lambda, 1^*, \psi, 1)$ in $\cat{D}$ which contains a fermionic annihilation operator, and let $\omega^*$ be the one occurring closest to the root (so $\omega$ is $\theta, \xi$ or $\Bar{\xi}$) and use the available (anti)commutation relations to move it up the tree past the other operators in \eqref{eq:possible_operators}, changing the sign of $\lambda$ as appropriate. The only nontrivial anticommutators that we encounter are:
\begin{itemize}
\item[(a)] $\omega^*$ meets $m$ and generates two additional terms 
\[
\omega^* m = m( 1 \otimes \omega^* ) + m( \omega^* \otimes 1)
\]
\begin{center}
\includegraphics[scale=0.6]{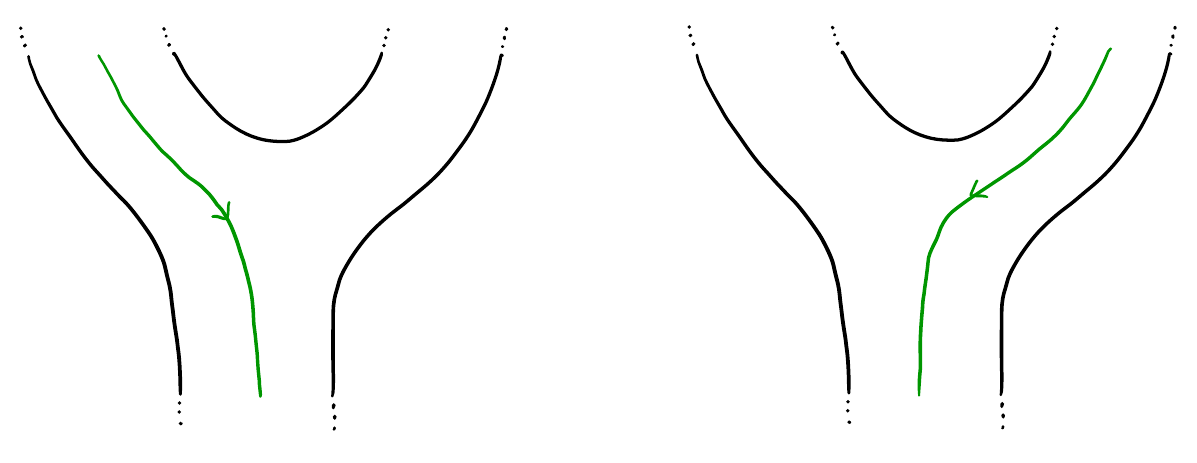}
\end{center}

\item[(b)] $\omega^*$ meets $\omega$ and generates two additional terms
\[
\omega^* \omega = - \omega \omega^* + 1
\]
\begin{center}
\includegraphics[scale=0.6]{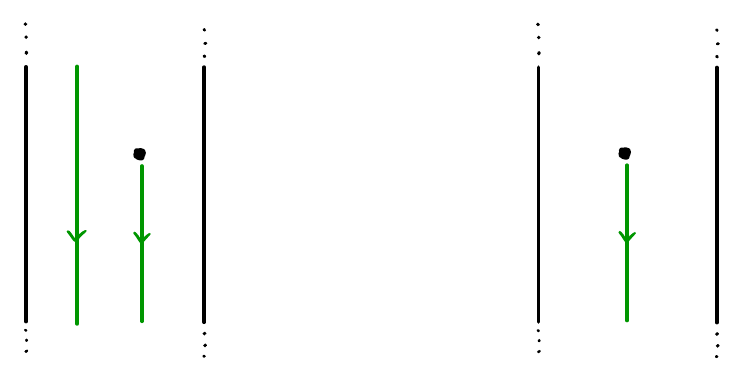}
\end{center}

\item[(c)] $\omega^*$ meets the input $1$.
\end{itemize}
In (a), (b) the tuple $(\lambda, 1^*, \psi, 1)$ is replaced in $\cat{D}$ by two new tuples, in which the decoration of the tree differs from the one determining $\psi$ only in the indicated way (changing the coefficient $\lambda$ by a sign if $\omega^* \omega$ is replaced by $\omega \omega^*$). In (c) we remove the tuple $(\lambda, 1^*, \psi, 1)$ from $\cat{D}$, since it contributes zero to $C_\tau$. The meaning of the pictures will become clear later. We say the occurrence of $\omega^*$ in these new tuples is \emph{descended} from the original $\omega^*$ and we continue the process of commuting these descendents upwards until in $\cat{D}$ there are no tuples containing operators $\omega^*$ descended from our original annihilation operator. Once this is done we return to the beginning of the loop, choosing a new fermionic annihilation operator in $\cat{D}$ as our $\omega^*$. This part of the algorithm terminates when there are no fermionic annihilation operators remaining in $\cat{D}$. It is possible that $\cat{D}$ is now empty, so that $C_\tau = 0$ and the overall algorithm terminates.

We next treat the occurrences of the bosonic annihilation operators $\partial_{t_i}$ in the same way, with the only nontrivial commutation relation being $\partial_{t_i} t_i = t_i \partial_{t_i} + 1$ which generates two new tuples in $\cat{D}$. The $z_h^*$ operators act on the next $z_l$ to give scalar factors $\delta_{h=l}$. This part of the algorithm terminates when there are no annihilation operators remaining, that is, we have replaced $\cat{D}$ by a sequence of tuples $(\lambda, 1^*, \psi, 1)$ in which every $\psi$ is the denotation of an operator decorated tree with operators taken from
\[
\theta, t, \xi, \Bar{\xi}, m, \pi\,.
\]
Since we apply $\pi$ and $1^*$ at the bottom of our diagrams, we do not change the coefficient \eqref{eq:ctauconstraint} if we delete from $\cat{D}$ any tuple in which $\psi$ contains a creation operator. After doing so, the remaining tuples all have the same decoration by $m$'s and $\pi$ and hence the same operator $\psi_{final}$, so at the completion of the second stage we have replaced $\cat{D}$ by a set
\[
\Big\{ (C_{\tau, F}, 1^*, \psi_{final}, 1) \Big\}_{F \in \cat{F}}
\]
where $\cat{F}$ is some index set. Hence $C_\tau = \sum_{F \in \cat{F}} C_{\tau, F}$.
\\

\textbf{Stage three: drawing the diagram.} An index $F \in \cat{F}$ contains the information of a sequence of binary choices made at each nontrivial step of the rewrite process: for example, the choice to replace $\omega^* \omega$ by either $\pm \omega \omega^*$ or $1$, where $\omega$ is one of $\theta, t, \xi, \Bar{\xi}$. The tuple $(\lambda, 1^*, \psi, 1)$ to which this choice refers has a unique ancestor among the tuples in $\cat{D}$ at the end of stage one. In that ancestor tuple, the creation operator $\omega$ is associated with a particular monomial $P$ inserted at a location on $T'$, and the annihilation operator $\omega^*$ is associated with a monomial $P'$ at some other location. By construction $P'$ occurs lower on the tree than $P$. These monomials are precisely the \emph{interaction vertices} to which we have assigned names and diagrams above.

To encode the information in $F$ diagrammatically, we draw all the interaction vertices (in the unique ancestor of the tuple $(C_{\tau, F}, 1^*, \psi_{final}, 1)$ among the tuples of stage one) at the location that they occur in the decoration of $T'$. Each incoming line labelled $\omega$ to such an interaction vertex is uniquely associated by the choices in $F$ with an outgoing line labelled $\omega$ at some vertex higher up the tree, namely, the first and only creation operator where $F$ chooses $1$ rather than $\pm \omega \omega^*$ for the annihilation operator $\omega^*$ whose ancestor is the chosen incoming line. We connect the two interaction vertices by joining them with a line labelled $\omega$. The resulting diagrammatic representation of $F$ is called a \emph{Feynman diagram}, see Figure \ref{fig:feynman_1} for an example.
\\

In summary, we have the following Feynman rules for enumerating Feynman diagrams $F$ and computing their coefficient $C_{\tau, F}$. See \cite[\S 6.1]{weinberg} for real Feynman rules in QFT.

\begin{definition}[(Feynman rules)]\label{defn:feynman_rules} The coefficient $C_\tau$ is a sum $\sum_{F \in \cat{F}} C_{\tau, F}$ of coefficients $C_{\tau, F}$ associated to Feynman diagrams $F$. To enumerate the possible Feynman diagrams, first draw $\beta_k,\ldots,\beta_1$ in order on the input leaves and $\tau$ on the root of $T'$, as sequences of incoming and outgoing particle lines. Then choose for each
\begin{itemize}
\item \textbf{input} a sequence of A-type vertices, then a sequence of C-type vertices (from $e^{\delta} \sigma_\infty$). 
\item \textbf{internal vertex} a sequence of vertices of type (D.1),(D.2) and exactly one (D.3) vertex (this arising from multiplication in $R/I$).
\item \textbf{internal edge} a sequence of C-type vertices, then a single B-type vertex, then a sequence of A-type vertices and a sequence of C-type vertices (from $e^{\delta} \phi_\infty e^{-\delta}$).
\item and for the \textbf{outgoing edge} a sequence of C-type vertices (from $\pi e^{-\delta}$).
\end{itemize}
Draw the chosen interaction vertices on the thickened tree and choose a way of connecting outgoing lines at vertices to incoming lines of the same type lower down the tree. Since the only lines incident with the incoming and outgoing boundary of the tree are those arising from the $\beta_i$ and $\tau$ no ``virtual particles'' ($t$'s or $\theta$'s) may enter or leave the diagram. In each case the number of A, C, D-type vertices chosen may be zero, but every internal edge has precisely one B-type vertex. These choices parametrise a finite set of Feynman diagrams $\cat{F}$. The coefficient $C_{\tau, F}$ contributed by the Feynman diagram $F$ is the product of the following five contributing factors:
\begin{itemize}
\item[(i)] The coefficients associated to each A,C-type interaction vertex as given above (for example the vertex $(A.1)^{\nu}$ has coefficient $\sum_{\alpha + \beta = \delta } \sum_{m=1}^\mu (u_j)_{(m,\alpha)} \Gamma^{m h}_{l \beta}$ multiplied by the scalar arising from the partial derivative $\partial_{t_k}(t^\delta)$).
\item[(ii)] For each sequence of $m$ C-type vertices a factorial $\frac{1}{m!}$ and if the sequence immediately preceedes a B-type vertex or the root, a sign $(-1)^m$ (from $e^{-\delta}$ versus $e^\delta$).
\item[(iii)] $Z^{\,\rightarrow}$ factors from $\zeta$ operators (see Section \ref{section:fenyman_diagram_2}).
\item[(iv)] For each time two fermion lines cross, a factor of $-1$.
\item[(v)] For each A-type vertex a factor of $-1$ (from the $(-1)^m$ in $\sigma_\infty,\phi_\infty$).
\end{itemize}
\end{definition}

\begin{remark}\label{remark:boson_symmetry} Recall that we draw an interaction vertex $v$ with an outgoing line labelled $t^\delta = t_1^{\delta_1} \cdots t_n^{\delta_n}$ in our pictures, the convention is that such a line stands for $|\delta|$ separate lines labelled $t_i$ for some $i$. The above description of the Feynman rules involves choosing, for each such $v$, a series of B-type vertices to pair with each of these lines. 

This leads to $\delta_1! \cdots \delta_n!$ otherwise identical diagrams, in which the only difference is \emph{which} of the $\delta_i$ lines labelled $t_i$ is paired with which $(B)_i$ vertex. The usual convention is to draw just one such diagram, counted with a symmetry factor $\delta_1! \cdots \delta_n!$.
\end{remark}

\begin{remark}\label{remark:symmetry_A_type} The rules involve \emph{sequences} of A and C-type interactions: different orderings are different Feynman diagrams. In general the Feynman diagrams $F,F'$ associated to different orderings have $C_{\tau, F} \neq C_{\tau, F'}$ because the interaction vertices are operators that do not necessarily commute. In practice, however, one can often infer that the incoming state to a particular edge is constrained to lie in a subspace $\KK \subseteq \HH$ on which all the relevant operators \emph{do} commute, in which case the different orderings can be grouped together and counted as a single diagram with an appropriate symmetry factor (cancelling, in the C-type case, the scalar factor of (ii) in the Feynman rules).

The situation for a length $m$ sequence of A-type vertices is more subtle, because these arise from powers of $\zeta \vAt_{\AA}$ and the operators arising from $\vAt_{\AA}$ do not commute with $\zeta$. However, when the incoming state lies in a subspace $\KK$ to which Lemma \ref{lemma:technical_antic} applies, we can group together permutations of the $m$ vertices and the $Z^{\,\rightarrow}$ factors become a symmetrised $Z$. This applies in the context of Section \ref{section:generator}.
\end{remark}

\begin{example} Consider the potential $W = \frac{1}{5} x^5$ and matrix factorisations
\begin{align}
X &= \big( \bigwedge(k \xi) \otimes k[x], x^2 \xi^* + \frac{1}{5} x^3 \xi \big)\\
Y &= \big( \bigwedge(k \eta) \otimes k[x], x^3 \eta^* + \frac{1}{5} x^2 \eta \big)
\end{align}
so $f = x^2, u = x^3, g = \frac{1}{5}x^3, v = \frac{1}{5} x^2$. We set $t = \partial_x W = x^4$ and choose our connection $\nabla$ and operators $\partial_t$ as in Example \ref{example:dt_xd} with $d = 4$. For the homotopies $\lambda^Y, \lambda^X$ we use the default choices of Remark \ref{remark:default_homotopies}, so that
\begin{align*}
\lambda^X &= \partial_x( d_X ) = 2 x \xi^* + \frac{3}{5} x^2 \xi && F^X = 2x, \quad G^X = \frac{3}{5} x^2\\
\lambda^Y &= \partial_x( d_Y ) = 3 x \eta^* + \frac{2}{5} x \eta && F^Y = 3x, \quad G^Y = \frac{2}{5} x\,.
\end{align*}
To compute the coefficients associated to all the interaction vertices, we need to compute for various $r \in R$ the coefficients $r_{(m,\alpha)}$ (see Definition \ref{defn:rsharp}) as well as the tensor $\Gamma$. This involves fixing the $k$-basis $R/I = k[x]/x^4 = k1 \oplus kx \oplus kx^2 \oplus kx^3$ that is, $z_h = x^h$ for $0 \le h \le 3$, and the section $\sigma(x^i) = x^i$. Then for example
\[
x^3 = 1 \cdot \sigma( x^3 ) t^0
\]
and in general $(x^a)_{(m,\alpha)} = \delta_{m = a} \delta_{\alpha = 0}$ for $0 \le a \le 3$. Since all the polynomials occurring in $f,g,u,v,F,G$ have degree $\le 3$ these coefficients are all easily calculated as delta functions in this way. The tensor $\Gamma$ encodes the multiplication in $R/I$ and is given by Definition \ref{defn_gamma} for $0 \le m,h \le 3$ by
\[
\Gamma^{mh}_{l \beta} = \delta_{m+h \le 3}\delta_{l = m+h}\delta_{\beta = 0} + \delta_{m+h > 3} \delta_{l = m+h-4} \delta_{\beta = 1}\,.
\]
To present interaction vertices as creation and annihilation operators we use $\rho$ on all edges of the tree involving pairs $(Y,Y)$ and $(X,X)$ and $\nu$ on edges involving $(X,Y)$. Such edges involve $\eta$'s travelling downward and $\xi$'s travelling upward, and the interactions with their coefficients are (recall our convention is to write $h$ for $z_h$ in these diagrams):
\begin{center}
\begin{tabular}{ >{\centering}m{5cm} >{\centering}m{5cm} >{\centering}m{5cm} }
\textbf{(A.1)${}^\nu_{h > 0}$ $+1$}
\vspace{0.1cm}
\[
\xymatrix@C+1pc@R+1.5pc{
& \ar@{.}[d]^-{h} & \ar[dl]^-{\eta}\\
& \bullet \ar@{=}[dl]^-{\theta} \ar@{.}[d]^-{h-1}\\
& &
}
\] 
&
\textbf{(A.2)${}^\nu_{h > 1}$ $+\frac{1}{5}$}
\vspace{0.1cm}
\[
\xymatrix@C+1pc@R+1.5pc{
& \ar@{.}[d]^-{h}\\
& \bullet \ar@{=}[dl]_-{\theta} \ar@{.}[d]^-{h-2} \ar[dr]^-{\eta}\\
& &
}
\]
&
\textbf{(A.3)${}^\nu_{h>1}$ $-1$}
\vspace{0.1cm}
\[
\xymatrix@C+1pc@R+1.5pc{
& \ar@{.}[d]^-{h}\\
& \bullet \ar@{=}[dl]_-{\theta} \ar@{.}[d]^-{h-2} \\
& & \ar[ul]_-{\xi}
}
\]
\end{tabular}
\end{center}

\begin{center}
\begin{tabular}{ >{\centering}m{5cm} >{\centering}m{5cm} >{\centering}m{5cm} }
\textbf{(A.4)${}^\nu_{h>0}$ $+\frac{1}{5}$}
\vspace{0.1cm}
\[
\xymatrix@C+1pc@R+1.5pc{
& \ar@{.}[d]^-{h} &\\
& \bullet \ar[ur]_-{\xi} \ar@{=}[dl]^-{\theta} \ar@{.}[d]^-{h-1} \\
& &
}
\] 
&
\textbf{(B)}
\vspace{0.1cm}
\[
\xymatrix@R+1.5pc{
\ar@{~}[d]^-{t}\\
\bullet \ar@{=}[d]^-{\theta}\\
\;
}
\]
& 
\textbf{(C.1)${}^\nu_{h\le2}$ $+3$}
\vspace{0.1cm}
\[
\xymatrix@C+1pc@R+1.5pc{
\ar[dr]_-{\eta} & \ar@{.}[d]^-{h} & \ar@{=}[dl]^-{\theta}\\
& \bullet \ar@{.}[d]^-{h+1} \\
& &
}
\] 
\end{tabular}
\end{center}

\begin{center}
\begin{tabular}{ >{\centering}m{5cm} >{\centering}m{5cm} >{\centering}m{5cm} }
\textbf{(C.1)${}^\nu_{h>2}$ $+3$}
\vspace{0.1cm}
\[
\xymatrix@C+1pc@R+1.5pc{
\ar[dr]_-{\eta} & \ar@{.}[d]^-{h} & \ar@{=}[dl]^-{\theta}\\
& \bullet \ar@{.}[d]^-{h-3} \ar@{~}[dl]^-{t}\\
& &
}
\] 
&
\textbf{(C.2)$^\nu_{h \le 2}$ $+\frac{2}{5}$}
\vspace{0.1cm}
\[
\xymatrix@C+1pc@R+1.5pc{
& \ar@{.}[d]^-{h} & \ar@{=}[dl]^-{\theta} \\
& \bullet \ar[dl]^-{\eta}\ar@{.}[d]^-{h+1} \\
& &
}
\]
&
\textbf{(C.2)$^\nu_{h > 2}$ $+\frac{2}{5}$}
\vspace{0.1cm}
\[
\xymatrix@C+1pc@R+1.5pc{
& \ar@{.}[d]^-{h} & \ar@{=}[dl]^-{\theta} \\
& \bullet \ar[dl]^-{\eta}\ar@{.}[d]^-{h-3} \ar@{~}[dr]^-{t}\\
& &
}
\] 
\end{tabular}
\end{center}

On the edges involving $X$ purely, where we are using the $\rho$ presentation, we have the following interaction vertices (we omit the B-type which is as above):

\begin{center}
\begin{tabular}{ >{\centering}m{5cm} >{\centering}m{5cm} >{\centering}m{5cm} }
\textbf{(A.1)${}^{X,\rho}_{h>1}$ $+1$}
\vspace{0.1cm}
\[
\xymatrix@C+1pc@R+1.5pc{
& \ar@{.}[d]^-{h} & \ar[dl]^-{\xi}\\
& \bullet \ar@{=}[dl]_-{\theta} \ar@{.}[d]^-{h-2} \\
& & &
}
\] 
&
\textbf{(A.4)${}^{X,\rho}_{h>0}$ $+\frac{1}{5}$}
\vspace{0.1cm}
\[
\xymatrix@C+1pc@R+1.5pc{
& \ar@{.}[d]^-{h} &\\
& \bullet \ar[ur]_-{\xi} \ar@{=}[dl]^-{\theta} \ar@{.}[d]^-{h-1}\\
& &
}
\] 
&
\textbf{(C.1)${}^{X,\rho}_{h \le 2}$ $+2$}
\vspace{0.1cm}
\[
\xymatrix@C+1pc@R+1.5pc{
\ar[dr]_-{\xi} & \ar@{.}[d]^-{h} & \ar@{=}[dl]^-{\theta}\\
& \bullet \ar@{.}[d]^-{h+1}\\
& &
}
\]
\end{tabular}
\end{center}

\begin{center}
\begin{tabular}{ >{\centering}m{5cm} >{\centering}m{5cm} >{\centering}m{5cm} }
\textbf{(C.1)${}^{X,\rho}_{h > 2}$ $+2$}
\vspace{0.1cm}
\[
\xymatrix@C+1pc@R+1.5pc{
\ar[dr]_-{\xi} & \ar@{.}[d]^-{h} & \ar@{=}[dl]^-{\theta}\\
& \bullet \ar@{.}[d]^-{h-3} \ar@{~}[dl]^-{t}\\
& &
}
\]
&
\textbf{(C.2)${}^{X,\rho}_{h \le 1}$ $+\frac{3}{5}$}
\vspace{0.1cm}
\[
\xymatrix@C+1pc@R+1.5pc{
& \ar@{.}[d]^-{h} & \ar@{=}[dl]^-{\theta} \\
& \bullet \ar[dl]^-{\xi}\ar@{.}[d]^-{h+2}\\
& &
}
\]
&
\textbf{(C.2)${}^{X,\rho}_{h > 1}$ $+\frac{3}{5}$}
\vspace{0.1cm}
\[
\xymatrix@C+1pc@R+1.5pc{
& \ar@{.}[d]^-{h} & \ar@{=}[dl]^-{\theta} \\
& \bullet \ar[dl]^-{\xi}\ar@{.}[d]^-{h-2} \ar@{~}[dr]^-{t}\\
& &
}
\]
\end{tabular}
\end{center}

\begin{center}
\begin{tabular}{ >{\centering}m{6cm} >{\centering}m{6cm} }
\textbf{(C.3)${}^{X,\rho}_{h \le 2}$ $+2$}
\vspace{0.1cm}
\[
\xymatrix@C+2pc@R+1.5pc{
& \ar@{.}[d]^-{h} & \ar@{=}[dl]^-{\theta} \\
& \bullet \ar@{.}[d]^-{h+1} \\
\ar[ur]^-{\xi} & &
}
\]
&
\textbf{(C.3)${}^{X,\rho}_{h > 2}$ $+2$}
\vspace{0.1cm}
\[
\xymatrix@C+2pc@R+1.5pc{
& \ar@{.}[d]^-{h} & \ar@{=}[dl]^-{\theta} \\
& \bullet \ar@{.}[d]^-{h-3} \ar@{~}[dr]^-{t}\\
\ar[ur]^-{\xi} & &
}
\]
\end{tabular}
\end{center}

On edges involving $Y$ purely, where again we use the $\rho$ presentation, we have:

\begin{center}
\begin{tabular}{ >{\centering}m{5cm} >{\centering}m{5cm} >{\centering}m{5cm} }
\textbf{(A.1)${}^{Y,\rho}_{h > 0}$ $+1$}
\vspace{0.1cm}
\[
\xymatrix@C+1pc@R+1.5pc{
& \ar@{.}[d]^-{h} & \ar[dl]^-{\eta}\\
& \bullet \ar@{=}[dl]_-{\theta} \ar@{.}[d]^-{h-1} \\
& &
}
\]
&
\textbf{(A.4)${}^{Y,\rho}_{h>1}$ $+\frac{1}{5}$}
\vspace{0.1cm}
\[
\xymatrix@C+1pc@R+1.5pc{
& \ar@{.}[d]^-{h} &\\
& \bullet \ar[ur]_-{\eta} \ar@{=}[dl]^-{\theta} \ar@{.}[d]^-{h-2} \\
& &
}
\]
&
\textbf{(C.1)${}^{Y,\rho}_{h \le 2}$ $+3$}
\vspace{0.1cm}
\[
\xymatrix@C+1pc@R+1.5pc{
\ar[dr]_-{\eta} & \ar@{.}[d]^-{h} & \ar@{=}[dl]^-{\theta}\\
& \bullet \ar@{.}[d]^-{h+1} \\
& &
}
\]
\end{tabular}
\end{center}

\begin{center}
\begin{tabular}{ >{\centering}m{5cm} >{\centering}m{5cm} >{\centering}m{5cm} }
\textbf{(C.1)${}^{Y,\rho}_{h>2}$ $+3$}
\vspace{0.1cm}
\[
\xymatrix@C+1pc@R+1.5pc{
\ar[dr]_-{\eta} & \ar@{.}[d]^-{h} & \ar@{=}[dl]^-{\theta}\\
& \bullet \ar@{.}[d]^-{h-3} \ar@{~}[dl]^-{t}\\
& &
}
\]
&
\textbf{(C.2)${}^{Y,\rho}_{h \le 2}$ $+\frac{2}{5}$}
\vspace{0.1cm}
\[
\xymatrix@C+2pc@R+1.5pc{
& \ar@{.}[d]^-{h} & \ar@{=}[dl]^-{\theta} \\
& \bullet \ar[dl]^-{\eta}\ar@{.}[d]^-{h+1}\\
& &
}
\]
&
\textbf{(C.2)${}^{Y,\rho}_{h > 2}$ $+\frac{2}{5}$}
\vspace{0.1cm}
\[
\xymatrix@C+2pc@R+1.5pc{
& \ar@{.}[d]^-{h} & \ar@{=}[dl]^-{\theta} \\
& \bullet \ar[dl]^-{\eta}\ar@{.}[d]^-{h-3} \ar@{~}[dr]^-{t}\\
& &
}
\]
\end{tabular}
\end{center}

\begin{center}
\begin{tabular}{ >{\centering}m{6cm} >{\centering}m{6cm} }
\textbf{(C.3)${}^{Y,\rho}_{h \le 2}$ $+3$}
\vspace{0.1cm}
\[
\xymatrix@C+2pc@R+1.5pc{
& \ar@{.}[d]^-{h} & \ar@{=}[dl]^-{\theta} \\
& \bullet \ar@{.}[d]^-{h+1}\\
\ar[ur]^-{\eta} & &
}
\]
&
\textbf{(C.3)${}^{Y,\rho}_{h > 2}$ $+3$}
\vspace{0.1cm}
\[
\xymatrix@C+2pc@R+1.5pc{
& \ar@{.}[d]^-{h} & \ar@{=}[dl]^-{\theta} \\
& \bullet \ar@{.}[d]^-{h-3} \ar@{~}[dr]^-{t}\\
\ar[ur]^-{\eta} & &
}
\]
\end{tabular}
\end{center}

Suppose we want to evaluate the forward suspended product
\[
\rho_3: \HH(X,X)[1] \otimes \HH(X,Y)[1] \otimes \HH(Y,Y)[1] \lto \HH(X,Y)[1]
\]
on an input tensor
\[
x^2 \xi \otimes x \eta \Bar{\xi} \otimes x^3 \Bar{\eta} = z_2 \xi \otimes z_1 \eta \Bar{\xi} \otimes z_3 \Bar{\eta}\,.
\]
We examine one Feynman diagram contributed by our standard tree $T$ of Figure \ref{eq:explicit_tree_operator}. We are in the situation examined in the proof of Lemma \ref{prop:replacer2}, with
\[
\rho_T( z_2 \xi, z_1 \eta \Bar{\xi}, z_3 \Bar{\eta} ) = \pi e^{-\delta} \mu_2\Big\{ e^{\delta} \phi_\infty e^{-\delta} \mu_2\Big( e^\delta \sigma_\infty( z_3 \Bar{\eta} ) \otimes e^\delta \sigma_\infty( z_1 \eta \Bar{\xi} ) \Big) \otimes e^\delta\sigma_\infty( z_2 \xi ) \Big\}\,.
\]
Next we expand the exponentials and $\phi_\infty, \sigma_\infty$, among the summands is
\begin{align*}
\pi \mu_2\Big\{ \delta \zeta \nabla (-\delta) \mu_2\Big( z_3 \Bar{\eta} \otimes (-\zeta \vAt_{\AA})( z_1\eta \Bar{\xi} ) \Big) \otimes \delta (-\zeta \vAt_{\AA})( z_2 \xi ) \Big\}\,.
\end{align*}
If we now further expand $\delta, \vAt_{\AA}, \nabla$ among the summands is
\begin{gather*}
- \pi \mu_2\Big\{ [ 3 z_1 \eta^* \theta^* ] \zeta \theta \partial_t [ \tfrac{2}{5} \eta t z_3^* \theta^* ] \mu_2\Big( z_3 \Bar{\eta} \otimes \zeta [\tfrac{1}{5} \theta z_1^* \Bar{\xi}^*] z_1 \eta \Bar{\xi}\Big) \otimes [2z_1 \Bar{\xi} \theta^*] \zeta [\theta \xi^* z_2^*] z_2 \xi \Big\}\\
= - \tfrac{3 \cdot 2 \cdot 1 \cdot 2}{5^2} \pi \mu_2\Big\{ [ z_1 \eta^* \theta^* ] \zeta \theta \partial_t [ \eta t z_3^* \theta^* ] \mu_2\Big( z_3 \Bar{\eta} \otimes [ \theta z_1^* \Bar{\xi}^*] z_1 \eta \Bar{\xi}\Big) \otimes [z_1 \Bar{\xi} \theta^*] [\theta \xi^* z_2^*] z_2 \xi \Big\}
\end{gather*}
which are, reading from left to right, the vertices $(\textup{C}.1)^\nu_{h=0}, (B), (\textup{C}.2)^{\nu}_{h=3}, (\textup{A}.4)^\nu_{h=1}, (\textup{C}.3)^{X,\rho}_{h=0}$ and $(\textup{A}.1)^{X,\rho}_{h=2}$. Next we replace all occurrences of $\mu_2$ according to Section \ref{section:feynman_diagram_3}. Among the summands are (we omit the boundary condition operators $P$ for legibility)
\begin{gather*}
- \tfrac{3 \cdot 2 \cdot 1 \cdot 2}{5^2} \pi m\Big\{ [ z_1 \eta^* \theta^* ] \zeta \theta \partial_t [ \eta t z_3^* \theta^* ] m ( \eta^* \Bar{\eta}^* ) \Big( z_3 \Bar{\eta} \otimes [ \theta z_1^* \Bar{\xi}^*] z_1 \eta \Bar{\xi}\Big) \otimes [z_1 \Bar{\xi} \theta^*] [\theta \xi^* z_2^*] z_2 \xi \Big\}\,.
\end{gather*}
Now we commute the leftmost fermionic annihilation operator $\eta^*$ to the right. The only nonzero contribution is when this pairs with the next $\eta$. The leftmost $\theta^*$ has to annihilate with the closest $\theta$, the $\partial_t$ has to annihilate with the $t$, and so on. There is only one pattern of contractions which has a nonzero coefficient, and the pairings of fermionic creation and annihilation operators is shown in the diagram
\begin{center}
\includegraphics[scale=0.45]{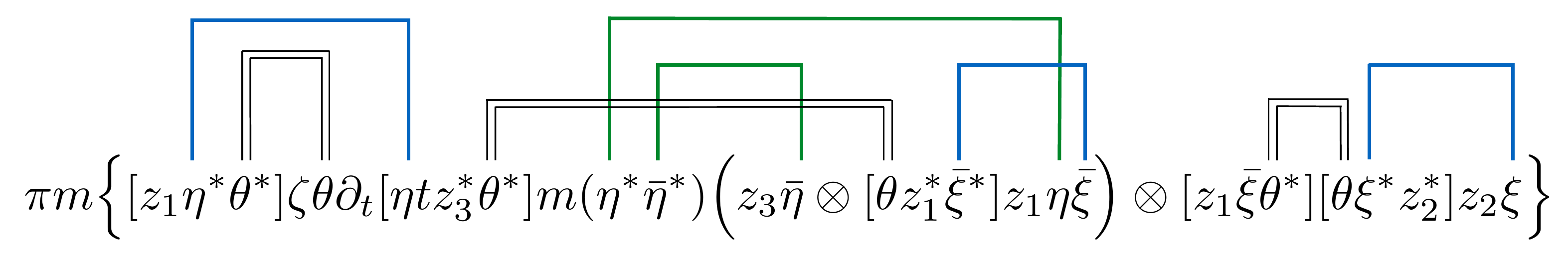}
\end{center}
The corresponding Feynman diagram $F$ with outgoing state $\tau = z_2 \Bar{\xi}$ is shown in Figure \ref{fig:feynman_1}. If we had taken $\tau$ as our outgoing state in the Feynman rules, then this Feynman diagram $F$ would be one of the contributors to $C_{\tau}$ and its contribution is $C_{\tau, F} = - \tfrac{12}{25}$.

\begin{figure}
\begin{center}
\includegraphics[scale=1.0]{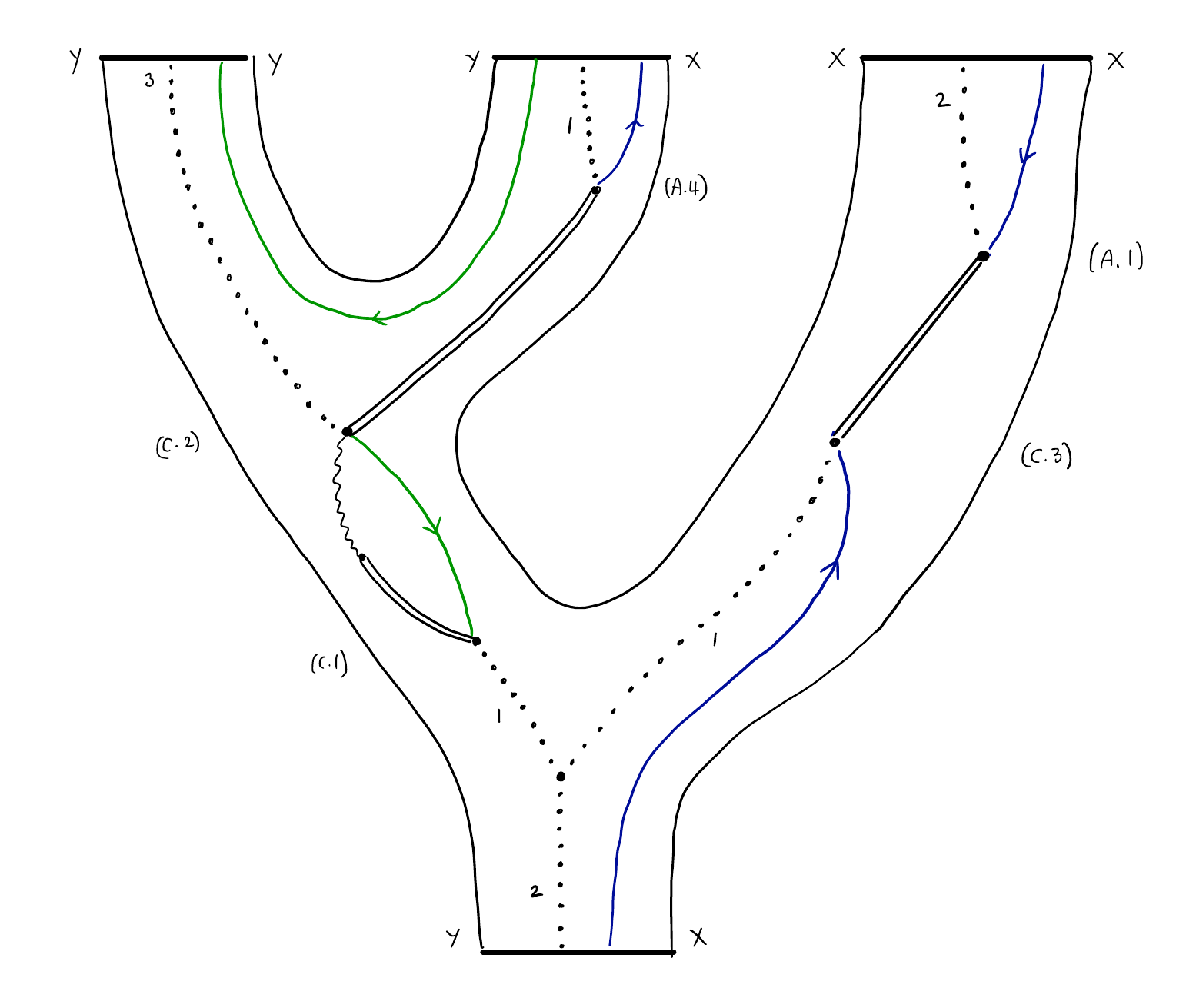}
\end{center}
\centering
\caption{Example of a Feynman diagram, where blue lines denote $\eta$ and green lines $\xi$.}\label{fig:feynman_1}
\end{figure}
\end{example}

\section{The stabilised residue field}\label{section:generator}

In this section we sketch how our approach recovers the usual $A_\infty$-minimal model \cite{seidel_hms, d0904.4713, efimov, sheridan} of $\AA(k^{\stab},k^{\stab})$ when $k$ is a field, as an illustration of an example where $E$ may be split by hand. Let $k$ be a characteristic zero field, $W \in k[x_1,\ldots,x_n]$ a potential and $\AA$ the DG-category with the single object
\[
k^{\stab} = \big( \bigwedge F_\xi \otimes R, \sum_{i=1}^n x_i \xi_i^* + \sum_{i=1}^n W^i \xi_i \big)
\]
for some chosen decomposition $W = \sum_{i=1}^n x_i W^i$ where $F_\xi = \bigoplus_{i=1}^n k \xi_i$. Hence
\be
\AA(k^{\stab}, k^{\stab}) = \Big( \End_k\big(\bigwedge F_\xi\big) \otimes R, \sum_{i=1}^n x_i [\xi_i^*,-] + \sum_{i=1}^n W^i [\xi_i,-] \Big)\,.
\ee
This Koszul matrix factorisation is a classical generator of the homotopy category of matrix factorisations over $k\llbracket \bold{x} \rrbracket$ as was shown first by Schoutens in the setting of maximal Cohen-Macaulay modules \cite{schoutens} and then rediscovered by Orlov \cite{orlov} (see \cite[Lemma 12.1]{seidel_hms}) Dyckerhoff \cite[Corollary 5.3]{d0904.4713} and others \cite[Proposition A.2]{keller}. In Setup \ref{setup:overall} we take:
\begin{itemize}
\item $\bold{t} = (x_1,\ldots,x_n)$ so $R/I = k$.
\item $\lambda = \xi_i$ meaning as usual $\xi_i \wedge (-)$.
\item $\sigma: k \lto R$ is the inclusion of scalars, and $\nabla = \sum_{i=1}^n \partial_{x_i} \theta_i$.
\end{itemize}
It can be shown that
\[
e^{-\delta} r_2 = r_2 e^{\Xi} ( e^{-\delta} \otimes e^{-\delta} )
\]
where $\Xi = \sum_{i=1}^n \theta_i^* \otimes [\xi_i, -]$, and hence in the operator decorated trees computing the higher operations we can remove all the occurrences of $e^{\delta}, e^{-\delta}$ at the cost of replacing $r_2$ at each internal vertex by $r_2 e^{\Xi}$. Under the isomorphism $\rho$ of Lemma \ref{lemma:iso_rho}
\be\label{eq:rho_now}
\xymatrix@C+2pc{
\bigwedge F_\xi \otimes \bigwedge F_{\Bar{\xi}} \otimes R \ar[r]^-{\rho}_-{\cong} & \End_k\big( \bigwedge F_\xi \big) \otimes R
}
\ee
the differential $d_{\AA}$ on the right hand corresponds on the left hand side to
\be\label{eq:daa_stab}
d^{\rho}_{\AA} = \sum_{i=1}^n x_i \xi_i^* + \sum_{i=1}^n W^i \Bar{\xi}_i^*
\ee
by Lemma \ref{lemma:commutators_on_rho}. Moreover under $\rho$ the operator $\Xi$ corresponds to $\theta_i^* \otimes \Bar{\xi}_i^*$, so $e^{\Xi}$ is a family of interaction vertices that allows a downward travelling $\theta_i$ on the left branch to convert into an upward travelling $\xi$ on the right. We identify $\BB(k^{\stab}, k^{\stab})$ with $\bigwedge F_\xi \otimes \bigwedge F_{\Bar{\xi}}$ with zero differential. The Atiyah classes $\At_i$ as operators on $\BB$ are
\[
\gamma_j = \At_j = [d_{\AA}^{\rho},\partial_{x_j}] = -\xi_j^* - \sum_{i=1}^n \partial_{x_j}(W^i)\Big|_{\bold{x} = \bold{0}} \Bar{\xi}_i^*\,.
\]
Together with $\gamma_j^\dagger = - \xi_j$ these Atiyah classes form a representation of the Clifford algebra, which determines a subspace $F_W$ such that 
\[
\operatorname{Im}(E_1) = \bigcap_{i=1}^n \Ker(\gamma_i) = \bigwedge F_W \subseteq \bigwedge F_\xi \otimes \bigwedge F_{\Bar{\xi}}\,.
\]
For example, if $W \in (x_1,\ldots,x_n)^3$ so that $\gamma_j = - \xi_j^*$ then $\bigwedge F_W = \bigwedge F_{\Bar{\xi}}$. This subspace is closed under the higher operations on $\BB$ since neither $\vAt_{\AA}$ nor $\Xi$ can introduce a $\xi_i$. It follows that $\bigwedge F_{\Bar{\xi}}$ equipped with the restricted operations is a minimal $A_\infty$-category which splits the idempotent $E$.

\appendix

\section{Formal tubular neighborhoods}\label{section:formaltub}

In this section we introduce the geometric content of the strong deformation retract which forms the basis of this paper, based on the idea of a formal tubular neighborhood \cite{cuntzquillen, lipman}. We begin with a brief introduction to quasi-regular sequences, for more on which see \cite[\S 15.B]{matsumura}, \cite[Chapitre $0$ \S 15.1]{EGA4} and \cite[Section\,10.68]{stacks_project}.

\begin{definition} A sequence $t_1,\ldots,t_n$ in a commutative ring $R$ is \emph{quasi-regular} if, writing $I = (t_1,\ldots,t_n)$, the morphism of $R/I$-algebras
\begin{gather*}
\phi: R/I[z_1,\ldots,z_n] \lto \operatorname{gr}_I R = \bigoplus_{i \ge 0} I^i / I^{i+1}\\
\phi(z_i) = \overline{t_i} \in I/I^2
\end{gather*}
is an isomorphism. In particular, this means that $I/I^2 \cong \bigoplus_{i=1}^n R/I \cdot \overline{t_i}$.
\end{definition}

We recall the motivation for this definition from algebraic geometry.

\begin{remark} Let $Y$ be a Noetherian scheme and let $i: X \lto Y$ be a closed subscheme with ideal sheaf $\mathscr{I}$. The first-order deformations of $X$ in $Y$ are controlled \cite[Theorem VI-29]{eisenbudharris} by the normal sheaf $\mathscr{N}_{X/Y} = \Hom_{\cat{O}_X}(\mathscr{I}/\mathscr{I}^2, \cat{O}_X)$.

In general, the full description of the space of normal directions to $X$ in $Y$ requires more than just the normal sheaf. The correct approach is to start with the blowup $\pi$ of $Y$ along $X$ (which is described by a universal property) and then look at the closed subscheme of that blowup induced by $X$, as in the diagram:
\[
\xymatrix@C+2pc{
\operatorname{Proj}_X( \oplus_{i \ge 0} \mathscr{I}^i/\mathscr{I}^{i+1} ) \ar[r]^-j\ar[d] & \operatorname{Proj}_Y( \oplus_{i \ge 0} \mathscr{I} ) \ar[d]^-{\pi}\\
X \ar[r]_-{i} & Y\,.
}
\]
The closed subscheme $j$ of the blowup (called the strict transform of $X$) is given by the relative Proj of the sheaf of graded algebras $\oplus_{i \ge 0} \mathscr{I}^i/\mathscr{I}^{i+1}$ on $X$, and its points give the correct notion of a point on $X$ together with a normal direction in $Y$. For this reason the scheme $C_X(Y) = \Spec_X( \oplus_{i \ge 0} \mathscr{I}^i / \mathscr{I}^{i+1} )$, whose projectivisation is the strict transform of $X$ in the blowup, is called the \emph{normal cone} of $X$ in $Y$.
\end{remark}

Thus, to say that a sequence $t_1,\ldots,t_n$ is quasi-regular is to say that the projectivised normal cone of $X = \Spec(R/I)$ in $Y = \Spec(R)$ is the space of lines in the normal bundle, since the definition of quasi-regularity gives 
\begin{align*}
\Spec_{Y}( \operatorname{Sym}( \mathscr{N}^*_{X/Y} ) ) &\cong \Spec( \operatorname{Sym}_{R/I}( I/I^2 ) )\\
&= \Spec( R/I[z_1,\ldots,z_n] )\\
&\cong \Spec( \operatorname{gr}_ I R)\\
&= C_X(Y)\,.
\end{align*}
Recall that in differential geometry if we are given a submanifold $X$ of a smooth manifold $Y$ the tubular neighborhood theorem \cite[\S 4.5]{hirsch} identifies an open neighborhood of the zero section of the normal bundle $N_{X/Y} \lto X$ with an open neighborhood of $X$ in $Y$. 

There is no direct analogue of this identification of neighborhoods for a general closed immersion $i: X \lto Y$ of schemes. Let us discuss what such an identification would mean at the level of \emph{formal} neighborhoods in the affine case, assuming that $I = (t_1,\ldots,t_n)$ is generated by a quasi-regular sequence. Completing the normal cone $C_X(Y)$ along the zero section means completing $R/I[z_1,\ldots,z_n]$ in the $(z_1,\ldots,z_n)$-adic topology, while completing $Y$ along $X$ means taking the $I$-adic completion of $R$, so that such an identification would amount to an isomorphism of topological rings
\be\label{eq:formaltubular}
R/I\llbracket z_1,\ldots,z_n \rrbracket \lto \widehat{R}
\ee
where $\widehat{R}$ denotes the $I$-adic completion. It easy to produce examples where this fails:

\begin{example}\label{eq:example_nontub} Let $k$ be a field, $R = k[x]$ and $I = (x^d)$ for $d > 1$. Since the $I$-adic topology is the same as the $(x)$-adic topology, $R/I\llbracket z \rrbracket = k[x]/(x^d) \llbracket z \rrbracket \neq k\llbracket x \rrbracket = \widehat{R}$ since one ring is reduced while the other is not.
\end{example}

However, if $R/I$ is smooth \cite[Definition 28.D]{matsumura}, so that we are in a situation more resembling the one in differential geometry, there is indeed an isomorphism of topological rings of the form \eqref{eq:formaltubular} (see e.g. Lemma \ref{prop_algtube} below). This \emph{formal tubular neighborhood theorem} was developed in the noncommutative setting by Cuntz-Quillen \cite[Theorem 2]{cuntzquillen}. 

In our applications, however, Example \ref{eq:example_nontub} is more typical, as $I$ is generated by the partial derivatives of a potential $W$ and the critical locus $R/I$ of a potential is rarely smooth. While in general there is no strong analogue of the tubular neighborhood theorem, Lipman proves in \cite{lipman} that the cases we care about, there is still an isomorphism of $k\llbracket z_1,\ldots,z_n \rrbracket$-modules, and this is enough to produce connections.

\begin{setup}\label{setup:connection_t} For the rest of this section let $k$ be a commutative $\mathbb{Q}$-algebra, $R$ a $k$-algebra, and let $t_1,\ldots,t_n$ be a quasi-regular sequence in $R$ such that, writing $I = (t_1,\ldots,t_n)$, the quotient $R/I$ is a finitely generated projective $k$-module.
\end{setup}

\begin{lemma}\label{prop_algtube} Any $k$-linear section $\sigma$ of the quotient $\pi: R \lto R/I$ induces an isomorphism of $k\llbracket z_1,\ldots,z_n \rrbracket$-modules
\[
\sigmastar: R/I \otimes k\llbracket z_1,\ldots,z_n \rrbracket \lto \widehat{R}
\]
where $\widehat{R}$ denotes the $I$-adic completion. If further $R/I$ is smooth over $k$ then there exists a section $\sigma$ such that $\sigmastar$ is an isomorphism of topological $k$-algebras.
\end{lemma}
\begin{proof}
This is \cite[Lemma 3.3.2]{lipman}. The map $\sigmastar$ is induced from $\sigma$ by extension of scalars, where $k\llbracket \bold{z} \rrbracket$ acts on $\widehat{R}$ by making $z_i$ act as multiplication by $t_i$. The main point is that every $r \in \widehat{R}$ has a \emph{unique} representation as a power series
\be\label{eq:uniq_decompr}
r = \sum_{M \in \mathbb{N}^n} \sigma(r_M) t^M
\ee
for elements $r_M \in R/I$, and the inverse to $\sigmastar$ sends $r$ to $\sum_M r_M z^M$. The uniqueness of \eqref{eq:uniq_decompr} is a consequence of quasi-regularity, while the existence is proven as follows: given $r \in R$ we have $\pi( r - \sigma(r) ) = 0$ and so there exist $a_1,\ldots,a_n$ with
\be\label{eq:uniq_decompr_choices}
r = \sigma(r) + \sum_{i=1}^n a_i t_i\,.
\ee
Applying the same argument to each $a_i$ yields 
\begin{align*}
r &= \sigma(r) + \sum_{i=1}^n \big[ \sigma(a_i) + \sum_{j=1}^n a_{ij} t_j \big] t_i\\
&= \sigma(r) + \sum_{i=1}^n \sigma(a_i) t_i + \sum_{i,j=1}^n a_{ij} t_i t_j\,.
\end{align*}
This process converges in the $I$-adic topology to a series \eqref{eq:uniq_decompr}. If $R/I$ is smooth then by a standard argument \cite{matsumura} we can produce a section $\sigma: R/I \lto \widehat{R}$ which is a morphism of $k$-algebras. It is clear then that $\sigmastar$ is an isomorphism of topological algebras.
\end{proof}

Recall from \cite[\S 8.1.1]{loday} the notion of a connection on a module. As alluded to above, we can use the isomorphism of Lemma \ref{prop_algtube} to produce connections. For a more complete discussion on this point, see \cite[Appendix B]{pushforward}.

\begin{corollary}\label{corollary:nabla_sigma} Associated to any $k$-linear section $\sigma$, there is a $k$-linear connection on $\widehat{R}$ as a $k[\bold{z}]$-module, where $z_i$ acts as multiplication by $t_i$:
\be
\nabla_\sigma: \widehat{R} \lto \widehat{R} \otimes_{k[\bold{z}]} \Omega^1_{k[\bold{z}]/k}\,.
\ee
\end{corollary}
\begin{proof}
The usual partial derivatives give a $k$-linear connection on $k\llbracket \bold{z} \rrbracket$ as a $k[\bold{z}]$-module which extends, by Lemma \ref{prop_algtube}, to a connection on $\widehat{R}$, see \cite[\S 8.1.3]{loday}. In terms of the power series representation \eqref{eq:uniq_decompr} the connection is given by
\be\label{eq:explicit_connection}
\nabla_\sigma(r) = \sum_{j=1}^n \sum_{M \in \mathbb{N}^n} M_j \sigma(r_M) t^{M - e_j} \otimes dz_j
\ee
which completes the proof.
\end{proof}

\begin{definition}
With a section $\sigma$ fixed, we will write $\nabla$ for the associated connection $\nabla_\sigma$. We also introduce $k$-linear operators $\frac{\partial}{\partial z_i}: \widehat{R} \lto \widehat{R}$ by the identity $\nabla = \sum_{j=1}^n \frac{\partial}{\partial z_i} dz_j$. The operators $\frac{\partial}{\partial z_i}$ depend on the choice of section $\sigma$, but we will abuse notation and write
\[
\partial_{t_i} := \frac{\partial}{\partial z_i}: \widehat{R} \lto \widehat{R}\,.
\]
\end{definition}

\begin{example}\label{example:dt_xd} Let $k$ be a field, $R = k[x]$ and $I = (x^d)$, with $t = x^d$. Choose the $k$-linear section $\sigma: R/I \lto R$ defined by $\sigma(x^i) = x^i$ for $0 \le i \le d - 1$. Then
\[
\partial_t( x^2 + x^{d+1} ) = \partial_t( x^2 \cdot 1 + x \cdot x^d ) = x\,.
\]
\end{example}

In an important class of examples there is an explicit algorithm for computing $r_M$ and thus the derivatives $\partial_{t_i}(r)$ for any element $r \in R$.

\begin{remark}\label{remark:grobner} Suppose that $k$ is a characteristic zero field, and $R = k[x_1,\ldots,x_n]$. Choose a monomial ordering for $R$ and let $G = (g_1,\ldots,g_c)$ be a Gr\"obner basis for $I$, which may be computed from the polynomials $t_1,\ldots,t_n$ by Buchberger's algorithm \cite[\S 2.7]{cox_little_oshea}. Moreover this algorithm also computes an expression $g_i = \sum_{j=1}^n h_{ij} t_j$ for $1 \le i \le c$.

Following the notation of \cite{cox_little_oshea} we write $\Bar{r}^G$ for the remainder on division of $r$ by $G$. This is the unique polynomial with no term divisible by any of the leading terms of the $g_i$, and for which there exists $h \in I$ with $r = h + \Bar{r}^G$ \cite[\S 2.6]{cox_little_oshea}. The generalised Euclidean division algorithm produces $\Bar{r}^G$ together with polynomials $b_1,\ldots,b_c$ such that
\be\label{eq:uniq_decompr_2}
r = \Bar{r}^G + \sum_{i=1}^c b_i g_i = \Bar{r}^G + \sum_{j=1}^n \Big[ \sum_{i=1}^c b_i h_{ij} \Big] t_j\,.
\ee
It is easy to check that the function
\[
\sigma: R/I \lto R\,, \qquad \sigma(r) = \Bar{r}^G
\]
is well-defined and is a $k$-linear section of $\pi$ \cite[Corollary 2, \S 2.6]{cox_little_oshea}. This means that \eqref{eq:uniq_decompr_2} is the desired expression in \eqref{eq:uniq_decompr_choices} used to generate the $r_M$'s. So for example in the case $M = \bold{0} \in \mathbb{N}^n$ we take $r_{\bold{0}} = \Bar{r}^G$ and for $1 \le j \le n$ in the case $M = e_j$ we have
\[
r_{e_j} = \sum_{i=1}^c \overline{b_i h_{ij} }^G\,.
\]
Observe that the monomials $x^\alpha$ with $\alpha$ not divisible by any of the leading terms of the $g_i$ give a $k$-basis of $R/I$. Denote the set of such tuples $\alpha$ by $\Lambda$. Then $R/I = \oplus_{\alpha \in \Lambda} k x^\alpha$ and $\sigma(x^\alpha) = x^\alpha$, generalising Example \ref{example:dt_xd}. 

What we have already said computes the coefficients of $r_M$ in the basis $\{x^\alpha\}_{\alpha \in \Lambda}$ in the cases where $|M| \le 1$. For the cases where $|M| = 2$ we simply have to run the division algorithm with $\sum_{i=1}^c b_i h_{ij}$ in place of $r$ in \eqref{eq:uniq_decompr_2}. Proceeding in this way constitutes an algorithm for computing the coefficient of $x^\alpha$ in $r_M$ for any $M \in \mathbb{N}^n$.
\end{remark}

The upshot is that when $k$ is a field, the map
\[
(\sigma_{\bold{t}})^{-1}: \widehat{R} \lto R/I \otimes k \llbracket \bold{t} \rrbracket
\]
sends $r \in R$ to a power series $\sum_{M \in \mathbb{N}^n} r_M t^M$ whose coefficients $r_M$ are obtained by \emph{iterated Euclidean division} of $r$ by a Gr\"obner basis of $I$. When $k$ is not a field, it is not obvious how to calculate the $r_M$ algorithmically.

\subsection{Analogy to the Euler field}\label{section:more_on_connections}

In this section we are in the situation of Setup \ref{setup:connection_t}. The Koszul complex of the sequence $t_1,\ldots,t_n$ over $\widehat{R}$ is $(K, d_K) = \big( \bigwedge( k \theta_1 \oplus \cdots \oplus k \theta_n ) \otimes \widehat{R}, \,\,\sum_i t_i \theta_i^* \big)$. By identifying $\theta_i$ with $dz_i$ we identify $K$ with $\widehat{R} \otimes_{k[\bold{z}]} \Omega^*_{k[\bold{z}]/k}$. For the remainder of the section we fix a section $\sigma$ and associated connection $\nabla = \nabla_\sigma$. We have the following $k$-linear operator on $\widehat{R}$:
\be\label{eq:explicitdknabla}
d_K \nabla_\sigma = \sum_i t_i \frac{\partial}{\partial t_i}\,.
\ee
Using \eqref{eq:explicit_connection} we have explicitly $d_K \nabla_\sigma(r) = \sum_{M \in \mathbb{N}^n} |M| \sigma(r_M) t^M$ where $|M| = \sum_{j=1}^n M_j$.

\begin{lemma} If $R/I$ is smooth, $\sigma$ may be chosen such that $d_K \nabla_\sigma$ is a derivation of $\widehat{R}$.
\end{lemma}
\begin{proof}
Suppose $\sigma: R/I \lto \widehat{R}$ is an algebra morphism. Then for $r,s \in R$ we have
\[
\sum_M (rs)_M z^M = \sum_M \sum_{A + B = M} r_As_B z^M\,,
\]
and hence $(rs)_M = \sum_{A+B = M} r_A s_B$. From this the claim follows.
\end{proof}

Recall that for a submanifold $X \subseteq Y$, any tubular neighborhood $f: N_{X/Y} \lto Y$ of $X$ in $Y$ determines an open neighborhood $U = f(N_{X/Y})$ of $X$ and on this open neighborhood a vector field $\mathscr{F} = f_*(\mathscr{E})$ induced by the Euler field $\mathscr{E}$ on $N_{X/Y}$, which is the infinitesimal generator of the scaling action of $\mathbb{R}$ in the fibers. In local coordinates $\mathscr{E} = \sum_i x_i \frac{\partial}{\partial x_i}$ where $x_i$ are the coordinates in the fiber directions. Conversely, the integral curves of the vector field $\mathscr{F}$ determine the tubular neighborhood $f$ \cite[\S 2.3]{burs} and, in particular, $\mathscr{F}$ gives rise to a deformation retract of $U$ onto $X$.

When $R/I$ is smooth, the derivation $d_K \nabla_\sigma$ is in light of \eqref{eq:explicitdknabla} clearly analogous to $\mathscr{F}$. In general $d_K \nabla_\sigma$ will \emph{not} be a derivation, but nonetheless the operators $d_K, \nabla_\sigma$ still give rise to a deformation retract of $\widehat{R}$ onto $R/I$ in a sense that we will make precise below.

\begin{example}\label{example:dt_xd2} In the situation of Example \ref{example:dt_xd} observe that $d_K \nabla = t \frac{\partial}{\partial t}$ is \emph{not} a derivation, as for example $t\frac{\partial}{\partial t}( x^d ) = t$ but $x t \frac{\partial}{\partial t}(x^{d-1}) + t \frac{\partial}{\partial t}(x) x^{d-1} = 0$.
\end{example}

Let us now examine the way in which, in the smooth setting, the vector field $\mathscr{F}$ gives rise to idempotents. The Lie derivative $[-, \mathscr{F}]$ gives rise to an idempotent linear operator on the bundle $TY|_X$ which is a splitting of the exact sequence 
\[
0 \lto TY \lto TX|_Y \lto N_{X,Y} \lto 0\,.
\]
While the Lie derivative with $\mathscr{F}$ gives an idempotent operator on vector fields, it is not idempotent on smooth functions. However, a rescaling of its action on smooth functions does give an idempotent, the corresponding algebra being the following:

\begin{remark} Consider the following diagram of $k$-linear operators
\[
\xymatrix@C+3pc@R+2pc{
\widehat{R} \otimes_{k[\bold{z}]} \Omega^2_{k[\bold{z}]/k} \ar@<0.8ex>[r]^-{d_K} &
\widehat{R} \otimes_{k[\bold{z}]} \Omega^1_{k[\bold{z}]/k} \ar@<0.8ex>[l]^-{\nabla^1_\sigma}\ar@<0.8ex>[r]^-{d_K} & \widehat{R} \ar@<0.8ex>[l]^-{\nabla^0_\sigma} \ar@<0.8ex>[r]^-{\pi} & R/I \ar@<0.8ex>[l]^-{\sigma}
}
\]
where $\pi$ is the quotient map, and $\nabla^*_\sigma$ denotes the extension of $\nabla_\sigma$ to an operator on $\widehat{R} \otimes_{k[\bold{z}]} \Omega^*_{k[\bold{z}]/k}$. It is easy to check that
\[
\left( d_K \nabla^1_\sigma + \nabla^0_\sigma d_K \right)\left( \sum_N \sigma(s_N)t^N \otimes dz_j \right) = \sum_N (1 + |N|) \sigma(s_N) t^N \otimes dz_j\,,
\]
and hence $d_K \left( d_K \nabla^1_\sigma + \nabla^0_\sigma d_K \right)^{-1} \nabla^0_\sigma(r) = r - \sigma(r_0)$. So while $d_K \nabla^0_\sigma$ is not an idempotent operator on $\widehat{R}$, the rescaled operator $d_K \left( d_K \nabla^1_\sigma + \nabla^0_\sigma d_K \right)^{-1} \nabla^0_\sigma$ is an idempotent, which splits $\widehat{R}$ as a $k$-linear direct sum $\widehat{R} \cong R/I \oplus I$.
\end{remark}

\section{Proofs}\label{section:proofs}

In this section we prove Theorem \ref{theorem:main_ainfty_products} and Theorem \ref{theorem:homotopy_clifford}. All notation is as in Setup \ref{setup:overall}, and more generally as in Section \ref{section:the_model}. For example $\AA$ denotes the DG-category of matrix factorisations of $W$ and $\AA_{\theta}$ is the extension of \eqref{eq:defn_AAtheta}. The proofs mostly consist of applying the techniques already developed in \cite{pushforward, cut} using homological perturbation to the situation at hand, but there are two technical points worth noting:
\begin{itemize}
\item[(i)] In \cite{cut} the quasi-regular sequence $t_1,\ldots,t_n$ is always $\partial_{x_1} W, \ldots, \partial_{x_n} W$.
\item[(ii)] In \cite{cut} there is an assumption that $k$ is Noetherian.
\end{itemize}
We explain in Appendix \ref{section:noetherian} why the Noetherian hypothesis is not necessary. Regarding (i), we reiterate the relevant parts of \cite{cut} below in the current generality (that is, the hypotheses of Setup \ref{setup:overall}) where $t_i$ is not necessarily $\partial_{x_i} W$. The upshot is that the only place where the hypothesis that $t_i = \partial_{x_i} W$ plays any meaningful role is in the final form of the Clifford operator $\gamma_i^\dagger$ and we address this explicitly in the proof of Theorem \ref{theorem:homotopy_clifford}.

The aim is to put a strict homotopy retract \cite[\S 3.3]{lazaroiu} on $\AA_\theta$, where
\[
\AA_\theta(X,Y) = \bigwedge F_\theta \otimes \Hom_R(X,Y) \otimes_R \widehat{R}\,.
\]
The $I$-adic completion $\widehat{R}$ of $R$ has $t_1,\ldots,t_n$ as a quasi-regular sequence and $\widehat{R}/I \widehat{R} \cong R/I$. By \cite[Appendix B]{pushforward} there is a standard flat $k$-linear connection
\[
\nabla^0: \widehat{R} \lto \widehat{R} \otimes_{k[\bold{t}]} \Omega^1_{k[\bold{t}]/k}
\]
where $k[\bold{t}]$ denotes the polynomial ring in variables $t_i$. Let $X,Y \in \ob(\AA)$. We run through the steps of \cite[\S 4.3]{cut} with $\bigwedge F_\theta$ our new notation for $S_m$, and the complex $\Hom_R(X,Y) \cong X^{\vee} \l Y$ replacing $Y \l X$ (see \cite[\S 4.5]{cut} which also discusses this special case). The role of $k[\bold{y}]$ in \emph{loc.cit} is now played by $R = k[\bold{x}]$. In the first step $\nabla^0$ extends to a $k$-linear operator $\nabla$ on $\widehat{R} \otimes_{k[\bold{t}]} \Omega^*_{k[\bold{t}]/k}$. Choosing a homogeneous $R$-basis for $X,Y$ and taking the induced basis on $\Hom_R(X,Y)$ over $R$, and extending $\nabla$, we get a $k$-linear splitting homotopy
\[
H = [d_K, \nabla]^{-1} \nabla \quad \text{ on the complex } \quad \xymatrix{ \big( K \otimes_R \Hom_R(X,Y) \otimes_R \widehat{R}, d_K \big) }
\]
where $(K,d_k) = \big( \bigwedge F_\theta \otimes R, \sum_{i=1}^n t_i \theta_i^* \big)$ and we identify $\theta_i$ with $dt_i$. This splitting homotopy corresponds to the strong deformation retract (with homotopy $H$)
\[
\xymatrix@C+3pc{
\Big(R/I \otimes_R \Hom_R(X,Y),0\Big) \ar@<-1ex>[r]_-\sigma & \Big( K \otimes_R \Hom_R(X,Y) \otimes_R \widehat{R}, \,\,d_K = \sum_i t_i \theta_i^*\, \Big)\,, \ar@<-1ex>[l]_-\pi
}
\]
In the second step we view $d_{\Hom}$, the differential on $\Hom_R(X,Y)$, as a perturbation and we learn that
\be\label{eq:phi_infty_appendix}
\phi_\infty = \sum_{m \ge 0} (-1)^m (H d_{\Hom})^m H
\ee
is a $k$-linear splitting homotopy, to which is associated the following $k$-linear strong deformation retract of complexes (with homotopy $\phi_{\infty}$)
\[
\xymatrix@C+3pc@R+3pc{
\Big( R/I \otimes_R \Hom_R(X,Y), d_{\Hom} \Big) \ar@<-1ex>[r]_-{\sigma_\infty} &
\Big( K \otimes_R \Hom_R(X,Y) \otimes_R \widehat{R}, d_K + d_{\Hom} \Big) \ar@<-1ex>[l]_-{\pi}
}
\]
where
\be\label{eq:sigma_infty_appendix}
\sigma_\infty = \sum_{m \ge 0} (-1)^m (H d_{\Hom})^m \sigma\,.
\ee
In the third step, since each $t_i$ acts null-homotopically on $\Hom_R(X,Y)$ we have an isomorphism of complexes over $R$
\[
\xymatrix@C+3pc{
\Big( K \otimes_R \Hom_R(X,Y) \otimes_R \widehat{R}, d_K + d_{\Hom} \Big) \ar@<-1ex>[r]_-{e^{\delta}} & \Big( \bigwedge F_\theta \otimes \Hom_R(X,Y) \otimes_R \widehat{R}, d_{\Hom} \Big)\,, \ar@<-1ex>[l]_-{e^{-\delta}}
}
\]
where $\delta = \sum_{i=1}^n \lambda_i \theta_i^*$. Note we do \emph{not} assume that $\lambda_i$ is a partial derivative of $d_X$. In the fourth step the canonical map $\varepsilon: \Hom_R(X,Y) \lto \Hom_R(X,Y) \otimes_R \widehat{R}$ is by Lemma \ref{lemma:completion_he} (based on \cite[Remark 7.7]{pushforward}) a homotopy equivalence over $k$. Hence we have homotopy equivalences of $k$-complexes, combining the above

\begin{center}
\begin{tabular}{ >{\centering}m{9cm} >{\centering}m{1cm}}
\xymatrix@C+3pc@R+3pc{
\Big( \bigwedge F_\theta \otimes \Hom_R(X,Y), d_{\Hom} \Big) \ar[d]^-{\varepsilon}\\
\Big( \bigwedge F_\theta \otimes \Hom_R(X,Y) \otimes_R \widehat{R}, d_{\Hom} \Big)  \ar@<-2ex>[d]_-{e^{-\delta}}\\
\Big( K \otimes_R \Hom_R(X,Y) \otimes_R \widehat{R}, d_K + d_{\Hom} \Big) \ar@<-2ex>[u]_-{e^{\delta}}\ar@<-2ex>[d]_-{\pi}\\
\Big( R/I \otimes_R \Hom_R(X,Y), \overline{d_{\Hom}} \Big) \ar@<-2ex>[u]_-{\sigma_\infty}
}
&
\tagarray{\label{eq_4_1}}
\end{tabular}
\end{center}
Note that $R/I \otimes_R \Hom_R(X,Y)$ is by our hypotheses a $\nZ_2$-graded complex of finite rank free $k$-modules. Next we argue that \eqref{eq_4_1} gives a strict homotopy retraction of $\AA_{\theta}$. With this in mind the overall content of \eqref{eq_4_1} is a $k$-linear homotopy equivalence
\[
\xymatrix@C+3pc{
\big( \bigwedge F_\theta \otimes \Hom_R(X,Y) \otimes_R \widehat{R}, d_{\Hom} \big) \ar@<-1ex>[r]_-{\Phi} & \big( R/I \otimes_R \Hom_R(X,Y), \overline{d_{\Hom}}\big)\,, \ar@<-1ex>[l]_-{\Phi^{-1}}
}
\]
where we have $\Phi \circ \Phi^{-1} = 1$ and $\Phi^{-1} \circ \Phi = 1 - [d_{\Hom}, \widehat{H}]$ where
\begin{align}
\Phi &= \pi \circ e^{-\delta}\,,\\
\Phi^{-1} &= e^{\delta} \circ \sigma_\infty\,,\\
\widehat{H} &= e^{\delta} \circ \phi_{\infty} \circ e^{-\delta}\,.
\end{align}
By construction we have:

\begin{lemma} The data $(P,G)$ consisting of
\begin{gather*}
P_{X,Y} = \Phi^{-1} \circ \Phi: \AA_\theta(X,Y) \lto \AA_\theta(X,Y)\\
G_{X,Y} = \widehat{H}: \AA_\theta(X,Y) \lto \AA_\theta(X,Y)
\end{gather*}
form a strict homotopy retract on the DG-category $\AA_{\theta}$ in the sense of \cite[\S 3.3]{lazaroiu}.
\end{lemma}

\begin{proof}[Proof of Theorem \ref{theorem:main_ainfty_products}] 
By homological perturbation \cite[\S 3.3, p.33]{lazaroiu} (see also \cite[\S I]{seidel} and particularly \cite[Remark 1.15]{seidel}) applied to the strict homotopy retract of the lemma, we obtain forward suspended $A_\infty$-products $\{ \rho_k \}_{k \ge 1}$ defined on the family of spaces
\[
\xymatrix@C+2pc{
\BB(X,Y) = R/I \otimes_R \Hom_R(X,Y) \ar[r]_-{\cong}^-{\Phi^{-1}} & \operatorname{Im}(P_{X,Y})
}
\]
which make $(\BB, \rho)$ into an $A_\infty$-category. These higher products have the given description in terms of sums over trees decorated by $\sigma_\infty, \phi_\infty, e^{\delta}, e^{-\delta}, r_2$ but the description of $\sigma_\infty, \phi_\infty$ used in Definition \ref{defn:important_operators} differs from that given in \eqref{eq:sigma_infty_appendix}, \eqref{eq:phi_infty_appendix} above and we have to address this discrepancy. This uses an argument that first appeared in a slightly different form in \cite[(10.3),(10.4)]{pushforward}. We set $\zeta = [d_K, \nabla]^{-1}$ so that
\[
\sigma_\infty = \sum_{m \ge 0} (-1)^m (\zeta \nabla d_{\Hom})^m \sigma\,.
\]
Now $\nabla \zeta = \zeta \nabla$, $\nabla^2 = 0$ and $\nabla \sigma = 0$ hence we can replace any string $\zeta \nabla d_{\Hom} \cdots \zeta \nabla d_{\Hom} \sigma$ by the string $\zeta [d_{\Hom}, \nabla] \cdots \zeta [d_{\Hom}, \nabla] \sigma$ and so
\be\label{eq:rewritten_sigma_infty}
\sigma_\infty = \sum_{m \ge 0} (-1)^m \Big( \zeta [\nabla, d_{\Hom}] \Big)^m \sigma = \sum_{m \ge 0} (-1)^m (\zeta \vAt_{\AA})^m \sigma \,.
\ee
In the same way we obtain the form of $\phi_\infty$ given in Definition \ref{defn:important_operators}. Finally, the fact that $\zeta$ may be computed by the formula of Definition \ref{definition:zeta} is easily checked.

From the theory of homological perturbation we also obtain $A_\infty$-functors $F,G$ with $F_1 = \Phi, G_1 = \Phi^{-1}$ and an $A_\infty$-homotopy $G \circ F \simeq 1$, see \cite{markl_transfer}. It remains to check strict unitality. Clearly $\AA_{\theta}$ is strictly unital, since it is a DG-category. Now we apply \cite[p.37]{lazaroiu} which shows that $\BB$ is strictly unital, provided $P_{X,Y} \circ G_{X,Y} = 0$ and $G_{X,X}(u_X) = 0$ for all units $u_X$. The first condition follows because we have a strong deformation retract (not just a strict homotopy retract in the sense of \cite{lazaroiu}) and the second condition because
\[
u_X = 1 \otimes \operatorname{id}_{X,X} \in \bigwedge F_\theta \otimes \Hom_R(X,X) \otimes_R \widehat{R}
\]
and hence $G_{X,X}(u_X) = 0$ since $e^{-\delta}(1 \otimes \operatorname{id}_{X,X}) = 1 \otimes \operatorname{id}_{X,X}$ and $\nabla(1) = 0$.
\end{proof}

\begin{proof}[Proof of Theorem \ref{theorem:homotopy_clifford}] We have $F_1 = \Phi, G_1 = \Phi^{-1}$ so that $\gamma_i = \Phi \theta_i^* \Phi^{-1}, \gamma^\dagger_i = \Phi \theta_i \Phi^{-1}$. These transfers were studied in \cite{cut} quite generally, but the final formulas were given only in the special case where $t_i = \partial_{x_i} W$, and our job here is to recapitulate the argument to the degree that is necessary to derive the general formula and exhibit that everything we need from \cite{cut} works in full generality.

Firstly observe that since $1 \simeq G_1 F_1$ there is a $k$-linear homotopy
\begin{align*}
E_1 &= F_1 e G_1 = F_1 \theta_n^* \cdots \theta_1^* \theta_1 \cdots \theta_n^* G_1 \simeq \gamma_n \cdots \gamma_1 \gamma_1^\dagger \cdots \gamma_n^\dagger\,.
\end{align*}
From \cite[Lemma 4.17]{cut} we get $e^{-\delta} \theta_i^* e^{\delta} = \theta_i^*$ and \cite[Theorem 4.28]{cut} gives
\begin{align*}
e^{-\delta} \theta_i e^{\delta} &= \theta_i - \sum_{m \ge 0} \sum_{q_1,\ldots,q_m} \frac{1}{(m+1)!} \big[ \lambda_{q_m}\,, \big[ \lambda_{q_{m-1}}, \big[ \cdots [ \lambda_{q_1}, \lambda_i ] \cdots \big] \theta^*_{q_1} \cdots \theta^*_{q_m}\\
&= \theta_i - \lambda_i - \frac{1}{2} \sum_q [\lambda_q, \lambda_i] \theta_q^* + \cdots
\end{align*}
Next we copy elements from the proof of \cite[Proposition 4.35]{cut}. Throughout $\At_i$ means $[d_{\Hom}, \partial_{t_i}]$ defined as in \cite{pushforward} with respect to the ring morphism $k[\bold{t}] \lto R$. Note that in the formula \eqref{eq:rewritten_sigma_infty} for $\sigma_\infty$ the Atiyah class $\vAt_{\AA}$ is $k[\bold{t}]$-linear and so, using that as an operator on $\bigwedge F_\theta \otimes R/I \otimes k\llbracket \bold{t} \rrbracket$ the operator $\zeta$ has the property that modulo $(\bold{t}) k\llbracket \bold{t} \rrbracket$
\[
\zeta( \omega \otimes z \otimes f ) = \frac{1}{|\omega|} \omega \otimes z \otimes f \quad \text{(mod $\bold{t}$)}
\]
we may calculate
\begin{align*}
\pi \theta_{q_1}^* \cdots \theta_{q_m}^* \sigma_\infty &= \sum_{m \ge 0} (-1)^m \pi \theta_{q_1}^* \cdots \theta_{q_m}^* \Big\{ \zeta [\nabla, d_{\Hom}] \Big\}^m \sigma\\
&= \sum_{m \ge 0} (-1)^m \frac{1}{m!} \pi \theta_{q_1}^* \cdots \theta_{q_m}^* \Big\{ \sum_{k=1}^n \theta_k [\partial_{t_k}, d_{\Hom}] \Big\}^m \sigma\\
&= \frac{1}{m!} \sum_{\tau \in S_m} (-1)^{|\tau|} \At_{q_{\tau 1}} \cdots \At_{q_{\tau m}}\,.
\end{align*}
So far this is an \emph{equality} of operators on the complex $\BB(X,Y)$, but we can now apply the fact that Atiyah classes anti-commute up to $k$-linear homotopy (the argument of \cite[Theorem 3.11]{cut} applies) to see that
\[
\pi \theta_{q_1}^* \cdots \theta_{q_m}^* \sigma_\infty \simeq \At_{q_1} \cdots \At_{q_m}\,.
\]
To conclude we apply this to the operators $\gamma_i = \pi e^{-\delta} \theta_i^* e^\delta \sigma_\infty = \pi \theta_i^* \sigma_\infty$ and $\gamma_i^\dagger = \pi e^{-\delta} \theta_i e^\delta \sigma_\infty$ exactly as in \cite[Proposition 4.35]{cut}.
\end{proof}

\section{Removing Noetherian hypotheses}\label{section:noetherian}

In \cite{cut} there is a hypothesis that the base ring $k$ is Noetherian. This is not necessary, and we explain how to remove it. Along the way we prove Lemma \ref{lemma:completion_he}. Throughout $k$ is a commutative $\mathbb{Q}$-algebra and $R$ a $k$-algebra, and $t_1,\ldots,t_n$ is a quasi-regular sequence in $R$ with $I = (t_1,\ldots,t_n)$. We write $\widehat{R}$ for the $I$-adic completion. It is well-known that the sequence $t_1,\ldots,t_n$ is quasi-regular in $\widehat{R}$ and that the canonical map $R/I \lto \widehat{R}/I\widehat{R}$ is an isomorphism; see for example \cite[\S 15.B]{matsumura}, \cite[Chapitre $0$ \S 15.1]{EGA4}.

\begin{lemma}\label{lemma:trivialsplit} Suppose that
\begin{itemize}
\item[(i)] $R, R/I$ are projective $k$-modules.
\item[(ii)] The Koszul complex $K$ of $t_1,\ldots,t_n$ over $R$ is exact in nonzero degrees.
\end{itemize}
Then if $\sigma: R/I \lto R$ is any $k$-linear section of the quotient map $\pi: R \lto R/I$ and $\pi$ is is the canonical morphism of complexes $\pi: (K,d_K) \lto (R/I,0)$ then there is a degree $-1$ $k$-linear operator $h$ on $K$ as in the diagram
\[
\xymatrix@C+3pc{
(R/I,0) \ar@<-1ex>[r]_-\pi & (K,d_K)\,, \ar@<-1ex>[l]_-\sigma
}\quad h
\]
such that
\begin{itemize}
\item $\pi \sigma = 1$,
\item $\sigma \pi = 1_K - [d_K, h]$
\item $h^2 = 0, h\sigma = 0, \pi h = 0$.
\end{itemize}
\end{lemma}
\begin{proof}
Since $R$ and $R/I$ are projective over $k$, the Koszul complex may be decomposed into a series of \emph{split} short exact sequences, and $h$ is easily constructed from these splittings with the desired properties.
\end{proof}

Next we observe that \cite[Remark 7.7]{pushforward} holds in greater generality. 

\begin{remark}\label{remark:fixing_dm} Let $\varphi: S \lto R$ be a ring morphism, $W \in S$ and $(X,d)$ a matrix factorisation of $W$ over $R$ (again, not necessarily finite rank), $t_1,\ldots,t_n$ a quasi-regular sequence in $R$ with $t_i \cdot 1_X \simeq 0$, and homotopies $\lambda_i$. We assume that there is a deformation retract (of $\mathbb{Z}$-graded complexes) over $S$
\be\label{eq:fixing_dm_1}
\xymatrix@C+3pc{
(R/\bold{t} R,0) \ar@<-1ex>[r]_-\pi & (K_R(\bold{t}),d_K)\,, \ar@<-1ex>[l]_-\sigma
}\quad h
\ee
satisfying $h^2 = 0, h \sigma = 0, \pi h = 0$ where $\pi$ is the canonical map. By the previous lemma, it suffices for this to assume that $R,R/I$ are projective over $k$ and that $K_R(\bold{t}) = K$ is exact except in degree zero. Let $\alpha: R \lto R'$ be any ring morphism such that
\begin{itemize}
\item $t_1,\ldots,t_n$ is quasi-regular in $R'$
\item the induced map $R/\bold{t} R \lto R'/\bold{t} R'$ is an isomorphism
\item there is a deformation retract (of $\mathbb{Z}$-graded complexes over $S$)
\be\label{eq:fixing_dm_2}
\xymatrix@C+3pc{
(R'/\bold{t} R',0) \ar@<-1ex>[r]_-{\pi'} & (K_{R'}(\bold{t}), d'_K)\,, \ar@<-1ex>[l]_-{\sigma'}
}\quad h'
\ee
satisfying $(h')^2 = 0, h'\sigma' = 0, \pi' h' = 0$ where $K_{R'}(\bold{t})$ is the Koszul complex over $R'$.
\end{itemize}
Then the canonical morphism of linear factorisations of $W$ over $S$, $\varphi_* X \lto (\alpha \varphi)_* \alpha^* X$ is an $S$-linear homotopy equivalence. Here is the proof (note that we are not assuming $\alpha$ is flat, as is done in \cite[Remark 7.7]{pushforward}). The results of \cite{pushforward} show that we have homotopy equivalences over $S$ (where $X' = X \otimes_R R'$)
\begin{gather*}
\pi: X \otimes_R K_R(\bold{t}) \lto X/\bold{t} X\,,\\
\pi': X' \otimes_{R'} K_{R'}(\bold{t}) \lto X'/\bold{t} X'
\end{gather*}
and the canonical $R$-linear map $X/\bold{t} X \lto X'/\bold{t} X'$ is an isomorphism, by hypothesis. It is clear that the canonical $R$-linear map $X \lto X'$ induces a commutative diagram
\[
\xymatrix@C+2pc{
X \otimes_R K_R(\bold{t}) \ar[d]\ar[r]^-\pi & X/\bold{t} X \ar[d]^-{\cong}\\
X' \otimes_R K_{R'}(\bold{t}) \ar[r]_-{\pi'} & X'/\bold{t} X'
}
\]
hence a homotopy commutative diagram $(\pi = \varepsilon \pi^{-1}, \psi' = \varepsilon \circ (\pi')^{-1}$ in the notation of \cite{pushforward})
\[
\xymatrix@C+2pc{
X/\bold{t} X \ar[d]_-{\cong} \ar[r]^-{\pi^{-1}} & X \otimes_R K_R(\bold{t}) \ar[d] \ar[r]^-\varepsilon & X[n] \ar[d]\\
X'/\bold{t} X' \ar[r]_{(\pi')^{-1}} & X \otimes_R K_{R'}(\bold{t}) \ar[r]_-{\varepsilon} & X'[n]
}
\]
If we produce $\vartheta, \vartheta'$ using the homotopies $\lambda_i, \lambda_i' = \lambda_i \otimes_R R'$ as explained in \cite[\S 4]{pushforward} then it is also clear that the following diagram commutes
\[
\xymatrix@C+5pc{
X[n] \ar[d] \ar[r]^-{\vartheta = (-1)^n \pi \lambda_1 \cdots \lambda_n} & X/\bold{t} X \ar[d]^-{\cong} \\
X'[n] \ar[r]_-{\vartheta' = (-1)^n \pi' \lambda_1' \cdots \lambda_n'} & X'/\bold{t} X'\,.
}
\]
So in summary both squares in the following diagram commute in $\HF(S,W)$
\[
\xymatrix@C+2pc{
X/\bold{t} X \ar[d]_-{\cong} \ar[r]^-{\psi} & X[n] \ar[d] \ar[r]^-\vartheta & X/\bold{t} X \ar[d]\\
X'/\bold{t} X' \ar[r]_-{\psi'} & X'[n] \ar[r]_-{\vartheta'} & X'/\bold{t} X'
}
\]
Now $\psi \vartheta = 1$ and $\psi' \vartheta' = 1$ in the homotopy category, and the above shows that $\vartheta \psi = \vartheta' \psi'$ (identifying $X/\bold{t} X$ with $X'/\bold{t} X'$ as these are isomorphic, not just homotopy equivalent). Since they split the same idempotent, $X, X'$ must be isomorphic in $\HF(S,W)$ and so the middle column in the above diagram is a homotopy equivalence, as claimed.
\end{remark}

In \cite{cut} we assumed the ground ring $k$ was Noetherian. The only reason for this hypothesis was to guarantee that the ring homomorphism $k[\bold{y}] \lto \widehat{k[\bold{y}]}$ is flat, which allows us to use \cite[Remark 7.7]{pushforward} to infer that the canonical map $\varepsilon: Y \otimes_{k[\bold{y}]} X \lto Y \otimes_{k[\bold{y}]} \widehat{k[\bold{y}]} \otimes_{k[\bold{y}]} X$ is a homotopy equivalence, in \cite[\S 4.3]{cut}. Using Remark \ref{remark:fixing_dm} above we explain why everything in \cite{cut} holds for \emph{any} commutative $\mathbb{Q}$-algebra.

We assume we are in the context of \cite[Setup 4.1]{cut} and in particular that $U,V,W$ are potentials in the sense of \cite[Definition 2.4]{lgdual}. Set $t_i = \partial_{y_i} V$ for $1 \le i \le m$. By hypothesis this is a quasi-regular sequence in $R = k[\bold{y}]$, $R/\bold{t} R$ is a finitely generated projective $k$-module, and $K_R(\bold{t})$ is exact outside degree zero. We claim that Remark \ref{remark:fixing_dm} applies to
\begin{gather*}
\varphi: k[\bold{x},\bold{z}] \lto k[\bold{x},\bold{y},\bold{z}], \qquad ( Y \otimes_{k[\bold{y}]} X, d_Y \otimes 1 + 1 \otimes d_X)\\
\alpha: k[\bold{x},\bold{y},\bold{z}] \cong k[\bold{x},\bold{z}] \otimes k[\bold{y}] \lto k[\bold{x},\bold{z}] \otimes \widehat{k[\bold{y}]}
\end{gather*}
We check each of the hypotheses in turn:
\begin{itemize}
\item By Lemma \ref{lemma:trivialsplit} since $R,R/I$ are projective $k$-modules and $K_R(\bold{t})$ is exact outside degree zero we have a deformation retract \eqref{eq:fixing_dm_1} satisfying the necessary conditions.

\item Applying \cite[Appendix B]{pushforward} to the tuple $(k, \widehat{R}, \bold{t})$ gives an isomorphism of $k[\bold{z}]$-modules
\[
\sigmastar: R/I \otimes k[\bold{z}] \lto \widehat{R}\,.
\]
Observe that $R/I \otimes K_{k\llbracket \bold{z} \rrbracket}(\bold{z}) \cong K_{\widehat{R}}(\bold{t})$ via $\sigmastar$ as complexes of $k[\bold{z}]$-modules, and in particular $K_{\widehat{R}}(\bold{t})$ is exact outside degree zero. Moreover using connections we may produce a deformation retract \eqref{eq:fixing_dm_2} of the form required by Remark \ref{remark:fixing_dm}.
\end{itemize}
The conclusion is that in $\HF( k[\bold{x},\bold{z}], U - W)$ the canonical map
\[
Y \otimes_{k[\bold{y}]} X \lto Y \otimes_{k[\bold{y}]} \widehat{k[\bold{y}]} \otimes_{k[\bold{y}]} X\,.
\]
is an isomorphism, as claimed. This removes the Noetherian hypothesis from \cite{cut}.

\begin{proof}[Proof of Lemma \ref{lemma:completion_he}] This is a special case of the above: we apply Remark \ref{remark:fixing_dm} to
\begin{gather*}
\varphi: k \lto k[\bold{x}]\,, \qquad ( \Hom_R(X,Y), d_{\Hom} )\,, \qquad \alpha: k[\bold{x}] \lto \widehat{k[\bold{x}]}
\end{gather*}
Notice that in the above $t_1,\ldots,t_n$ need not be the sequence of partial derivatives, provided it satisfies the conditions set out in Setup \ref{setup:overall}.
\end{proof}

\section{Operator decorated trees}\label{section:trees}

Here we collect some standard material on decorated trees. A \emph{tree} is a connected acyclic unoriented graph. All our trees are finite. A \emph{rooted tree} is a tree in which a particular leaf vertex (meaning a vertex of valency $1$) has been designated the root. We view a rooted tree as an oriented graph by orienting all edges towards the root, that is, in the direction of the unique path to the root. The \emph{children} of a vertex $v$ are those vertices $w$ for which $v$ is the next vertex on the path from $w$ to the root. If $w$ is a child of $v$ we call $v$ the \emph{parent} of $w$. A \emph{plane} tree is a rooted tree together with a chosen linear order $w_1,\ldots,w_n$ for the set of children of each vertex $v$. A \emph{morphism} in the category of plane trees is a morphism of oriented graphs which preserves the root vertex and the ordering on children. A plane tree is \emph{valid} if it has $n + 1$ leaves (including the root) for some $n \ge 2$ and all non-leaves have valency at least three. We call leaves \emph{external vertices}, non-leaves \emph{internal vertices}, the edges meeting a leaf are \emph{external edges} and the others are \emph{internal edges}.

\begin{definition} $\cat{T}_n$ is the set of isomorphism classes of valid plane trees with $n + 1$ leaves. We denote by $\cat{BT}_n$ the set of isomorphism classes of valid plane \emph{binary} trees, that is, trees in which every non-leaf vertex has valency three.
\end{definition}

Let $S$ be a commutative, associative (but not necessarily unital) ring. When we speak of graded $S$-modules we mean either $\nZ$ or $\nZ_2$-graded.

\begin{definition} An $S$-linear \emph{decoration} of a plane tree $T$ is the following data:
\begin{itemize}
\item a graded $S$-module $L_v$ for each leaf $v$ (including the root).
\item a graded $S$-module $M_e$ for each edge $e$.
\item for each internal vertex $v$ with incoming edges (in order) $e_1,\ldots,e_k$ and outgoing edge $e$ an integer $N_v$ (in $\nZ$ or $\nZ_2$) and a degree $N_v$ $S$-linear map
\[
\phi_v: M_{e_1} \otimes_S \cdots \otimes_S M_{e_k} \lto M_e\,.
\]
\item for each non-root leaf vertex $l$ a degree zero $S$-linear map $\phi_l: L_l \lto M_e$ where $e$ is incident at $l$.
\item a degree zero $S$-linear map $\phi_r: M_e \lto L_r$ where $e$ is incident at $r$, the root vertex.
\end{itemize}
\end{definition}

\begin{definition}\label{defn:branch_decomposition} Let $T$ be any plane tree and $D$ an $S$-linear decoration of $T$. Let $r$ be the root vertex and $v$ the vertex adjacent to $r$. Suppose $v$ has valence $k+1$ (possibly $k = 0$). Consider the diagram
\begin{center}
\includegraphics[scale=0.9]{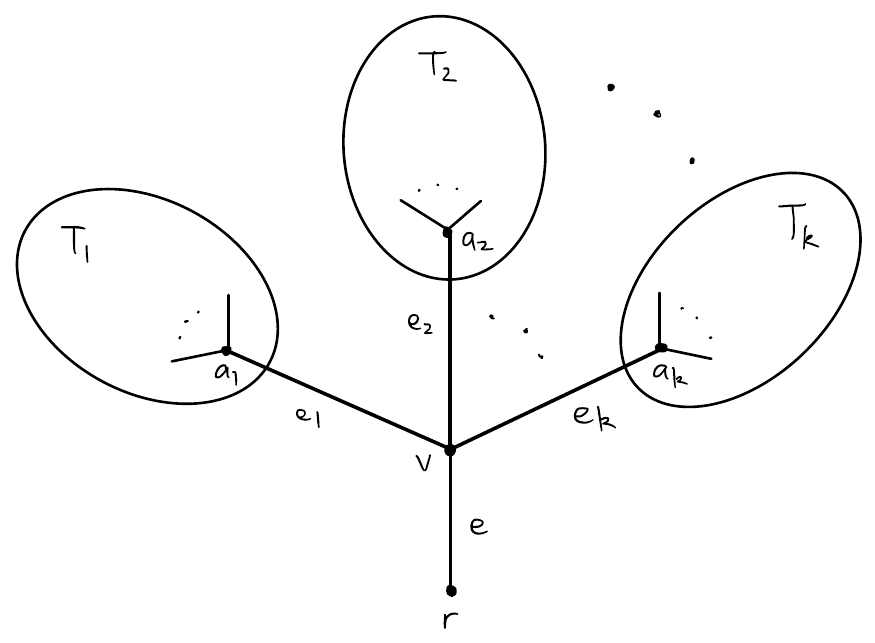}
\end{center}
We define plane trees $\widetilde{T}_1,\ldots,\widetilde{T}_k$ to be (note $e_i$ is \emph{not} in $\widetilde{T}_i$)
\begin{center}
\includegraphics[scale=0.9]{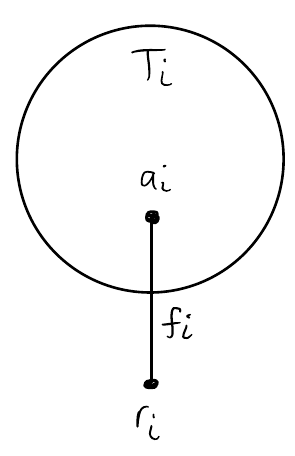}
\end{center}
where $r_i$ is some new vertex, which we declare to the root of $\widetilde{T}_i$. We make $\widetilde{T}_i$ a plane tree using the ordering from $T$, and decorate $\widetilde{T}_i$ according to the decoration $D_i$ which agrees with $D$ and assigns $M_{e_i}$ from $D$ to $f_i$ and also $L_{r_i} = M_{e_i}, \phi_{r_i} = \operatorname{id}$ and $\phi_{a_i}$ as in $D$.
\end{definition}


\begin{definition} Given an $S$-linear decoration $D$ of $T$ the \emph{denotation} $\langle D \rangle$ is a homogeneous $S$-linear map $L_{l_1} \otimes_S \cdots \otimes_S L_{l_n} \lto L_r$ where $l_1,\ldots,l_n$ denote the non-root leaves (in order) defined as follows:
\begin{itemize}
\item if $T$ has one edge, then $\langle D \rangle = \phi_r \circ \phi_v$.
\item otherwise we define recursively
\be\label{eq:end_of_line}
\langle D \rangle = \phi_r \circ \phi_v \circ \Big( \langle D_1 \rangle \otimes_S \cdots \otimes_S \langle D_k \rangle \Big)
\ee
using the branch decomposition of Definition \ref{defn:branch_decomposition}. Note that this is a tensor product of homogeneous operators on graded $S$-modules, and obeys the usual Koszul sign convention when evaluated on a homogeneous tensor.
\end{itemize}
\end{definition}

Note that this denotation differs by signs from another natural evaluation of the diagram, which composes operators by organising them according to their \emph{height} in the tree, rather than their \emph{branch} (as we have done).

\begin{definition}\label{defn:evaluation_tree} Let $T$ be a plane tree, $D$ a decoration and $L_1,\ldots,L_n,L_r$ the modules assigned to the leaves so that $\langle D \rangle: L_1 \otimes_S \otimes \cdots \otimes_S L_n \lto L_r$. We define
\[
\operatorname{eval}_D: L_1 \otimes_S \cdots \otimes_S L_n \lto L_r
\]
to be the $S$-linear map associated to $D$ as a diagram in the category of $S$-modules, \emph{ignoring the grading}. Recursively, if $T$ has one edge $\operatorname{eval}_D = \phi_r \circ \phi_v$ and in the case of \eqref{eq:end_of_line} we use the same formula but read the tensor product as being of plain $S$-linear maps (so there are no Koszul signs when we evaluate the results on an input tensor).
\end{definition}

\bibliographystyle{amsalpha}
\providecommand{\bysame}{\leavevmode\hbox to3em{\hrulefill}\thinspace}
\providecommand{\href}[2]{#2}

\end{document}